\documentclass[11pt,twoside]{amsart}
\usepackage{latexsym,amssymb,amsmath}
\usepackage[all]{xy}
\usepackage{tikz,pgfplots}
\usepackage{extpfeil} 
\usepackage{afterpage}
\usepackage{float}
\usepackage{hyperref}

\makeatletter  
\@namedef{subjclassname@2020}{%
\textup{2020} Mathematics Subject Classification}
\makeatother

\numberwithin{equation}{section}
\newtheorem{theorem}{Theorem}[section]
\newtheorem{lemma}[theorem]{Lemma}
\newtheorem{proposition}[theorem]{Proposition}
\newtheorem{corollary}[theorem]{Corollary}
\newtheorem{conjecture}[theorem]{Conjecture}

\theoremstyle{definition}
\newtheorem{definition}[theorem]{Definition}

\newtheorem{remark}[theorem]{Remark}
\newtheorem{example}[theorem]{Example}

\begin{document}

\title[Graph rings and ideals]
{Graph rings and ideals: Wolmer Vasconcelos' contributions} 

\thanks{The first author
was partially supported by the Center for Mathematical Analysis, Geometry and
Dynamical Systems of Instituto Superior T\'ecnico, Universidade de
Lisboa. The second author was supported by SNII, M\'exico.}

\author[M. Vaz Pinto]{Maria Vaz Pinto}
\address{Departamento de Matem\'atica, Instituto Superior T\'ecnico,
Universidade de Lisboa, Avenida Rovisco Pais, 1, 1049-001 Lisboa,
Portugal.
} 
\email{vazpinto@math.tecnico.ulisboa.pt}

\author[R. H. Villarreal]{Rafael H. Villarreal}
\address{
Departamento de
Matem\'aticas\\
Cinvestav, Av. IPN 2508, 07360, CDMX, M\'exico.
}
\email{rvillarreal@cinvestav.mx}


\dedicatory{Dedicated to the memory of Wolmer V. Vasconcelos}  

\begin{abstract}
This is a survey article featuring some of Wolmer
Vasconcelos' contributions to commutative
algebra, and explaining how Vasconcelos' work and
insights have contributed to the development of commutative
algebra and its interaction with other areas to the present. 
We discuss the Vasconcelos' function and the Vasconcelos'
number ({\rm v}-{number} for short) of graded ideals and their
relation to coding theory, and the interplay of Simis and normal monomial
ideals with combinatorial optimization problems, blowup algebras, and
resurgence theory.  
The regularity of subrings 
of normal $k$-uniform monomial ideals is shown to be a
monotone function, and we give a 
normality criterion for edge ideals of graphs using Ehrhart rings.
\end{abstract}

\maketitle 

\section{Introduction}\label{intro-section}  

In this work, we present some of Wolmer
Vasconcelos' contributions to computational methods in commutative algebra and
algebraic geometry, Koszul homology and the conormal module, 
graph rings and ideals, and explain how Vasconcelos' work and
insights have contributed to the development of commutative
algebra and its interaction with other areas to the present.

In Section~\ref{section-computational-methods}, 
we give a brief introduction to some of Vasconcelos' results 
that contributed to the development of computational methods in 
commutative algebra and algebraic geometry, see his book on the
subject \cite{Vas1}. 

One of Vasconcelos' research interests was the study of Koszul
homology and the conormal module, and especially the 1st Koszul
homology module of an ideal. His famous conjecture in the 1970' on the conormal
module \cite{Wolmer-conormal} was recently proved by Briggs
\cite[Theorem~A]{Briggs}, it was an open problem for more than forty
years. We discuss some of
his contributions and conjectures in this area in
Section~\ref{section-H1-conormal}, including a theorem of Huneke,
Ulrich and Vasconcelos that relates the 
normality of an ideal with residual intersections and other algebraic
properties \cite{HUV}.  

The origins of edge ideals are discussed in
Section~\ref{section-the-origins}. Vasconcelos introduced a criterion
using Gr\"obner bases to decide whether a graded ideal is syzygetic or
of linear type (Proposition~\ref{mar6-23}), and showed a criterion 
to determine whether the first Koszul homology module of a
graded ideal is Cohen--Macaulay (Proposition~\ref{mar6-23-1}). Using this criteria, for polynomial
rings in a small number of variables, one can determine the full list 
of ideals generated by squarefree monomials of degree two forming a
simple graph that
are strongly Cohen--Macaulay and of linear type
\cite[p.~25]{vila-tesis}. In the 1980', this motivated the
introduction and study of edge ideals of graphs using 
methods coming from commutative algebra and graph
theory \cite{cmg}. 

Stanley--Reisner rings or face rings of simplicial complexes were
defined 
independently by Hochster and Stanley in the 1970' to study 
combinatorial problems, Betti numbers, and Cohen--Macaulay 
face rings, using topology and commutative algebra. Their initial
contributions \cite{Ho3,Sta3,Sta1} and the work of Reisner \cite{Rei}
were crucial for the growth and interest in the area. 

To complement the interplay between the algebra 
of monomial ideals and the combinatorics of 
simplicial complexes, Simis and Vasconcelos point
of view in the 1980' was
that the algebraic properties and invariants of a \textit{graph ring} (e.g., 
a quotient ring of the edge ideal of a graph, a blowup algebra of
an edge ideal, or a monomial subring of a graph) 
could be naturally related to the combinatorial properties and
invariants of the graph defining the graph ring. 

Graph rings were developed and
explored by Simis, Vasconcelos and the second author in \cite{ITG}.
The results of this paper are discussed throughout
Sections~\ref{edge-ideals-all-flavors}-\ref{blowuprings-section}. 
Simis and Vasconcelos, and their
students and co-workers, in the
1990' studied algebraic properties and invariants 
of commutative graph rings using commutative algebra, graph theory and 
polyhedral geometry \cite{BVV,CVV,luisa-tor,ELV,EV,Koszul-linear,aron-lnm,aron-hoyos,simis-ulrich,
ITG,bowtie}, \cite{cmg}--\cite{canela}. Other central 
papers in this period are
\cite{jesus-etal,hibi-ohsugi-normal,aron-jac}. The landmark books of Stanley \cite{Sta2} 
and Sturmfels \cite{Stur1} in the 1990'   
underline the powerful techniques of commutative algebra and Gr\"obner bases 
in the interplay with combinatorics, convex polytopes, and topology.

Edge ideals of graphs and clutters and their Alexander duals are
introduced in Section~\ref{edge-ideals-all-flavors}. A theorem of Simis, Vasconcelos
and the second author shows that all Koszul 
homology modules of the edge ideal of a Cohen--Macaulay tree are
Cohen--Macaulay \cite{ITG}. We recall some of the theorems in
this area 
to the present, including the classification of Cohen--Macaulay trees,
criteria for Cohen-Macaulay, unmixed, and
sequentially Cohen-Macaulay edge ideals, the Lyubeznik and Fr\"oberg theorems relating chordal
graphs with algebraic properties of edge ideals, the duality theorem
of Eagon and Reiner on edge ideals, a classification of the simplicial trees, introduced by Faridi, generalizing the
notion of trees for graphs, the duality formula of Terai for edge
ideals, the formula for the projective dimension of a sequentially
Cohen--Macaulay edge ideal, the formula of Mahmoudi, Mousivand, Crupi, Rinaldo, Terai
and Yassemi for the regularity of edge
ideals of very well-covered graphs, the description of Francisco,
H$\rm \grave{a}$ and Van Tuyl of the associated primes of the second power of ideals of covers of 
graphs and their formula for computing the chromatic number of a
simple graph. 

In Section~\ref{normality-symbolic-powers}, we present a formula of
Vasconcelos for computing the symbolic powers of a prime ideal of 
a polynomial ring $S$ (Proposition~\ref{effesymbvas}). Symbolic powers of ideals of $S$ can be computed using the
algorithms implemented in \textit{Macaulay}$2$ \cite{mac2} by Drabkin, Grifo, 
Seceleanu and Stone \cite{Grifo-etal}. One of these algorithms uses
the methods of  Eisenbud, Huneke, and Vasconcelos for finding 
primary decompositions of ideals of $S$ \cite{EHV}.  
A main result of Simis, Vasconcelos
and the second author \cite{ITG} shows that the edge ideal of a graph is normally torsion
free if and only if the graph is bipartite. A complementary result of
\cite{alexdual} shows that the ideal of covers of a graph is normally
torsion free if and only if the graph is bipartite. In this setting, 
normally torsion free is equivalent to equality of ordinary and
symbolic powers.    
We recall how to find the
normalization of the edge ideal of a graph using the formula of
Vasconcelos in terms of Hochster configurations \cite[p.~459]{monalg-rev}, and we also recall
how to find the normalization of an edge subring of a graph using 
a formula, in terms of
bowties, that was shown independently by Hibi and Ohsugi \cite{hibi-ohsugi-normal},
and by Simis, Vasconcelos and the second author \cite{bowtie}. Then, we give
normality criteria for edge ideals and for edge subrings of
graphs. 

Blowup algebras associated to ideals, see
Eqs.~\eqref{jan21-25-1}-\eqref{jan21-25-4}, appear in many
constructions in  
commutative algebra and algebraic geometry.  
These algebras were 
 introduced and systematically studied in Vasconcelos' book
 \cite{Vas}, 
where many significant results and methods are included. 
In the 2000' blowup algebras, rings, and ideals associated to
monomials were studied 
using linear programming and combinatorial
optimization methods \cite{unimod,ainv,BG-book,mc8,roundp,cm-mfmc,Lisboar,alexdual,reesclu,wolmer65,
clutters,hemmecke,cover-algebras,unimodular-hyp,hhtz,birational,birational-linear,
Trung,shiftcon,accota,perfect,matrof}.
This theory was further developed in the 2010' 
\cite{Waldschmidt-Bocci-etal,symbolic-powers-survey,
Francisco-TAMS,roundp,poset,mfmc,ringraphs,HaM,Ha-Trung-19,depth-monomial,cremona,monalg-rev}. For some
recent results, see 
\cite{Alilooee-Benerjee,Seceleanu-packing,Seceleanu-convex-bodies,
icdual,intclos,Ha-Nguyen,Haase-Santos,hoa,Aron-bookII,ic-resurgence} and references
therein. In
Sections~\ref{normality-section}-\ref{blowuprings-section}, we will examine these methods. 
Some essential features of this theory are 
the normality of ideals and rings, the equality of ordinary and symbolic
powers of ideals, and the use of polyhedral geometry to study these
properties for ideals and algebras defined by monomials \cite{graphs,handbook,webster}. 

The normality of monomial ideals is discussed in 
Section~\ref{normality-section}. We  
present the linear programming (LP) membership
test of Delfino, Taylor, Vasconcelos, Weininger and
the second author that determines whether or not a given monomial lies in the integral
closure of a monomial ideal \cite{mc8}. 
We examine normality criteria
for monomial ideals using linear algebra \cite{icdual}. 
Other normality criteria can be found in the paper of Brennan and
Vasconcelos \cite{BreVas1}, in the paper of Brummati, Simis
and Vasconcelos \cite{BSV}, in the paper of Escobar, Yoshino
and the second author \cite{Lisboar}, in the book of Huneke
and Swanson \cite{huneke-swanson-book}, and in the book of Vasconcelos
\cite{bookthree}. A monomial ideal is 
\textit{uniform} if it is minimally generated by monomials of the same degree. 
For a normal uniform monomial ideal, its
monomial subring is normal (Theorem~\ref{feb16-23}), this is called
the \textit{descent of normality 
criterion}, and is due to Simis, Vasconcelos and the
second author \cite{ITG}. 

Ehrhart rings are introduced in Section~\ref{regularity-section}. 
These rings are normal. We present a \textit{generalized descent of
normality criterion} showing that subrings of normal
uniform monomial ideals are Ehrhart rings (Theorem~\ref{mar3-01}). 
Hochster's theorem shows that normal
monomial subrings are Cohen--Macaulay
\cite{Ho1}. Bruns, Vasconcelos and the
second author gave a generating set for the canonical module of 
the $k$-th {squarefree Veronese subring} (Theorem~\ref{main-t}), computed  the
$a$-invariant and the regularity of this subring (Proposition~\ref{a-invariant}), and found 
a sharp upper bound for the $a$-invariant of the normalization 
of a subring generated by squarefree 
monomials of the same degree which is attained at a squarefree
Veronese subring (Theorem~\ref{main-in}).  
 For graded Cohen--Macaulay
monomial subrings, computing 
the $a$-invariant is equivalent to computing the
regularity (Corollary~\ref{jan27-25}). We discuss these two
invariants for edge subrings of bipartite graphs. Then, we study the regularity and the $a$-invariant of
normal subrings generated by monomials of the same degree 
using Hochster's theorem, Ehrhart theory, Stanley's 
monotonicity theorem for polytopes, and the descent of normality
criterion. Fixing an integer $k\geq 1$, these results allow us to show that the function 
$I\mapsto {\rm reg}(K[I])$, $K$ a field, $I$ a $k$-uniform normal
monomial ideal, is a monotone function (Theorem~\ref{vila-obs}). 
For edge ideals of graphs, we give
a normality criterion in terms of Ehrhart rings showing that the
converse of the generalized descent of normality criterion 
holds in this case (Theorem~\ref{normality-criterion-ehrhart}).  

In Section~\ref{reduction-section}, we discuss some of Vasconcelos'
work on normalizations of monomial ideals and reduction numbers.
Bruns, Vasconcelos and the second author gave degree bounds for the
generators of normalizations of Rees algebras and uniform monomial subrings
(Theorems~\ref{may4-23} and \ref{feb13-25}). Vasconcelos proved that the filtration 
of integral closures of powers of a monomial ideal stabilizes at the dimension of the base ring
 (Theorem~\ref{may3-23}). 
Polini, Ulrich, Vasconcelos and the second author \cite[Theorem~2.4]{PUVV}
gave upper bounds for the \textit{normalization index} (Definition~\ref{norm-index-def}) of homogeneous
$\mathfrak{m}$-primary ideals over a field of characteristic zero. 
Reduction numbers of ideals, introduced by Northcott and Rees
\cite{Northcott-Rees}, 
were used by Aberbach, Huneke, Polini, Trung,
Ulrich, Vasconcelos and Vaz
Pinto to study blowup algebras, see \cite{Huneke-etal,Polini-Ulrich,Vas,Vas5,Maria} and references
therein. Techniques from the theory of Rees algebra of modules were 
introduced by Vasconcelos to produce estimates for the \textit{reduction
number} $r(I)$ of an ideal $I$
(Definition~\ref{reduction-number-def}) for classes of ideals 
of dimension one and two \cite{Vas6}. 
Ghezzi, Goto, Hong and Vasconcelos \cite{ghezzi-etal} 
examined the relationship between the reduction number and the multiplicity of
the Sally module, and obtained upper bounds for the reduction 
number. The
\textit{core} of an ideal is the intersection of all of its 
reductions. Recently, Fouli, Monta\~no, Polini and Ulrich \cite{FMPU} 
give an explicit description for the core of a monomial ideal
$I$, satisfying certain residual 
conditions, showing that the core of $I$ is the largest
monomial ideal contained 
in a general reduction of the ideal $I$. A description of the core of a module is given by
Costantini, Fouli and Hong \cite{CFH}. 

Some of Vasconcelos' work on multiplicities, Hilbert functions and $\mathfrak{m}$-fullness
of $\mathfrak{m}$-primary monomial ideals is
introduced in Section~\ref{multiplicity-section}. Delfino, Taylor, Vasconcelos, Weininger and
the second author \cite{mc8}, \cite[Section~7.3]{bookthree}, gave a
Monte-Carlo-based approach for the computation of volumes of lattice
polytopes, and multiplicities of monomial ideals
\cite[p.~131]{Teissier}. Gimenez, Simis,
Vasconcelos and the second author classified the 
$\mathfrak{m}$-full $\mathfrak{m}$-primary monomial ideals in 
dimension two (Theorem~\ref{p-w-v}), and showed that an $\mathfrak{m}$-primary
$\mathfrak{m}$-full ideal with special fiber
ring Cohen-Macaulay has Cohen-Macaulay Rees algebra (Theorem~\ref{CMviafiber}).

In Section~\ref{blowuprings-section}, we discuss the 
relation of a result of Huneke, Simis and Vasconcelos---classifying
the reducedness of the associated graded ring of an ideal---with the theory of
symbolic powers and the packing problem for edge ideals of clutters, and present some 
recent results on the resurgence theory of edge ideals that are
related to the containment problem for ordinary and symbolic 
powers of ideals. The equality $I^n=I^{(n)}$ for all $n\geq 1$ of
ordinary and symbolic powers of the edge ideal $I=I(\mathcal{C})$ of a clutter
$\mathcal{C}$ is equivalent to the max-flow min-cut property of the
clutter $\mathcal{C}$ (Theorem~\ref{ntf-char}). A monomial ideal $I$ is called a \textit{Simis ideal} 
if $I^{(n)}=I^n$ for all $n\geq 1$. The term \textit{Simis ideal} is
introduced in \cite{weightedma} to recognize the
pioneering work of Aron Simis on symbolic powers of monomial ideals 
\cite{bahiano,HuSV,aron-hoyos,Aron-bookII,simis-ulrich,ITG}. 

The \textit{Vasconcelos' function} (v-\textit{function} for short) of a graded
ideal was introduced in \cite{rth-footprint}, 
see Eq.~\eqref{vas-function}, as an extension of the generalized Hamming
weights of projective Reed--Muller-type codes
(Theorem~\ref{rth-min-dis-vi}), and the 
\textit{Vasconcelos' number} ({\rm v}-\textit{number} for short) of a
graded ideal was introduced by Cooper, Seceleanu, Toh\v{a}neanu, 
Vaz Pinto and the second author \cite{min-dis-generalized}, see
Eq.~\eqref{v-number-def}. These two notions are discussed in
Section~\ref{coding-theory-section}. We explain
why the v-number plays a role in the theory of error-correcting codes
and linear codes \cite{MacWilliams-Sloane} 
(Theorem~\ref{min-dis-v-number}) and in the theory of graphs
\cite{Har} (Theorem~\ref{v-numbers-w2}). 

\section{Computational methods}\label{section-computational-methods}

Let $S=K[t_1,\ldots,t_s]$ be a polynomial ring over a field $K$. 
The discovery of Buchberger's algorithm \cite{Buch} to compute
Gr\"obner bases of ideals of $S$ and the
implementation of this algorithm in computer algebra systems
\cite{magma,CNR,mac2,singular,cocoa-book} is nowadays a widely used tool in
commutative algebra and algebraic geometry with multiple applications
\cite{coding-theory-package,BG-book,CLO,Ene-Herzog,weights-matroid,Stur1}. Vasconcelos 
made substantial contributions to advance the
computational 
methods in these areas \cite{BreVas,Vas4}, \cite[Chapter~10]{Vas},
\cite{Vas1,bookthree}, and together with Eisenbud and 
Huneke gave methods implemented in \textit{Macaulay}$2$ \cite{mac2}
for finding the primary decomposition of an 
ideal of $S$ \cite{EHV}. 

In the 1980'  Brummati, Simis and Vasconcelos \cite{BSV} 
studied the question---in the boundary between commutative algebra
and computer algebra---of deciding 
the completeness of all the powers of an ideal of a polynomial ring.
For monomial ideals \textit{Normaliz} \cite{normaliz2} solves this problem.

The computer program \textit{Normaliz} was developed
by Bruns and Koch in the late 1990' \cite{BK} and by Bruns, 
Ichim, R\"omer, Sieg and S\"oger since the mid 2000' \cite{normaliz2}.
The program \textit{Normaliz} has been continuously improving to
date. 
It provides an invaluable effective 
tool to compute normalizations of monomial subrings 
and their algebraic invariants, see \cite[Section~1.7]{monalg-rev} for
a list of integer programming properties that can be solved using
\textit{Normaliz}. One of the first third-party publications citing
\textit{Normaliz} is the paper of Vasconcelos et al. \cite{mc8}, see also Vasconcelos 
joint paper with Bruns and the second author \cite{BVV} where
\textit{Normaliz} was used in 
\cite[Example~3.9]{BVV} to compute a normalization.

Volumes of polytopes can be computed with \textit{Normaliz}
\cite{normaliz2} and they can be used to compute multiplicities of
 monomial ideals (Proposition~\ref{teissier-e}). For a 
 recent survey on the algorithms of \textit{Normaliz} for volumes, see Bruns
 \cite{Bruns-Volume}.   
Delfino, Taylor, Vasconcelos, Weininger and
the second author \cite{mc8} gave a
probabilistic Monte-Carlo-based approach to compute multiplicities of
zero-dimensional monomial ideals and the volume of the underlying
lattice polytopes, see Section~\ref{multiplicity-section}. Years later, methods using Ehrhart functions
of polytopes and polynomial interpolation were given to compute the
Hilbert function of 
the normalization of these
ideals (Proposition~\ref{may29-1-05}).

\section{The first Koszul homology module and the conormal
module}\label{section-H1-conormal}

Let $S$ be a commutative Noetherian ring, let $I\neq(0)$ be an ideal of
$S$, and let $H_i(I)$ be the $i$-th  
homology module of the Koszul complex $H_\star(\underline{x}, S)$ associated to a set
$\underline{x}=\{x_1,\ldots x_n\}$ of
generators of $I$. The first Koszul homology module $H_1(I)$ of $I$ is related to $I/I^2$, 
the {\it conormal module\/} of $I$, 
by the following exact sequence \cite{SV1}: 
\begin{equation}\label{feb5-23}
H_1(I)\stackrel{f}{\longrightarrow}S^n\otimes (S/I)
\stackrel{h}{\longrightarrow}I \otimes (S/I)= I/I^2 
\longrightarrow 0,
\end{equation}
where $f([z]) = z \otimes 1$ and $h(e_i \otimes 1) = x_i \otimes 1$. 
Ring theoretic properties of $S/I$ are
often reflected in module theoretic properties of $I/I^2$, for
instance Ferrand and Vasconcelos proved that if $S$ is a local Noetherian
ring and $I$ has finite projective dimension, 
then the conormal module is projective over $S/I$ if and only if $I$ is 
locally generated by a regular sequence \cite{Ferrand,Vas-old}. 

Vasconcelos made the conjecture: ``Let $I$ be an ideal of finite
projective dimension 
in a Noetherian local ring $S$, if the conormal module
$I/I^2$  has finite projective dimension over $S/I$, 
then $I$ must be locally generated by a regular sequence''
\cite{Wolmer-conormal}. The problem has inspired several interesting
research directions in the last four decades, see 
\cite{Briggs,Herzog-conjectures} and references therein. The conjecture 
was recently proved by Briggs
\cite[Theorem~A]{Briggs}, where he also proves a similar result for the first
Koszul homology module of $I$ \cite[Theorem~B]{Briggs}. 

The ideal $I$ is called {\it syzygetic\/} if $\ker(f)=0$ \cite{SV1}.
For syzygetic Cohen--Macaulay ideals of codimension $3$, another
conjecture of Vasconcelos states that an ideal in this class 
is Gorenstein if its conormal module is Cohen--Macaulay 
\cite[Conjecture~(B)]{Vas2}.

Let $I$ and $J$ be two ideals in a Cohen--Macaulay local ring $S$. The
ideals $I$ and $J$ are said to be (algebraically) \textit{linked} 
if there is an $S$-regular sequence $\underline{x}=\{x_1,\ldots,x_n\}$ in $I\cap
J$ such that $I=((\underline{x})\colon J)$ and
$J=((\underline{x})\colon I)$. 
When $I$ and $J$ are linked one writes $I\sim J$. For the notion of
\textit{geometrically linked ideals}, see \cite[p.~327]{Vas1}. Another notion of
linkage, \textit{residual intersection}, replaces the regular sequence
$\underline{x}$ by more general ideals \cite{Artin-Nagata,Huneke-TAMS}. We 
say that $J$ is in the {\it linkage class\/}
 of $I$ if there are ideals $I_1,\ldots, I_{m}$ such
that 
$$I\sim I_1 \sim\cdots\sim I_{m}\sim J.$$ 
\quad The ideal $J$ is 
said to be in the {\it even linkage class} of $I$ if $m$ is odd. 
Let $S$ be a Gorenstein ring and let $I$ be a 
Cohen--Macaulay ideal of $S$. If $J$ is linked to $I$, then Peskine and Szpiro \cite{PS1} 
showed  that $J$ is Cohen--Macaulay. 

The ideal $I$ is called \textit{strongly Cohen-Macaulay} (SCM for
short) if $H_i(I)$ is
Cohen-Macaulay for all $i\geq 0$ \cite{Huneke-TAMS}. Ideals in a Gorenstein local ring that are
in the linkage 
class of a
complete intersection are SCM \cite{Hu2} and so are 
the perfect ideals of codimension two and perfect Gorenstein ideals of
codimension three \cite[pp.~75--76]{Vas}. Vasconcelos was interested in the 
Koszul homology of $I$ and especially in the Koszul homology
module $H_1(I)$ on some system of generators of $I$ and in the conormal
module $I/I^2$, both considered over $S/I$. If $I$ is a perfect ideal
of codimension three, then the Cohen-Macaulayness of $H_1(I)$ is an invariant of the
whole linkage class of $I$ \cite[Theorem~2.4]{Vas2}. A fundamental result
of Huneke asserts that the SCM property of $I$ is
invariant under even linkage \cite{Hu2}. 
If $S$ is a Gorenstein local
ring and $I$ is a Gorenstein ideal of
codimension four, then a theorem of Vasconcelos shows that $H_1(I)$
is Cohen--Macaulay if and only if
$I/I^2$ is Cohen--Macaulay \cite[Theorem~3.1]{Vas2}. The if part of
this theorem was shown by Vasconcelos and the second author
\cite[Theorem~1.1]{VV}. 

Several numerical computations and
\cite[Theorem~3.1.5]{vila-tesis-ja}, \cite[Theorem~2.1]{Koszul-linear}, support the
following conjecture of
Vasconcelos which is still open
(cf.~\cite[Conjecture~3.1.4]{vila-tesis}). Recall that 
the \textit{deviation} 
of an ideal $I$ is the deficit $n-g$, where
$n=\mu(I)$ is the minimum number of generators
of $I$ and $g={\rm ht}(I)$ is the height of $I$.

\begin{conjecture}\cite[Conjecture (A)]{Vas2}
Let $I$ be a homogeneous ideal of a polynomial ring $S$. 
 Assume that $I$ has height $3$, deviation at least $3$, and is
generically a complete intersection. If $I$ is not a Gorenstein ideal
and the resolution of $S/I$ is
pure, then $I$ is not SCM.
\end{conjecture}

Let $I$ be an ideal in a Gorenstein local ring $S$ and 
let $\underline{x} = \{x_1,\ldots, x_n\}$ be a generating 
set for $I$. We say that $I$ satisfies \textit{sliding depth} if 
\begin{equation*}
{\rm depth}\, H_{i}(\underline{x})
\geq  \dim(S) - n + i,\ \ i\geq 0.
\end{equation*}
\quad This property localizes and depends solely on the number of elements in the sequence
$\underline{x}$ \cite[p.~676]{HSV2}, and is an invariant of even
linkage (cf. \cite{HVV,Hu2,vila-tesis}).

The strongly Cohen--Macaulay and sliding depth conditions are often necessary for
the study of the properties of blowup rings, see 
\cite[Corollaries~3.3.21, 3.3.24]{Vas}. The
approximation complex $\mathcal{M}(I)$ of an ideal $I$ is a complex of graded modules introduced
by Simis and Vasconcelos \cite{SV1} to study blowup rings, its
acyclicity bears a striking resemblance to that of an ordinary Koszul
complex, we refer to
Vasconcelos book \cite[Chapter~3]{Vas} for the theory of this
complex. The notion of a $d$-\textit{sequence},  
introduced by Huneke \cite{Hu5}, plays a key role here. A result of Herzog, Simis and
Vasconcelos shows that $\mathcal{M}(I)$ is acyclic if and only if $I$ is
generated by a $d$-sequence \cite{HSV1}. Let $S$ be a local 
Cohen--Macaulay domain and let $I$ be an ideal of $S$. If $I$ is Cohen--Macaulay
and has sliding depth, then the Rees algebra and the associated graded
ring of $I$ are Cohen--Macaulay and $I$ is generated by a $d$-sequence
\cite{HSV1,Hu0}, \cite[Remark~2.1.5]{vila-tesis}. 

We close this section with a theorem that relates the 
normality of an ideal with residual intersections and other algebraic
properties (cf. \cite[Corollary~1.7]{HUV}). 

\begin{theorem}{\rm(Huneke, Ulrich, Vasconcelos
\cite[Theorem~1.6]{HUV})} Let $(S,\mathfrak{m})$ be a regular local
ring with infinite residue class field, and let $I\neq\mathfrak{m}$ be a
reduced strongly Cohen--Macaulay $S$-ideal such that
$\mu(I)= \dim(S)$ and $\mu(I_\mathfrak{p})\leq\max\{{\rm
ht}(I),\dim(S_\mathfrak{p})-1\}$ for all $\mathfrak{p}\in
V(I)\setminus\{\mathfrak{m}\}$. Then, 
the following are equivalent:
\begin{enumerate}
\item[(a)] $I$ is a normal ideal. {\rm(b)} $I^k$ is integrally closed
for some $k>\dim(S/I)$. 
\item[(c)] $I$ has a residual intersection $J$ with $S/J$ a normal Gorenstein domain.
\item[(d)] $I$ has a geometric residual intersection $J$ with $S/J$ a
discrete valuation ring.
\end{enumerate}
\end{theorem}

\section{Origins of graph ideals: Syzygetic and linear type ideals}\label{section-the-origins}

Let $S=K[t_1,\ldots,t_s]$ be a polynomial ring over a field $K$ and 
let $I$ be a non-zero ideal of $S$. 
We set $\delta(I)=\ker(f)$, where $f$ is the map that appears in 
Eq.~\eqref{feb5-23}. Recall that the ideal $I$ is called {\it syzygetic\/}
if $\delta(I)=0$. Simis and Vasconcelos \cite{SV1} proved that
\[
\delta(I)\simeq\ker({\rm Sym}_2(I)\rightarrow I^2),
\]
where ${\rm Sym}_2(I)$ is the symmetric algebra of $I$ in degree $2$ and 
${\rm Sym}_2(I)\rightarrow I^2$ is 
the surjection induced by the multiplication map. In particular, 
$\delta(I)$ depends only on $I$. The ideal $I$ is said to be 
{\it generically a complete intersection\/} if $IS_{\mathfrak p}$ 
is a  
complete intersection for all ${\mathfrak p}\in {\rm Ass}_S(S/I)$.
If $I$ is unmixed, that is, all associated primes of $I$ have the
same height, and is generically a complete intersection, then
the $S/I$-torsion of $H_1(I)$ is equal to $\delta(I)$ 
\cite[Remark~4.2.4]{monalg-rev}.

Let $\underline{x}=\{x_1,\ldots, x_n\}$ be a set of generators of $I$. 
The {\it Rees algebra\/} of $I$, denoted 
by  $S[Iz]$, is the subring of $S[z]$ given 
by
\[
S[Iz]=S[x_1z,\ldots,x_nz]\subset S[z],
\]
where $z$ is a new variable. There is an 
epimorphism of    
$S$-algebras
\begin{equation}\label{feb6-23}
\varphi\colon\, B=S[y_1,\ldots,y_n]\longrightarrow
S[Iz]\longrightarrow 0,\ \ \ y_i\stackrel{\varphi}{\longmapsto} x_iz, 
\end{equation}
where $B=S[y_1,\ldots,y_n]$ is a polynomial ring over the ring $S$
with the standard grading induced by setting $\deg(y_i)=1$
for $i=1,\ldots,n$. The kernel 
of $\varphi$, denoted by $J$, is the {\em presentation
ideal\/} of $S[Iz]$ with respect to
$\underline{x}$. The mapping 
$\psi\colon S^n \longrightarrow I$ given by 
$\psi(a_1,\ldots, a_n) = \sum_{i=1}^n a_ix_i$ induces an 
$S$-algebra epimorphism 
$$\beta\colon S[y_1,\ldots,y_n] \longrightarrow 
{\rm Sym}(I),$$ 
where ${\rm Sym}(I)$ is the symmetric algebra of $I$ as an
$S$-module. Thus, 
\[
{\rm Sym}(I)\simeq S[y_1,\ldots,y_n]/{\rm ker}(\beta),
\]
where ${\rm ker}(\beta)$ is an ideal of $S[y_1,\ldots,y_n]$ 
generated by linear forms: 
$$
{\rm ker}(\beta)=
\left.\left(\left\{\sum_{i=1}^nb_iy_i\right\vert\, 
\sum_{i=1}^nb_ix_i = 0\mbox{ and }b_i\in S
\right\}\right).
$$ 
\quad The
kernel of $\varphi$ is generated by all forms $F(y_1,\ldots, y_n)$
such that $F(x_1,\ldots,x_n)=0$ and one may factor
$\varphi$ through ${\rm Sym}(I)$ and obtain the 
commutative diagram:
$$
\begin{matrix} S[y_1,\ldots, y_n]&\buildrel \varphi \over \longrightarrow &
S[Iz]\cr
&&&\cr
\downarrow^{\beta} &^{\alpha} \nearrow \cr
&&&\cr
{\rm Sym}(I)\cr
\end{matrix}
$$
\quad For a beautiful treatment of symmetric algebras of modules and 
explicit methods to compute ideal transforms, see 
Vasconcelos' paper \cite{Wolmer-symmetric}.

We say that $I$ is an {\it ideal of linear
type\/} if $\alpha$ is an isomorphism. An important module-theoretic obstruction to 
``${\rm Sym}(I) \simeq S[Iz]$'' is given by the following 
result. 

\begin{proposition}{\rm(Herzog, Simis, Vasconcelos \cite{HSV1})}
\label{obst-to-lin-type} 
If ${\rm Sym}(I) \simeq S[Iz]$, then for 
each prime ideal ${\mathfrak p}$ containing $I$, $I_{\mathfrak p}$ 
can be generated 
by ${\rm ht}({\mathfrak p})$ elements.
\end{proposition}

\begin{theorem}{\rm(Huneke \cite[Theorem~3.1]{Hu-linear-type})} If the
ideal $I=(x_1,\ldots,x_n)$ is generated by a $d$-sequence, then $I$ is of linear type.
\end{theorem}

The presentation ideal $J$ of the Rees algebra $S[Iz]$ can be obtained as follows \cite{BSV}:
$$
J = (y_1- zx_1,\ldots, y_n- zx_n)\cap B.
$$
\quad As $J$ is a graded ideal in the $y_i$-variables, 
$J=\bigoplus_{i\geq 1}J_i$. The relationship between $J$ and the
first Koszul homology module $H_1(I)$ of $I$ is tight. The exact sequence 
of Eq.~\eqref{feb5-23} can be made precise:
\[
0 \longrightarrow J_2/B_1J_1 = \delta(I) \longrightarrow H_1(I)
\longrightarrow (S/I)^{n} \longrightarrow I/I^2 \longrightarrow 0,
\]
and we obtain the following result:
\begin{proposition}{\rm(Vasconcelos \cite[Chapter~7]{Vas1})}\label{mar6-23}
Using Gr\"obner basis, one can decide whether $I$ is 
syzygetic---that is, $J_2=B_1J_1$---or of linear type, that is,
$J=J_1B$.
\end{proposition}

Vasconcelos gave a procedure to compute
Noether normalizations and 
systems of parameters \cite[p.~609]{Vas2} and showed the following
criterion to determine whether the first Koszul homology module of a
graded ideal is 
Cohen--Macaulay.

\begin{proposition}{\rm(Vasconcelos
\cite[p.~610]{Vas2})}\label{mar6-23-1} 
Let $I$ be a homogeneous, Cohen--Macaulay, syzygetic ideal of height $g$. Let
$K[y_1,\ldots,y_{n-g}]$ be a Noether normalization of
$K[y_1,\ldots,y_n]/I$. Then $H_1(I)$ is Cohen--Macaulay if and only
if the following condition holds:
$$
I^2\cap((y_1,\ldots,y_{i-1})I^2\colon
y_i)=(y_1,\ldots,y_{i-1})I^2\quad\mbox{for }i=1,\ldots,n-g.
$$
\end{proposition}

For $s=6,7,8$, using the methods of Propositions~\ref{mar6-23} and
\ref{mar6-23-1}, in his Ph.D. thesis the second author determined the full list 
of ideals of $S$ generated by squarefree monomials of degree two 
forming a connected graph that
are strongly Cohen--Macaulay  and of linear type
\cite[p.~25]{vila-tesis}. This motivated the second author to 
introduce and study edge ideals of graphs using 
methods coming from commutative algebra \cite{Mats} and graph
theory \cite{cmg}. The Koszul homology of monomial ideals was
studied in \cite{eduardo}.  

\section{Edge ideals of graphs and clutters}\label{edge-ideals-all-flavors}

Stanley--Reisner rings or face rings were  defined independently by
Hochster and Stanley to use commutative algebra on
combinatorial problems, see
\cite{BHer,froberg-survey,Ho3,Rei,RT,Sta3,Sta2,handbook} and
references therein.   In this section, following the approach of Simis and Vasconcelos
\cite{ITG,bowtie,cmg,raei}, we relate the algebraic properties and invariants of edge ideals of graphs and 
clutters with those of the graphs and clutters
defining the edge ideals.  

Let $S=K[t_1,\ldots,t_s]$ be a polynomial ring over a field $K$ and
let $G$ be a simple graph with vertex set $V(G)=\{t_1,\ldots,t_s\}$ 
and edge set $E(G)$. As usual, the monomials of $S$ are denoted
by $t^a:=t_1^{a_1}\cdots t_s^{a_s}$, $a=(a_1,\dots,a_s)$ in $\mathbb{N}^s$, where
$\mathbb{N}=\{0,1,\ldots\}$. 

The \textit{edge ideal} or \textit{graph ideal} of $G$, denoted $I(G)$,  
is the ideal of $S$ given by:
$$I(G):=(\{t_it_j \mid \{t_i,t_j\}\in E(G)\}),$$
and the \textit{edge ring} of $G$ is the quotient ring $S/I(G)$. 
Edge ideals were introduced and studied by the second author in
\cite{cmg}. 
Since the 1990', many authors have been interested in using the edge ideal
construction to build a dictionary between graph theory and
commutative algebra and numerous papers on edge ideals have been 
written, see  \cite{barile,BM,CV,fhv,fhv1,edgeideals,FVT2,froberg-jaco,froberg-vila-eliahou,bounds,graphs,
Herzog-Hibi-book,v-number,kimura-non-v2,Kimura-Terai,kimura-terai, edge-ideals,van-tuyl-survey,monalg-rev} and references
therein. 

There is a very active area initiated in the 2010' that studies the
algebraic properties of binomial edge ideals
of graphs \cite{herzog-binomial-edge-ideal}, see the paper of Conca,
De Negri and E. Gorla \cite{conca-etal}, the book of Herzog,
Hibi and Ohsugi
\cite{herzog-hibi-ohsugi}, the survey paper of Saeedi Madani
\cite{madani}, the paper of Lerda, Mascia, Rinaldo and
Romeo \cite{rinaldo-jaco}, and  references therein. 

Aside from edge ideals of graphs and graph rings,
there has been a lot of work done in the 
intersection of ring theory and graph theory since the 1990' after
the introduction of zero-divisor graphs, see
\cite{Coykendall,Anderson-book,Anderson} and references therein. 

A \textit{tree} is a connected graph without cycles and a
\textit{bipartite} graph is a graph without odd cycles.  The following
theorem shows a family of ideals that are strongly Cohen--Macaulay.

\begin{theorem}{\rm(Simis, Vasconcelos, -\,, \cite[Theorem~3.1]{ITG})} 
For any tree $G$ the ideal $I(G)$ has sliding depth. In particular, if
$G$ is a tree and $I(G)$ is Cohen--Macaulay, then $I(G)$ is strongly Cohen--Macaulay.
\end{theorem}

\begin{definition}{\rm\cite{Faridi}} Let $\Delta$ be a simplicial
complex with vertex set $\{t_1,\ldots,t_s\}$. The 
{\it facet ideal} of $\Delta$, denoted $I(\Delta)$, is the ideal of
$S$ generated by all $\prod_{t_i\in e}t_i$ such that $e$ is a facet of $\Delta$.
\end{definition}

Faridi \cite{faridi2} introduced the notion of a tree for simplicial
complexes and generalized the fact stated above that edge ideals of trees have
sliding depth. 

\begin{theorem}{\rm(Faridi \cite[Theorem~1]{faridi2})}\label{faridi-sd} Let $I = I(\Delta)$ be the
facet ideal of a simplicial complex $\Delta$. If $\Delta$ is a tree, 
then $I(\Delta)$ has sliding depth.
\end{theorem}

Using a formula of Huneke and Rossi for the Krull dimension of the
symmetric algebra of a module \cite{HR,SV2},
\cite[Theorem~1.1.1]{Wolmer-symmetric}, 
a formula for the
Krull dimension of the symmetric algebra of $I(G)$ is given in
\cite{cmg} along
with a description of when this algebra is a domain. A compact
expression for the canonical module of the Rees algebras of 
 edge ideals of complete bipartite graphs, and some formulas for the Cohen--Macaulay type and
Hilbert series of those algebras are presented in
\cite{canela}. Hilbert series of edge rings are studied by Brennan,
Morey, Renteln and Watkins in \cite{BM,renteln,watk}.

The following proposition shows that Cohen-Macaulay edge ideals have a
linear resolution if and only if they have the maximum possible number
of generators.

\begin{proposition}\cite[pp.~54--56]{Koszul-linear}
Let $G$ be a graph with $m$ edges and let $I(G)\subset S$ be its edge ideal. If
$I(G)$ is a Cohen--Macaulay ideal of height $g$, then
$m\leq{g(g+1)}/{2}$, 
with equality if and only if $S/I(G)$  has a $2$-linear resolution.
\end{proposition}

Let $G_0$ be a graph with vertex set 
$Y=\{y_1,\ldots,y_n\}$ and let 
$X=\{x_1,\ldots, x_n\}$ be a new set of vertices. The \textit{whisker
graph} or \textit{suspension} of $G_0$, denoted by
$G_0\cup W(Y)$, is the graph obtained from $G_0$ by attaching to each 
vertex $y_i$ a new vertex $x_i$ and the edge $\{x_i,y_i\}$.  

The significance of the notion of a whisker graph lies 
partly in the next result. 

\begin{theorem}\cite[Theorem~2.4]{cmg}\label{mainVi2}
If $G$ is a tree, then $I(G)$ is Cohen--Macaulay 
 if and only if $G=G_0\cup W(Y)$ for some tree $G_0$ with vertex
set $Y$.
\end{theorem}

Below, we present some generalizations of this theorem to edge ideals of bipartite
graphs and to edge ideals of weighted oriented graphs. 

A set of vertices of a graph $G$ is called 
{\it independent\/} or {\it
stable\/} if no two of them are adjacent. 
The \textit{independence complex} of $G$, denoted $\Delta_G$,
is the simplicial complex whose faces are the stable sets of $G$. Note
that the Stanley--Reisner ideal of $\Delta_G$ is $I(G)$. 
We say that a graph $G$ is \textit{Cohen--Macaulay} (resp. sequentially
Cohen--Macaulay) if $S/I(G)$ is Cohen--Macaulay (resp. sequentially
Cohen--Macaulay), and $G$ is called \textit{shellable} if $\Delta_G$
is a shellable simplicial complex.

The following was the first classification of Cohen--Macaulay
bipartite graphs. In particular, for this family the Cohen--Macaulay
property does not depend on the field $K$. 

\begin{theorem}{\rm(Estrada, -\,, \cite[Theorem~2.9]{EV})}  A bipartite graph $G$ is
Cohen--Macaulay if and only if $\Delta_G$ is pure shellable 
\end{theorem}

The following nice result gives a graph theoretical 
classification of
the family of
Cohen--Macaulay bipartite graphs by looking at the combinatorial
structure of the graphs that define the edge ideals. 
This family of graphs is contained in the class of uniform 
admissible clutters studied in
\cite{floystad,linearquotients,MRV}.  

\begin{theorem}{\rm(Herzog and Hibi 
\cite{herzog-hibi-crit})}\label{herzog-hibi-teo} Let $G$ be a bipartite graph 
without isolated vertices. Then, $G$ is a Cohen--Macaulay graph if and only 
if there is a bipartition $V_1=\{x_1,\ldots,x_g\}$,
$V_2=\{y_1,\ldots,y_g\}$ of $G$ such that: 
{\rm (i)} $\{x_i,y_i\}\in E(G)$ for all $i$, {\rm (ii)} 
if $\{x_i,y_j\}\in E(G)$, then $i\leq j$, and {\rm (iii)} 
if $\{x_i,y_j\}$ and $\{x_j,y_k\}$ are in $E(G)$ and $i<j<k$, 
then $\{x_i,y_k\}\in E(G)$.
\end{theorem} 

Graphs with no chordless cycles of length other than $3$ or $5$ are
sequentially Cohen--Macaulay by a theorem of Woodroofe \cite{Woodroofe}. 
The following result classifies the sequentially Cohen--Macaulay
bipartite graphs. 

\begin{theorem}{\rm(Van Tuyl, -\,, 
\cite[Theorem~3.10]{bipartite-scm})}\label{SCM=shellable}
Let $G$ be a bipartite graph. Then, $G$ is
shellable if and only if $G$ is sequentially Cohen--Macaulay. 
\end{theorem}

There is a recursive procedure to verify whether or not a bipartite graph
is shellable \cite{bipartite-scm}.  
Van Tuyl \cite{vantuyl} has shown that the independence complex
$\Delta_G$ must be vertex decomposable for any 
bipartite graph $G$ whose edge ring $S/I(G)$ is sequentially
Cohen--Macaulay. Thus, Theorem~\ref{SCM=shellable} remains valid if we
replace shellable by vertex decomposable.

A graph is \textit{unmixed} if all its maximal stable sets have the
same cardinality.  Unmixed graphs are also called 
{\it well-covered\/} 
\cite{plummer-unmixed}. 
The following is a combinatorial characterization of all 
unmixed bipartite graphs.  Another characterization was given 
by Ravindra \cite{ravindra}.

\begin{theorem}\cite[Theorem~1.1]{unmixed}\label{unmixed-bip} Let $G$ be a bipartite graph
without isolated vertices. Then $G$ is unmixed if and only 
if $G$ has a bipartition $V_1=\{x_1,\ldots,x_g\}$,
$V_2=\{y_1,\ldots,y_g\}$ such that: 
{\rm (a)} $\{x_i,y_i\}\in E(G)$ for all $i$, and {\rm (b)} 
if $\{x_i,y_j\}$ and $\{x_j,y_k\}$ are in $E(G)$ and $i,j,k$ are
distinct, then $\{x_i,y_k\}\in E(G)$.
\end{theorem} 

Some of the structure theorems for Cohen--Macaulay trees, Cohen--Macaulay bipartite graphs, and 
unmixed bipartite graphs have 
been generalized to very well-covered graphs
\cite{CV,Kimura-Terai,Mahmoudi-et-al} and to weighted 
oriented graphs \cite{unmixed-c-m}. 

\begin{theorem}{\rm(\cite[Theorem~2.3]{CV}, \cite[Theorem~1.1]{Mahmoudi-et-al})} 
Let $G$ be a very well-covered graph. Then, $G$ is Cohen--Macaulay 
if and only if $G$ is vertex decomposable.
\end{theorem}

Let $G=(V(G), E(G))$ be a graph without isolated vertices 
with vertex set $V(G)$ and edge
set $E(G)$. A {\it weighted oriented graph\/} $D$, whose {\it underlying
graph\/} is $G$, is a triplet $(V(D),E(D),w)$ where $V(D)=V(G)$,
$E(D)\subset V(D)\times V(D)$ such that 
$$
E(G)=\{\{x,y\}\mid (x,y)\in
E(D)\},
$$ 
$|E(D)|=|E(G)|$, and $w$ is a \textit{weight function} 
$w\colon V(D) \to\mathbb{N}_+$, where $\mathbb{N}_+=\{1,2,\ldots\}$. The \textit{vertex set} of $D$ and the \textit{edge set} of $D$
are $V(D)$ and $E(D)$, respectively. For simplicity we denote
these sets by $V$ and $E$, respectively. The \emph{weight} of $x\in V$
is $w(x)$ and the set of vertices $\{x\in V\mid w(x)>1\}$ is denoted by
$V^{+}$. If $V(D)=\{t_1,\ldots,t_s\}$, we can regard each vertex $t_i$ as a
 variable and consider the polynomial ring
 $S=K[t_1,\ldots,t_s]$ over a ground field $K$. The
 \textit{edge ideal} of $D$, introduced in \cite{cm-oriented-trees,WOG}, 
is the ideal of $S$ given by 
$$I(D):=(\{t_{i}t_{j}^{w(t_{j})}\mid(t_{i},t_{j})\in E(D)\}).$$
\quad If $w(t_i)=1$ for each $t_i\in V(D)$, then $I(D)$ is the usual edge
ideal $I(G)$ of the graph $G$. The
motivation to study $I(D)$ comes from coding theory, see
\cite{carvalho-lopez-lopez}, \cite[p.~536]{reyes-vila}, and
\cite[p.~1]{WOG}.  In general, edge
ideals of weighted oriented 
graphs are different from edge ideals of 
edge-weighted (undirected) graphs defined by Paulsen and
Sather-Wagstaff \cite{PS}.  

The projective
dimension, regularity, and algebraic and
combinatorial properties of edge ideals of
weighted oriented graphs have been studied in
\cite{cm-oriented-trees,reyes-vila,depth-monomial,
WOG,Zhu-Xu,Zhu-Xu-Wang-Zhang}. The first major
result about $I(D)$ is an explicit combinatorial
expression of Pitones, Reyes and Toledo \cite[Theorem~25]{WOG} 
for the irredundant decomposition of $I(D)$ as a finite intersection of
irreducible monomial ideals. If $D$ is transitive, then Alexander duality holds
for $I(D)$ \cite[Theorem~4]{cm-oriented-trees}.

Following \cite[p.~136]{Mats}, we say $I(D)$ is \textit{unmixed} if all its associated primes have the same height and
$I(D)$ is called \textit{Cohen--Macaulay} if $S/I(D)$ is a Cohen--Macaulay
ring. We say that $D$ is \textit{unmixed} (resp.
\textit{Cohen--Macaulay}) if $I(D)$ is unmixed (resp.
Cohen--Macaulay). 

A subset $C\subset V(G)$ 
is a {\it vertex cover\/} of a graph $G$ if $V(G)\setminus C$ is a
stable set of $G$. The graph $G$ is \textit{well-covered} if all maximal stable sets of
$G$ have the same cardinality and the graph $G$ is  \textit{very
well-covered} if $G$ is well-covered and $|V(G)|=2\alpha_0(G)$, where
$\alpha_0(G)$ is the cardinality of a minimum vertex cover of $G$.  
The class of very well-covered graphs contains the bipartite well-covered graphs studied by Ravindra
\cite{ravindra} and more recently revisited in \cite{unmixed}. 
A set of pairwise disjoint edges of a graph $G$ is called {\it
independent\/} or a
{\it matching\/} and a set
of independent edges of $G$ whose union is $V(G)$ is called a
{\it perfect matching\/}. One of
the properties of very well-covered graphs is that they can be classified using combinatorial
properties of a perfect matching as shown by a central result of
Favaron \cite[Theorem~1.2]{favaron}.

The Cohen--Macaulay property of $I(D)$ depends on the characteristic of the field $K$
\cite[p.~214]{monalg-rev}. For graphs $I(G)$ is unmixed if and only if
$G$ is well-covered \cite[Lemma 6.3.37]{monalg-rev}. The
unmixed property of $I(D)$ 
depends only on the combinatorics of $D$ \cite[Theorem~31]{WOG}.

We denote the in- and out-neighborhood of a vertex $a$ of $D$ by 
$N_{D}^{-}(a)$ and $N_{D}^{+}(a)$, respectively, and the neighborhood 
of $a$ by $N_{D}(a)$. The graph $G$ is \textit{K\H{o}nig} if
$\alpha_0(G)$ is the \textit{matching number} $\beta_1(G)$ of $G$, that is, the maximum
cardinality of a matching of $G$.  A perfect matching $P$ of a
graph $G$ has \textit{property} {\bf(P)} 
if for all $\{a,b\}$, $\{a^{\prime},b^{\prime}\}\in E(G)$, and
$\{b,b^{\prime}\}\in P$, one has $\{a,a^{\prime}\}\in E(G)$.

The following two theorems give a combinatorial characterization of 
the unmixed property and the Cohen--Macaulay property of $I(D)$ when
the underlying graph $G$ is K\H{o}nig.
\begin{theorem}{\rm(Pitones, Reyes, -\,,
\cite[Theorem~3.4]{unmixed-c-m})}\label{prop-unmixed}  
If $D$ is a weighted oriented graph and $G$ is
K\H{o}nig, then $I(D)$ is unmixed if and only if $D$
satisfies the following conditions:    
\begin{enumerate}
\item[(1)] $G$ has a perfect matching $P$ with property {\bf(P)}.
\item[(2)] If $a\in V(D)$, $w(a)>1$, $b^{\prime}\in N_{D}^{+}(a)$, 
and $\{b,b^{\prime}\}\in P$, then $N_{D}(b)\subset N_{D}^{+}(a)$.
\end{enumerate}
\end{theorem}
Conditions (1) and (2) of Theorem~\ref{prop-unmixed} also 
characterize the unmixed property of the ideal $I(D)$ when 
$G$ is a graph without $3$-, $5$-, and $7$-cycles 
\cite[Proposition~3.7]{unmixed-c-m}. 

\begin{theorem}{\rm(Pitones, Reyes, -\,,
\cite[Theorem~4.3]{unmixed-c-m})}\label{theoremCM}
If $D$ is a weighted oriented graph and $G$ is
K\H{o}nig, then, $I(D)$ is Cohen--Macaulay if and only if $D$
satisfies the following conditions:  
\begin{enumerate}
\item[(1)] $G$ has a perfect matching $P$ with property {\bf(P)} and has no $4$-cycles 
with two edges in $P$.
\item[(2)] If $a\in V(D)$, $w(a)>1$, $b^{\prime}\in N_{D}^{+}(a)$, 
and $\{b,b^{\prime}\}\in P$, then $N_{D}(b)\subset N_{D}^{+}(a)$.
\end{enumerate}
\end{theorem}

Conditions (1) and (2) of Theorem~\ref{theoremCM}
also characterize the Cohen--Macaulay property of $I(D)$ when 
$G$ is a graph without $3$- and $5$-cycles
\cite[Proposition~4.5]{unmixed-c-m}.
In general any graded Cohen--Macaulay ideal is unmixed \cite{Mats}. 
The \textit{girth} of $G$ is the length of a shortest cycle
contained in $G$. If
$G$ is a K\H{o}nig graph without $4$-cycles or
$G$ has girth greater than $7$, then $I(D)$ is unmixed if and only if $I(D)$ is 
Cohen--Macaulay \cite[Corollaries~4.4 and 4.7]{unmixed-c-m}. For graphs this improves a result of
\cite[Corollary~2.19]{MRV} showing that unmixed K\H{o}nig
clutters without $3$- and $4$-cycles are
Cohen--Macaulay. If $I(D)$ is Cohen--Macaulay, then $I(D)$ is unmixed and $I(G)$ is
Cohen--Macaulay (see \cite[Theorem~2.6]{Radical-Herzog} and 
\cite[Proposition 51]{WOG}). The converse was a conjecture of
Pitones, Reyes and Toledo \cite[Conjecture~53]{WOG} that is disproved in
\cite{Fakhari-terai-etal}. This conjecture is true when $G$ has no
$3$- or $5$- cycles, or $G$ is K\H{o}nig
\cite[Corollary~4.6]{unmixed-c-m}. Recently, Dung and 
Trung proved the conjecture 
when $G$ has girth at least $5$ \cite[Theorem~2.8]{Dung-Trung}. 

Graphs with a whisker (i.e., a pendant edge) attached to 
each vertex are K\H{o}nig \cite[p.~277]{monalg-rev},  
very well-covered graphs are also K\H{o}nig
\cite[Remark~2.18]{unmixed-c-m}, and bipartite
graphs are K\H{o}nig and have no odd cycles \cite{Har}. Then, some of
the results above generalize those of
\cite{cm-oriented-trees,reyes-vila,WOG,unmixed}. 
More precisely, Theorem~\ref{prop-unmixed} (resp.
Theorem~\ref{theoremCM}) generalizes the unmixed criteria of
\cite[Theorem~46]{WOG} and \cite[Theorem~1.1]{unmixed} (resp.
Cohen--Macaulay criterion of \cite[Theorem~5.1]{reyes-vila}) 
 for weighted oriented bipartite graphs. 
From \cite[Corollary~4.4 and Proposition~4.5]{unmixed-c-m},
 we recover the Cohen--Macaulay criterion of 
\cite[Theorem~5]{cm-oriented-trees} for weighted oriented trees. 

Edge ideals of graphs can be generalized to clutters \cite{clutters,Ha-VanTuyl}. Let
$S=K[t_1,\ldots,t_s]$ be a polynomial ring over a field $K$ and let
$\mathcal C$ be a \textit{clutter} with vertex  
set $V(\mathcal{C})=\{t_1,\ldots,t_s\}$, that is, $\mathcal C$ is a 
family of subsets $E(\mathcal{C})$ of $V(\mathcal{C})$, called edges,
none of which is contained in
another. For example, a graph $G$ (no multiple edges or loops) is a
clutter. The
 \textit{edge ideal} of $\mathcal{C}$, denoted $I(\mathcal{C})$, 
is the ideal of $S$ given by 
$$I(\mathcal{C}):=(\textstyle\{\prod_{t_i\in e}t_i\mid e\in
E(\mathcal{C})\}).$$
\quad The
minimal set of generators of
$I(\mathcal{C})$ is  the
set of all squarefree monomials $t_e:=\prod_{t_i\in e}t_i$ such 
that $e\in E(\mathcal{C})$. 
Any squarefree monomial ideal $I$ of $S$ is the edge
ideal $I(\mathcal{C})$ of a clutter $\mathcal{C}$ with vertex set
$V(\mathcal{C})=\{t_1,\ldots,t_s\}$. Indeed, let
$I$ be a squarefree monomial ideal 
of $S$ minimally generated by the set of 
monomials $\mathcal{G}(I):=\{t^{v_1},\ldots,t^{v_m}\}$ and let
$\mathcal{C}$ be the clutter whose edges are $f_1,\ldots,f_m$, where 
$f_k$ is ${\rm supp}(t^{v_k})$, the support of $t^{v_k}$, consisting
of all $t_i$ that occur in $t^{v_k}$. Then, $\mathcal{C}$ is a clutter
 and $I=I(\mathcal{C})$. The $s\times m$ matrix $A$ with column vectors $v_1,\ldots,v_m$ is
 the \textit{incidence matrix} of the clutter $\mathcal{C}$ and is
 also the \textit{incidence matrix} of the edge ideal  $I=I(\mathcal{C})$.  

Thus, one has the
Stanley--Reisner correspondence between simplicial complexes and
squarefree monomial ideals, and the correspondence between
clutters and edge ideals
$$
\Delta\longleftrightarrow
I_\Delta
\ \quad \ \mbox{ and }\ \quad \ \mathcal{C}\longleftrightarrow
I(\mathcal{C}),
$$
respectively. A subset $F$ of $V(\mathcal{C})$
is called {\it independent\/} or {\it stable\/} if $e\not\subset F$
for any  
$e\in E(\mathcal{C})$. If $\Delta_\mathcal{C}$ is
the independence complex of $\mathcal{C}$ whose faces are the stable 
sets of $\mathcal{C}$, then
$I_{\Delta_{\mathcal{C}}}=I(\mathcal{C})$. 

A subset $C$ of $V(\mathcal{C})$ is called a {\it vertex cover\/} 
if $V(\mathcal{C})\setminus C$ is a stable set. The {\it covering
number\/} of $\mathcal{C}$, denoted $\alpha_0(\mathcal{C})$, is the 
number of vertices in a minimum vertex cover of $\mathcal{C}$.
A clutter $\mathcal{C}$ has the
{\it K\H{o}nig property\/} if the maximum number of pairwise disjoint
edges is $\alpha_0(\mathcal{C})$. 

Any unmixed clutter
with the K\H{o}nig property and without isolated vertices satisfies
the hypothesis of the following characterization of unmixed clutters
(cf.~Theorem~\ref{unmixed-bip}). 

\begin{theorem}{\rm(Morey, Reyes, -\,, \cite{MRV})}\label{unmixed-clutter-pm} 
A clutter $\mathcal{C}$ with a set of pairwise disjoint 
edges $\{e_i\}_{i=1}^g$, $g=\alpha_0(\mathcal{C})$, that covers
$V(\mathcal{C})$ is unmixed if and only if any of the
following conditions hold:
\begin{enumerate}
\item[(a)] For any two edges $e \neq e'$ and for any two distinct
vertices  
$x\in e$, $y\in e'$ contained in some $e_i$, one has that 
$(e\setminus\{x\})\cup (e'\setminus\{y\})$ contains an edge. 
\item[(b)] $I(\mathcal{C})=(I(\mathcal{C})^2\colon t_{e_1})+\cdots+(I(\mathcal{C})^2\colon t_{e_g})$.
\end{enumerate}
\end{theorem}

A graph $G$ is called {\it chordal\/} if 
every cycle $C_r$ of $G$ of length $r\geq 4$ has a 
chord in $G$. A {\it chord}
 of $C_r$ is an edge joining two 
non-adjacent vertices of $C_r$. A chordal graph
is called {\it strongly chordal} if 
every cycle $C_r$ of even length at least six has a chord that divides
$C_r$ into two odd length paths. A {\it clique\/} of a graph $G$
is a set of vertices inducing a complete subgraph. The {\it clique
clutter\/} 
of $G$, denoted by ${\rm cl}(G)$, is the clutter on $V(G)$ whose edges are the 
maximal cliques of $G$ (maximal with respect to
inclusion).

Let $A$ be the incidence matrix of a clutter $\mathcal{C}$. The clutter 
$\mathcal{C}$ has a {\it special cycle\/} of length
$r$ if there 
is a square
submatrix of $A$ of 
order $r\geq 3$ with exactly two $1$'s in 
each row and column.  A clutter with no special odd cycles is 
called {\it balanced} and a clutter with no
special cycles is called {\it totally
balanced\/}. A graph $G$ is balanced (resp. totally balanced) if and
only if $G$ is bipartite (resp. a forest). We say $t_i$ is a {\it free
variable\/} (resp. {\it free 
vertex}) of $I(\mathcal{C})$ (resp. ${\mathcal C}$) if $t_i$
only appears in one of the monomials of the minimal generating set
$\mathcal{G}(I)$ of $I$ (resp. in
one of the edges of $\mathcal C$).  

The notion of a minor is defined in Section~\ref{blowuprings-section}
after Theorem~\ref{ntf-char}.  
If all the minors of a clutter 
$\mathcal C$ have free vertices, we say that ${\mathcal C}$ 
has the {\it free vertex property}. This property is closed under
minors, that is, if $\mathcal C$ 
has the free vertex property, then so do all of its minors. The following result classifies the simplicial trees introduced by 
Faridi \cite{faridi2}, whose facet ideals have sliding depth
(Theorem~\ref{faridi-sd}).

\begin{theorem}\label{soleyman-zheng-char}
A clutter $\mathcal{C}$ is totally balanced if and only if any of the
following equivalent conditions hold:
\begin{itemize}
\item[\rm(a)] {\rm\cite[Theorem~3.2]{hhtz}} $\mathcal{C}$ is the
clutter of the facets of a simplicial forest.
\item[\rm(b)] {\rm \cite[Corollary~3.1]{soleyman-zheng}} 
$\mathcal{C}$ has the free vertex property.
\item[\rm(c)] {\rm \cite{farber}} $\mathcal{C}$ is the clique clutter of a 
strongly chordal graph.
\end{itemize}
\end{theorem}

The \textit{ideal of covers} of a clutter $\mathcal{C}$, denoted
$I_c(\mathcal{C})$, is the ideal of $S$ generated by
all squarefree monomials whose support is a minimal vertex cover of
$\mathcal{C}$ \cite[p.~221]{monalg-rev}. In the context of Stanley--Reisner theory of simplicial
complexes, $I_c(\mathcal{C})$ is called the
\textit{Alexander dual} of $I=I(\mathcal{C})$ and is denoted by
$I^\vee$ \cite[pp.~17--18]{Herzog-Hibi-book}. The clutter of minimal
vertex covers of $\mathcal{C}$ is denoted by $\mathcal{C}^\vee$ or
$b(\mathcal{C})$ and is called
the \textit{blocker} of $\mathcal{C}$ \cite{cornu-book}.

The interaction between graph
theory and  commutative
algebra is present in Lyubeznik thesis \cite{Lyu-thesis,Lyu1} where 
it is shown that the ideal of covers of a graph $G$ is
Cohen--Macaulay if and only if the complement of $G$ is a chordal
graph, and in the work of
Fr\"oberg \cite{Fro4} where it is shown that the edge ideal of a graph
$G$ has a $2$-linear resolution if and only if the complement of $G$
is a chordal graph. The next duality theorem is related to these two facts. 

\begin{theorem}{\rm(Eagon-Reiner \cite{ER})}\label{eagon-reiner-terai}  
Let $\mathcal{C}$ be a clutter and let $I_c(\mathcal{C})$ be its ideal
of covers. Then, $I(\mathcal{C})$ 
is Cohen--Macaulay if and only if $I_c(\mathcal{C})$
has a linear resolution.
\end{theorem}

There is a duality formula of Terai relating the regularity of an edge ideal 
with the projective dimension of its Alexander dual. A reference for
the regularity is the book of Eisenbud
\cite{eisenbud-syzygies}. 

\begin{theorem}{\rm(Terai~\cite{terai})}\label{terai-formula} 
 If $\mathcal{C}$ is a clutter, then ${\rm reg}(S/I(\mathcal{C}))={\rm pd}(
S/I_c(\mathcal{C}))-1$.
\end{theorem}

Some of the algebraic invariants of a sequentially Cohen--Macaulay 
monomial ideal $I$ of $S$ can be expressed in terms of the {\it big height}
of $I$. This number is denoted by ${\rm bight}(I)$ and is the 
largest height of an
associated prime of $I$. 

\begin{theorem}{\rm(Morey, -\,, \cite[Theorem~3.31,
Corollary~3.33]{edge-ideals})}
\label{reg-vila-morey-scm} Let $\mathcal{C}$ be a clutter, 
let $I(\mathcal{C})$ be its edge ideal and let 
$I_c(\mathcal{C})$ be its ideal of covers. Then,
\begin{enumerate} 
\item[(a)] 
${\rm reg}(S/I_c(\mathcal{C}))\geq {\rm bight}(I(\mathcal{C}))-1$  
and {\rm(b)} ${\rm pd}_S(S/I(\mathcal{C}))\geq {\rm
bight}(I(\mathcal{C}))$,
\end{enumerate}
with equality everywhere if $S/I(\mathcal{C})$ is sequentially Cohen--Macaulay. 
\end{theorem}

\begin{corollary}\cite{edge-ideals}\label{pd-formula-sequentially-cm} If
$I$ is a monomial ideal and $S/I$ is sequentially
Cohen--Macaulay, then 
${\rm pd}_S(S/I)={\rm bight}(I)$.
\end{corollary}

An {\it induced matching\/} in a graph $G$ is a set of
pairwise disjoint edges $f_1,\ldots,f_r$ such that the only edges of 
$G$ contained in $\bigcup_{i=1}^rf_i$ are $f_1,\ldots,f_r$. The
{\it induced matching number\/} of $G$,
denoted ${\rm im}(G)$, is the number of edges in a largest
induced matching.

\begin{theorem}{\rm(Mahmoudi, Mousivand, Crupi, Rinaldo, Terai, 
Yassemi \cite{Mahmoudi-et-al})}
Let $G$ be a very well-covered graph and let ${\rm im}(G)$ be its
induced matching number. Then, 
$$
{\rm reg}(S/I(G))={\rm im}(G).
$$
\end{theorem}

An \textit{odd hole} in a graph is an induced odd cycle and an
\textit{odd antihole} is an induced complement of an odd cycle. In
graph theory, these notions appear in the strong perfect graph
theorem, proved by Chudnovsky, Robertson, Seymour, and Thomas
\cite{seymour-perfect}, showing that a graph $G$ is perfect if and only
if $G$ has no odd holes or odd antiholes of length at least five. In
particular, by this theorem, one recovers the weak perfect graph
theorem proved by Lov\'asz \cite{Lovasz} showing that 
the complement of a perfect graph is perfect. 

In commutative algebra odd holes occurred in the work of Simis 
\cite{aron-hoyos}, and later in the description
of Simis and Ulrich of $I(G)^{\{2\}}$, the join of an edge ideal of a graph $G$ with 
itself \cite{simis-ulrich}, and in the following description of the
associated primes of the second power of ideals of covers of 
graphs. The paper of Francisco, H$\rm \grave{a}$ and 
Mermin \cite{francisco-survey}
surveys algebraic techniques for detecting odd cycles and odd holes
in a graph and for computing the chromatic number of a simple hypergraph.  

\begin{theorem}{\rm(Francisco, H$\rm \grave{a}$, Van Tuyl
\cite{fhv})} If $G$ is a graph and
$\mathfrak{p}$ is an associated prime of $I_c(G)^2$, then 
$\mathfrak{p}=(t_i,t_j)$ for some edge $\{t_i,t_j\}$ of
$G$ or $\mathfrak{p}=(t_i\vert\, \in A)$, where $A$ is a
set of vertices of $G$ that induces an 
odd hole. 
\end{theorem}

\begin{theorem}{\rm(Francisco, H$\rm \grave{a}$, Van Tuyl
\cite{fhv1})}\label{chromatic-membership} If $G$ is a simple
graph, then the chromatic number of $G$ is the minimal
$k$ such that $(t_1\cdots t_s)^{k-1}\in I_c(G)^k$.
\end{theorem}

\section{Rees algebras, symbolic powers and normalizations of graph
ideals}\label{normality-symbolic-powers}

There are methods to compute the symbolic powers of prime 
ideals in a polynomial ring $S$, see \cite{aron-symbolic} 
and \cite[Chapter~3]{Vas1}, and there are algorithms for
computing the symbolic powers of ideals of $S$ \cite{Grifo-etal}. 
One of these algorithms uses
the methods of  Eisenbud, Huneke, and Vasconcelos for finding 
primary decompositions of ideals of $S$ \cite{EHV}. 

Let $S=K[t_1,\ldots,t_s]$ be a 
polynomial ring over a 
field $K$ and $I=(f_1,\ldots,f_m)$ an ideal of 
height $g$ of $S$. The 
{\it Jacobian ideal\/} $J$ of $I$ is 
the ideal generated by 
the $g\times g$ minors of the Jacobian matrix 
${\mathfrak J}=(\partial f_i/\partial t_j)$.

The following is a subtle 
application of the Jacobian criterion for computing the symbolic
powers of a prime ideal.

\begin{proposition}{\rm(Vasconcelos~\cite[Proposition~3.5.13]{Vas1})}\label{effesymbvas}
Let $S$ be a 
polynomial ring over a 
field $K$ and let $I$ be a prime ideal. If $f$ is an element in the Jacobian ideal $J$
of $I$ which is not in the ideal $I$, then the $n$-the symbolic 
power $I^{(n)}$ of $I$ is given by
$$
\begin{array}{c}
I^{(n)}=(I^n\colon f^\infty)\ \mbox{ for all }\,n\geq 1,
\end{array}
$$
and such element $f$ exist if $K$ is a perfect field.
\end{proposition}

The symbolic powers of monomial ideals are easier to compute and in the
case of edge ideals they admit a description in terms 
of the powers of the associated primes, see Eq.~\eqref{symb-edge}
below.

An ideal $I\subset S$ has the {\it persistence
property\/} if the sets of associated primes ${\rm
Ass}(S/{{I^k}})$ form an {\it
ascending chain\/}. Edge ideals of graphs have the persistence property.  

\begin{theorem}{\rm(Mart\'inez-Bernal, Morey, -\,,
\cite[Theorem~2.15]{ass-powers})}\label{persistence-edge-ideals}  
Let $G$ be a graph and let $I=I(G)$ be its edge
ideal. Then 
$${\rm Ass}(S/I^k) \subset{\rm Ass}(S/I^{k+1})\mbox{ for all }k\geq 1. $$
\end{theorem}

In general squarefree monomial ideals do not have the persistence
property \cite{persistence-ce}. Persistence is studied in several papers, 
see for instance
\cite{bhat-persistence,francisco-persistence,persistence-herzog,
persistence-ce,jtt1,lam-trung,edge-ideals,jtt,reyes-toledo}.

Let $G$ be a graph with vertex set $V(G)$. A subset $C\subset V(G)$ 
is a {\it minimal vertex cover\/} of 
$G$ if every edge of $G$ contains at least one vertex in $C$, 
and there is no proper subset of $C$ with this property. A prime ideal
$\mathfrak{p}$ is an associated prime of $I(G)$ if and only 
if $\mathfrak{p}$ is generated by a minimal vertex cover of $G$
\cite[p.~279]{cmg}. 

Let ${\mathfrak p}_1,\ldots,{\mathfrak p}_r$ 
be the associated primes of $I(G)$. Given an integer $k\geq 1$, the 
$k$-th {\it symbolic power} of 
$I(G)$, denoted $I(G)^{(k)}$, is the ideal given by 
\begin{equation}\label{symb-edge}
I(G)^{(k)}=\mathfrak{p}_1^k\cap\cdots\cap \mathfrak{p}_r^k,
\end{equation}
see for instance \cite[Proposition~4.3.25]{monalg-rev}. 

An ideal $I\subset S$ is called
\textit{normally torsion-free} if ${\rm Ass}(I^n)\subset{\rm Ass}(I)$ 
for all $n\geq 1$. For edge ideals of graphs, the following theorem relates this
algebraic property 
with a graph theoretical property. This result has had a strong
impact to date
\cite{bahiano,AJ,icdual,faridi2,Ha-VanTuyl,Haase-Santos,
Herzog-Hibi-book,hhtz,Jayanthan-Kumar-Mukundan,Mandal-Pradhan,Pokora-Romer,ic-resurgence}.

\begin{theorem}{\rm(Simis, Vasconcelos, -\,,  
\cite[Theorem~5.9]{ITG})}\label{ntf-bipartite}
Let $G$ be a graph and let $I(G)$ be its edge ideal. The following 
conditions are equivalent:
\begin{enumerate}
\item[(i)] $G$ is a bipartite graph.
\item[(ii)] $I(G)$ is normally torsion free.
\item[(iii)] $I(G)$ is a Simis ideal, that is, $I(G)^{(k)}=I(G)^k$
for all $k\geq 1$.
\end{enumerate}
\end{theorem}

The \textit{odd girth} of a non-bipartite graph is the size of the shortest induced
odd cycle. 
The next theorem is due to Dao, De Stefani, Grifo, Huneke 
and N\'u\~nez-Betancourt \cite{symbolic-powers-survey}.

\begin{theorem}\cite[Theorem~4.13]{symbolic-powers-survey}\label{dao-etal}
Let $G$ be a non-bipartite graph and let $r_0$ be the least positive integer
such that $I(G)^{r_0}\neq I(G)^{(r_0)}$, then $2r_0-1$ is the odd
girth of $G$.
\end{theorem}

\begin{remark} Let $G$ be a non-bipartite graph and let $2r_0-1$ be
the odd girth of $G$. By the persistence property of $I=I(G)$, one has
the inclusions ${\rm Ass}(I)\subsetneq{\rm Ass}(I^{r_0})\subset
{\rm Ass}(I^{n})$ for all $n\geq r_0$. Then, we get 
$I^{n} \neq I^{(n)}$ for all $n\geq r_0$.         
\end{remark}

\quad The \textit{Rees algebra} and the \textit{symbolic Rees algebra} of 
$I=I(G)$ are given by 
\begin{align*}
\mathcal{R}(I):=&S[Iz]=\displaystyle S\oplus Iz\oplus\cdots\oplus
I^nz^n\oplus\cdots\subset S[z], \\
\mathcal{R}_s(I):=&\displaystyle S\oplus I^{(1)}z\oplus\cdots\oplus
I^{(n)}z^n\oplus\cdots\subset S[z], 
\end{align*}
respectively, where $z$ is a new variable. 
If $R$ is an integral domain with field
of fractions $K_R$, recall that the \textit{normalization} or
\textit{integral closure} of $R$ is the
subring $\overline{R}$ consisting of all the elements of $K_R$ that
are integral over $R$. If $R=\overline{R}$ we say that $R$ is {\it
normal\/}. The integral closure of $S[Iz]$ is given by
\cite[p.~168]{Vas1}: 
\begin{equation}\label{jul30-22-2}
\overline{S[Iz]}=S\oplus \overline{I}z\oplus\cdots
\oplus \overline{I^n}z^n\oplus\cdots\subset S[z],
\end{equation}
where $\overline{I^n}$ is the integral closure of $I^n$ (see
Proposition~\ref{icd}). If 
$I^n=\overline{I^n}$ for all $n\geq 1$, the ideal $I$ is called
\textit{normal}. Thus, by Eq.~\eqref{jul30-22-2}, 
the ring $S[Iz]$ is normal if and only if the ideal $I$
is normal. The ideal $I$ is said to be \textit{integrally closed} or \textit{complete} 
if $I=\overline{I}$. 

The minimal generators of $\overline{\mathcal{R}(I)}$ were described by
Vasconcelos using the cycle structure of $G$
(Theorem~\ref{wolmer-ic-rees}). The minimal generators of
$\mathcal{R}_s(I)$ also have a
combinatorial interpretation, they are in one to one
correspondence with the indecomposable parallelizations of $G$ \cite{symboli}. 

Symbolic powers of squarefree monomial ideals are integrally closed
\cite[Corollary 4.3.26]{monalg-rev}.  Interesting families of normal
monomial ideals include polymatroidal 
ideals \cite{matrof}, ideals of covers of perfect graphs
\cite{perfect}, 
edge ideals of graphs with no Hochster configurations
(Theorem~\ref{apr16-02}), normally torsion free squarefree monomial
ideals \cite[Corollary~5.3]{ITG}, and integrally closed ideals in two variables
\cite[Appendix~5]{zariski-samuel-ii}. The normality of the first two families were shown using 
combinatorial optimization methods and polyhedral geometry. 
Determinantal rings are normal and there is a
description of the symbolic powers of determinantal ideals, see the survey
article of Bruns and Conca \cite{bruns-conca}.  

The Normality of ideals is related to the persistence property of associated primes.

\begin{theorem}\label{integralclosurechain}{\rm(\cite[Proposition~3.9]{McAdam},
\cite{ratliff,ratliff-increasing})} 
If $S$ is a Noetherian ring and $I$ is an ideal, 
then the sets ${\rm Ass}(S/{\overline{I^k}})$ form an ascending
chain. 
\end{theorem}

The following result of Vasconcelos gives an easy-to-use formula for
the integral closure of powers of monomial ideals.

\begin{proposition}{\rm(Vasconcelos \cite[p.~169]{Vas1})}\label{icd}
If $I$ is a monomial ideal of $S$ and $n\in\mathbb{N}_+$, then 
\begin{equation*}
\overline{I^n}=(\{t^a\in S\mid (t^a)^{p}\in I^{pn}
\mbox{ for some }p\geq 1\}).
\end{equation*}
\end{proposition}

In an email communication, Hochster showed to Vasconcelos 
the first example of a connected graph whose edge ideal is not normal
\cite[p.~457]{monalg-rev} (cf. \cite[Example~4.9]{ITG}). This example leads to the following concept
\cite[Definition~6.7]{ITG}. 

A \textit{Hochster configuration} of a graph $G$ consists of two odd
cycles $C\sb{1}$, $C\sb{2}$ of $G$ satisfying the following
two conditions:
\begin{enumerate}
\item[(i)] 
$C\sb {1}\cap N_G(C\sb{2})=\emptyset$, where $N_G(C_2)$
is the neighbor set of $C_2$.  
\item[(ii)] No chord of $C\sb {i}$, $i=1,2$, is
an edge of $G$, i.e., $C_i$ is an induced cycle of $G$.
\end{enumerate}

\begin{lemma}\cite[Lemma~5.7]{icdual}\label{sc-lemma} Let $I=I(G)$ be
an edge ideal, let $C_1$,  
$C_2$ be two odd cycles of $G$ with $|C_1\cap C_2|\leq 1$, and let
$M\sb{C1,C_2}:=(\prod_{t_i\in C_1}t_i\prod_{t_i\in C_2}t_i)z\sp{{(|C_1|+|C_2|)}/{2}}$. The following hold.
\begin{enumerate}
\item[(a)] If $|C_1\cap C_2|=1$, then $M_{C_1,C_2}\in S[Iz]$.
\item[(b)] If $C_1\cap C_2=\emptyset$ and there is $e\in E(G)$
intersecting $C_1$ and $C_2$, then $M_{C_1,C_2}\in S[Iz]$.
\item[(c)] If $C_1,C_2$ form a Hochster configuration, then
$M_{C_1,C_2}\notin S[Iz]$.
\end{enumerate}
\end{lemma}

We come to the Vasconcelos formula in terms of Hochster configurations
for the normalization of the Rees
algebra of the edge ideal of a graph.

\begin{theorem}{\rm(Vasconcelos~\cite[p.~459]{monalg-rev})}\label{wolmer-ic-rees}
Let $I=I(G)$ be the edge ideal of a graph $G$ and let $\mathcal{U}$ be
the set of all 
monomials $M_{C_1,C_2}$ such that $C_1$, $C_2$ is a Hochster
configuration of $G$. Then, the integral
closure of $S[Iz]$ is given by 
$$
\overline{S[Iz]}=S[Iz][\mathcal{U}].
$$
\end{theorem}

\begin{proof} The integral closure $\overline{S[Iz]}$ of $S[Iz]$ is equal to 
$S[Iz][\mathcal{B}']$ \cite[p.~459]{monalg-rev},  
where $\mathcal{B}'$ is the set of all 
monomials $$M\sb{C1,C_2}:=\Big(\prod_{t_i\in C_1}t_i\prod_{t_i\in
C_2}t_i\Big)z\sp{{(|C_1|+|C_2|)}/{2}}$$ 
such
that $C_1$ and $C_2$ are two induced odd cycles of $G$ with at most one common
vertex. If $C_1$ and
$C_2$ intersect at a point or $C_1$ and $C_2$ are joined by at least
one edge of $G$, then $M_{C_1,C_2}$ is in
$S[Iz]$ by Lemma~\ref{sc-lemma}. Hence, 
$\overline{S[Iz]}=S[Iz][\mathcal{U}]$.
\end{proof}

It was conjectured in \cite{ITG} that the 
edge ideal of a graph $G$ is normal if and only if the graph has no
Hochster configurations. This 
conjecture was proved in \cite[Corollary~5.8.10]{graphs}, 
\cite[Corollary~10.5.9]{monalg-rev} (cf. 
\cite[p.~410]{hibi-ohsugi-normal}, \cite[p.~283]{bowtie}). The
following is a recent proof of this  
conjecture \cite[Theorem~5.8]{icdual} using Vasconcelos'
description of the integral closure
of the Rees algebra of the edge ideal $I(G)$. 

\begin{theorem}{\rm(\cite[Conjecture~6.9]{ITG}, 
\cite[Corollary~5.8.10]{graphs})}\label{apr16-02} 
The edge ideal $I(G)$ of a 
graph $G$ is
normal if and only if $G$ admits no Hochster configurations.
\end{theorem} 

\begin{proof} 
To show this result we use the description of
Theorem~\ref{wolmer-ic-rees} for the integral closure of the Rees
algebra $S[Iz]$ 
of the edge ideal $I=I(G)$:  
$$
\overline{S[Iz]}=S[Iz][\mathcal{U}],
$$
where $\mathcal{U}$ is the set of all 
monomials $M\sb{C1,C_2}$ such that $C_1$, $C_2$ is a Hochster
configuration of $G$. Therefore, by Lemma~\ref{sc-lemma}(c), $S[Iz]$ is normal if and only if $G$ has no Hochster
configurations, and the result follows from the fact that $I$ is
normal if and only if $S[Iz]$ is normal. 
\end{proof}

\begin{corollary}\cite{ITG}\label{bip-normal} 
If $G$ is a bipartite graph, then $I(G)$ is normal
\end{corollary}

The \textit{ideal of covers} of a graph $G$, denoted $I_c(G)$, is the
ideal of $S$ generated by all monomials $\textstyle\prod_{t_i\in C}t_i$ such
that $C$ is a vertex cover of $G$. By \cite{alexdual},
Theorem~\ref{ntf-dual-bipartite}, $I_c(G)$ is a Simis ideal if and
only if $G$ is a bipartite graph. A main problem 
in this area is the characterization of the
normality of the ideal of covers $I_c(G)$ of a graph $G$ in terms of
the combinatorics of $G$. This problem was solved for graphs with
independence number at most $2$ using
Hochster configurations \cite[Theorem~5.11]{icdual}. 
If $G$ is an odd cycle or a perfect graph, 
then $I_c(G)$ is normal, see \cite{Al-Ayyoub-Nasernejad-cover-ideals}
and \cite{perfect}, respectively. 

Let $\{t^{v_1},\ldots,t^{v_m}\}$ be the set of all
monomials $t_jt_j$ such that $\{t_i,t_j\}$ is an edge of $G$. 
The {\it edge subring\/} of $G$, denoted $K[G]$, is defined as 
$$K[G]:=K[t^{v_1},\ldots,t^{v_m}]\subset S.$$ 
\quad The edge subring was introduced and studied by Simis,
Vasconcelos and the second author in \cite{ITG}, where the
normality of $K[G]$ and $S[I(G)z]$ are related
\cite[Theorem~7.1]{ITG}, see Theorem~\ref{feb16-23}. The 
toric ideals of $K[G]$ and $S[I(G)z]$, see
Eq.~\eqref{toric-def} for the notion of toric ideal, is generated by
pure binomials that correspond to the even closed walks and the even cycles of the
graph $G$ \cite[Proposition~3.1, Theorem~3.1]{raei}. Reyes,
Tatakis and Thoma \cite{reyes-tatakis-thoma} characterized when a
binomial of the toric ideal of $K[G]$ is primitive, minimal,
indispensable, or fundamental, in terms of the even closed walks of
$G$.  Toric
ideals of edge subrings of uniform hypergraphs were studied more recently by
Petrovi\'{c} and 
Stasi \cite{petrovic-etal}. Standard references for binomial ideals
and Gr\"obner bases of toric ideals are the paper of Eisenbud and
Sturmfels \cite{EisStu}, the book of Sturmfels
\cite[Chapter~4]{Stur1}, and the book of Herzog, Hibi and Ohsugi
\cite{herzog-hibi-ohsugi}. 

As is seen below, the generators (Hochster configurations) of the
integral closure of the Rees algebra of the edge ideal of a graph
are related to the generators (bowties) of the integral closure of the
edge subring of the graph. 

To describe the integral closure of $K[G]$ using polyhedral geometry,
let ${\mathcal A}=\{v_1,\ldots,v_m\}$ be the set of exponent vectors
of the minimal generators of $I(G)$. By \cite[Theorem~9.1.1]{monalg-rev},
one has:
\[
\overline{K[G]}=
K[\{t^a\vert\, a\in\mathbb{R}_+{\mathcal A}\cap\mathbb{Z}{\mathcal A}\}],
\]
where $\mathbb{R}_+{\mathcal A}$ is the cone in $\mathbb{R}^s$ 
generated by ${\mathcal A}$ 
and $\mathbb{Z}{\mathcal A}$ is the subgroup of $\mathbb{Z}^s$ 
generated by ${\mathcal A}$. 

Hence, we can compute the monomial
generators of $\overline{K[G]}$ using
\textit{Normaliz} \cite{normaliz2}. A major result in the theory of
graph rings was the description of the integral closure of $K[G]$ using
the cycle structure of the graph. To describe this result, we need the
notion of a bowtie.

\begin{definition}
A \textit{bowtie} of $G$ is a connected subgraph $w$ of $G$ consisting of two 
induced odd cycles $C_1$, $C_2$ of $G$, with $|C_1\cap C_2|\leq 1$, that are joined by a path of $G$ intersecting each $C_i$ in
exactly one vertex when $C_1\cap C_2=\emptyset$. If $w$ is a bowtie, we 
 set $M_w:=(\prod_{t_i\in C_1}t_i)(\prod_{t_i\in C_2}t_i)$.  
\end{definition}

\begin{example} The graph depicted in Figure~\ref{figure-bowtie} is a
bowtie formed with 
a $5$-cycle and a $3$-cycle joined by a path of length $2$.
\begin{figure}[ht]
\setlength{\unitlength}{.03cm}
\thicklines
\begin{picture}(-30,80)(0,-40)
\put(0,0){\circle*{4.125}}
\put(-40,0){\circle*{4.125}}
\put(40,0){\circle*{4.125}}
\put(-60,30){\circle*{4.125}}
\put(-60,-30){\circle*{4.125}}
\put(-100,20){\circle*{4.125}}
\put(-100,-20){\circle*{4.125}}
\put(80,30){\circle*{4.125}}
\put(80,-30){\circle*{4.125}}
\put(-100,20){\line(0,-11){40}}
\put(-100,20){\line(4,1){40}}
\put(-100,-20){\line(4,-1){40}}
\put(-40,0){\line(-2,3){20}}
\put(-40,0){\line(-2,-3){20}}
\put(-40,0){\line(1,0){40}}
\put(1,0){\line(1,0){40}}
\put(40,0){\line(4,3){40}}
\put(40,0){\line(4,-3){40}}
\put(80,30){\line(0,-1){60}}
\put(-40,-11){$t_5$}
\put(-2,-11){$t_6$}
\put(30,-11){$t_7$}
\put(85,30){$t_8$}
\put(-55,30){$t_1$}
\put(-55,-30){$t_4$}
\put(85,-30){$t_9$}
\put(-115,20){$t_2$}
\put(-115,-20){$t_3$}
\end{picture}
\caption{A $5$-cycle and a $3$-cycle joined by a path of length $2$.}\label{figure-bowtie}
\vspace{-0.5cm}
\end{figure}
\end{example}
The following formula gives the normalization of an edge subring. 
\begin{theorem}{\rm(Hibi and Ohsugi \cite{hibi-ohsugi-normal}, Simis,
Vasconcelos, -\,, \cite{bowtie})}\label{bowtie-main-teo}
If $G$ is a graph, then the integral closure of $K[G]$ is given by
$$
\overline{K[G]}=K[\{t^{v_1},\ldots,t^{v_m}\}\cup \{M_{w}\mid w\mbox{ is a
bowtie of }G\}].
$$
\end{theorem}

There is a short proof of this result due to Johnson \cite{johnson}. 
Theorems~\ref{wolmer-ic-rees} and \ref{bowtie-main-teo} are valid 
for multigraphs by regarding loops as odd cycles
\cite{hibi-ohsugi-normal}, \cite[p.~450]{monalg-rev}.

A graph $G$ has the 
\textit{odd cycle condition} if every two vertex disjoint odd 
cycles of $G$ can be joined by at least one edge of $G$. We 
thank J. Brennan for pointing out that this
notion comes from graph theory \cite{FHMcA}. The odd cycle condition  was used to
classify the normality of edge subrings 
\cite{hibi-ohsugi-normal,ITG,bowtie} and the elementary integral
vectors of the kernel of the incidence matrix of a graph 
\cite[Proposition~4.2]{raei}.

\begin{corollary}\label{bowtie-criterion}
Let $G$ be a graph. If $G$ satisfies the odd cycle condition, then 
$K[G]$ is normal. The converse holds if $G$ is a connected graph.
\end{corollary}

\begin{example} If $G$ is the union of two vertex disjoint triangles,
then $K[G]$ is normal and $G$ does not satisfies the odd cycle
condition.
\end{example}

The {\em cone\/} $C(G)$, over the graph $G$, is
obtained by adding a new vertex $x$ to $G$ and joining every 
vertex of $G$ to $x$. Let  
\[
{\mathcal R}(I(G))=K[\{t_1,\ldots,t_s,
t^{v_j}z\vert\, 1\leq j\leq m\}]
\]
be the Rees algebra of the ideal $I(G)=(t^{v_1},\ldots,t^{v_m})$. As
before, we assume that 
$t^{v_1},\ldots, t^{v_m}$ are the monomials in the $t_i$'s corresponding to
the edges of $G$. Let 
\[
K[C({ G})]=
K[\{t_ix,t^{v_j}\vert\, 1\leq i\leq s,\,1\leq j\leq m \}]
\]
be the edge subring of the cone over $G$. The following result shows
that Rees algebras of edge ideals are isomorphic to edge subrings of
graphs. In retrospect this result together with
Theorem~\ref{bowtie-main-teo} are the key to finding the 
integral closure of Rees algebras of edge ideals. 

\begin{proposition}\cite[Remark~2.5]{normals}
${\mathcal R}(I({ G})){\simeq}K[C({ G})]$.
\end{proposition}

\begin{proof} As these two algebras are
integral domains of the same dimension (see \cite{E-H} and
\cite[Propositon~3.2]{raei}) it follows that there is an 
isomorphism 
\[\varphi\colon {\mathcal R}(I({ G}))\longrightarrow K[C({
G})] \mbox{,\ induced by }\varphi(t_i)=t_ix\mbox{ and }
\varphi(t^{v_j}z)=t^{v_j},
 \]
and the proof is complete.
\end{proof}

Let $K[u_1,\ldots,u_m]$ be a new polynomial ring over the field $K$ obtained by considering
one variable $u_i$ for each monomial $t^{v_i}$. The \textit{toric
ideal} of $K[G]$, denoted 
$P(G)$, is the kernel of the epimorphism of $K$-algebras:
\begin{equation}\label{toric-def}
K[u_1,\ldots,u_m]\longrightarrow K[G]\ \mbox{ induced by }\
u_i\mapsto t^{v_i}.
\end{equation}

The incidence matrix $A$ of $G$, with columns $v_1,\ldots,v_m$,  
determines a linear map $A\colon
\mathbb{Q}^m\rightarrow\mathbb{Q}^s$. Let $V=\ker(A)$ be the kernel of $A$. A  
\textit{circuit} or \textit{elementary integral vector} of $V$ is a 
non-zero integral vector $\alpha$ in $V$ whose support is minimal with
respect to inclusion and such that the non-zero entries of $\alpha$ are
relatively prime (see \cite[Section~22]{Rock1},
\cite[Section~10.3]{monalg-rev}). 

\begin{definition}
If $\alpha$ is a circuit of $V$, $\alpha$ can be written uniquely as 
$\alpha=\alpha_+-\alpha_-$, where $\alpha_+$ and $\alpha_-$ are two nonnegative vectors 
with disjoint support, we call the 
binomial $t^{\alpha_+}-t^{\alpha_-}$ a {\it circuit\/}
 of $P(G)$, and we call $t^{\alpha_+}$ and $t^{\alpha_-}$ the
 \textit{terms} of
 $t^{\alpha_+}-t^{\alpha_-}$. 
\end{definition}

\begin{theorem}\cite[Theorem~3.2]{circuits}\label{toric-gen-by-circuits-1} Let $G$ be a graph. 
If $K[G]$ is normal, then $P(G)$ is generated by circuits
with a square-free term.
\end{theorem}

We thank A. Thoma for pointing out that the converse of this result
is not true and for providing a counterexample (see
\cite[Exercise~10.3.25]{monalg-rev}).

Let $F=\{t^{v_1},\ldots,t^{v_m}\}$ be a finite set of monomials of
$S$. The subring $K[F]\subset S$ is
called \textit{homogeneous} if there 
is $x_0\in{\mathbb Q}^s$ such that $\langle v_i,x_0 \rangle=1\ \mbox{
for }\ i=1,\ldots,m$.

The following result complements
Theorem~\ref{toric-gen-by-circuits-1}.

\begin{theorem}{\rm(Simis, -\,, \cite[Proposition~4.1]{birational})}\label{aug13-01} Let ${\mathcal
B}$ be a finite set of   
binomials in the toric ideal $P(F)$ of $K[F]$. 
If $K[F]$ 
is a normal homogeneous subring and $P(F)$ is minimally generated by ${\mathcal B}$, 
then every element of ${\mathcal B}$ has a square-free term.
\end{theorem}

\begin{definition}\label{circuitsG} A subgraph $H$ of $G$ is called a {\it
circuit\/} of $G$ if $H$ has one 
of the following forms:\vspace{-1mm}
\begin{enumerate}
\item[(a)] $H$ is an even cycle.
\item[(b)] $H$ consists of two odd cycles intersecting in
exactly one vertex. 
\item[(c)] $H$ consists of two vertex disjoint odd cycles joined
by a path.
\end{enumerate}
\end{definition}

The circuits of $G$ are in one-to-one correspondence 
with the circuits of $P(G)$ \cite{raei}. Toric ideals of edge subrings of oriented graphs were
studied in \cite{ringraphs}. In this
case, the toric ideal is generated by circuits and the circuits correspond to the
cycles of the graph \cite[Proposition~4.3]{ringraphs}.

It is worth noticing that the cycles of a graph $G$ are the circuits of
the graphic matroid $M(G)$ of $G$ \cite{oxley,welsh} and that the 
circuits of a graph $G$ (Definition~\ref{circuitsG}) are the circuits
of the even cycle matroid $M(G_-)$ of $G$
\cite{Zaslavsky-signed-graphs}. 

The circuits of the matroids $M(G)$, $M(G_-)$ and 
those of their dual matroids occur also in coding theory 
\cite{Dankelmann-Key-Rodrigues,sole-zaslavsky},
\cite[Corollary~3.13]{weights-matroid}, 
and in matroid theory \cite{oxley,Simoes-Pereira,Zaslavsky-signed-graphs,Zaslavsky}.

\section{Normality of monomial ideals}\label{normality-section}

Let $S=K[t_1,\ldots,t_s]=\bigoplus_{i=0}^\infty S_i$ be a polynomial ring 
over a filed $K$ with the standard grading, let $I$ be a monomial
ideal of $S$, and let
$\mathcal{G}(I):=\{t^{v_1},\ldots,t^{v_m}\}$ be the minimal set of
generators of $I$. The \textit{incidence matrix} of $I$, denoted by $A$,
is the $s\times m$ matrix with column vectors $v_1,\ldots,v_m$. The \textit{Newton
polyhedron} of $I$, denoted ${\rm NP}(I)$, is the 
rational polyhedron 
\begin{equation*}
{\rm NP}(I)=\mathbb{R}_+^s+{\rm 
conv}(v_1,\ldots,v_m),
\end{equation*}
where $\mathbb{R}_+=\{\lambda\in\mathbb{R}\mid \lambda\geq
0\}$. The \textit{integral closure} of
 $I^n$ can be described as:
\begin{equation}\label{jun21-21}
\overline{I^n}=
(\{t^a\mid a/n\in{\rm NP}(I)\}),
\end{equation}
see  \cite[Theorem~3.1, Proposition~3.5]{reesclu}. 
There is a well-known characterization of the normality of
$I$ that comes from 
integer programming using Hilbert bases (Proposition~\ref{hilb-basis-charw}). 
The \textit{Rees algebra} of $I$ is the monomial subring  
$$S[Iz]=K[t_1,\ldots,t_s,t^{v_1}z,\ldots,t^{v_m}z],$$
where 
$z=t_{s+1}$ is a new variable. 
Following \cite{Lisboar}, we define the \textit{Rees cone} of the
ideal $I$, denoted ${\rm
RC}(I)$, as the rational cone 
\begin{equation}\label{rees-cone-eq}
{\rm RC}(I):=\mathbb{R}_+\mathcal{A}'
\end{equation}
generated by $\mathcal{A}':=\{e_1,\ldots,e_s,(v_1,1),\ldots,(v_m,1)\}$, where $e_i$
is the $i$-th unit vector in $\mathbb{R}^{s+1}$. The set
$\mathcal{A}'$ is called a \textit{Hilbert basis} if 
$\mathbb{Z}^{s+1}\cap \mathbb{R}_+{\mathcal A}'=\mathbb{N}{\mathcal
A}'$.  

\begin{proposition}\cite[p.~34]{icdual}\label{hilb-basis-charw}
The ideal $I$ is normal if and only if $\mathcal{A}'$ is a
Hilbert basis. 
\end{proposition}

One of the earlier works introducing linear programming (LP) 
methods to prove the normality of monomial ideals and monomial
subrings was the paper of Bonanzinga, Escobar and the second author
\cite{unimod}. Vasconcelos et al. gave the following linear programming membership
test that determines whether or not a given monomial lies in the integral
closure of a monomial ideal (cf. \cite[Proposition~1.1]{Ha-Trung-19}).

\begin{proposition}{\rm(Membership test
\cite[Proposition~3.5]{mc8})}\label{lp-membrship-test}  
Let $I=({t}^{v_1},\ldots,{t}^{v_m})$ be a monomial ideal of $S$ and let 
$A$ be the $s\times m$ matrix with columns 
vectors $v_1,\ldots,v_m$. Then, a 
monomial ${t}^\alpha$ lies in the integral closure of 
$I^n$, $n\geq 1$, $\alpha=(\alpha_1,\ldots,\alpha_s)$, if and only if the linear program:
\begin{enumerate}
\item[\ ] \ \ {\rm maximize} $y_1+\cdots+y_m$

\item[\ ] \ \ {\rm subject to}

\item[\ ] \ \ $Ay\leq \alpha$; $y\geq 0$
\end{enumerate}
has an optimal value greater than or equal to $n$, which 
is attained at a vertex of the rational polyhedron 
$\mathcal{P}_\alpha=\{y\in\mathbb{R}^m\, \vert\, Ay\leq \alpha;\, y\geq 0\}$. 
\end{proposition}

By linear programming duality 
\cite[Theorem~1.1.56]{monalg-rev}, one 
can also use the dual problem
\begin{enumerate}
\item[] {\rm minimize} $\alpha_1x_1+\cdots+\alpha_sx_s$

\item[] {\rm subject to}

\item[] $xA\geq 1$; $x\geq 0$, 
\end{enumerate}
where $1=(1,\ldots,1)$, 
to check whether or not $t^\alpha$ is in $\overline{I^n}$. In this
case, one has a fixed 
polyhedron 
$$\mathcal{Q}(A):=\{x\in\mathbb{R}^s\, \vert\, xA\geq{1};\, x\geq 0\}$$
that can be used to test membership 
of any monomial ${t}^\alpha$, while in the primal 
problem the polyhedron $\mathcal{P}_\alpha$ depends on $\alpha$. 

Given a vector $c=(c_1,\ldots,c_p)$ in $\mathbb{R}^p$, we set
$|c|:=\sum_{i=1}^pc_i$ and denote the integral part of $c$ by $\lfloor
c\rfloor$ and the ceiling of $c$ by $\lceil
c\rceil$. We denote the nonnegative rational numbers by $\mathbb{Q}_+$
and $\langle\ ,\, \rangle$ 
denotes the standard 
inner product. 

The following recent proposition gives 
a linear algebra membership test that complements  
the previous result (Proposition~\ref{lp-membrship-test}). 

\begin{proposition}\cite[Proposition~3.3]{icdual}\label{membership-test-n}
Let $I=(t^{v_1},\ldots,t^{v_m})$ be a monomial
ideal of $S$, let $A$ be its incidence matrix, and let $t^\alpha$ be a
monomial in $S$. The following are equivalent.
\begin{enumerate}
\item[(a)] $t^\alpha\in\overline{I^n}$, $n\geq 1$.
\item[(b)] $A\lambda\leq\alpha$ for some $\lambda\in\mathbb{Q}_+^m$
with $|\lambda|\geq n$. 
\item[(c)] $\max\{\langle y, 1\rangle\mid 
y\geq 0;\, Ay\leq\alpha \}=\min\{\langle\alpha,x\rangle\mid x\geq 0;\,
xA\geq
1\}\geq n$.
\end{enumerate}
\end{proposition}

As a consequence, we obtain a minimal generators test for the integral
closure of the powers of a monomial ideal.

\begin{proposition}\cite{icdual}\label{mg-test}
Let $I$ be a monomial
ideal of $S$ and let $A$ be its incidence matrix. A monomial $t^\alpha\in
S$ is a minimal generator of
$\overline{I^n}$ if and only if the following two conditions hold:
\begin{align}
&\max\{\langle y, 1\rangle\mid 
y\geq 0;\, Ay\leq\alpha \}=\min\{\langle\alpha,x\rangle\mid x\geq 0;\,
xA\geq
1\}\geq n;\\
&\max\{\langle y, 1\rangle\mid 
y\geq 0;\, Ay\leq\alpha-e_i\}=\min\{\langle\alpha-e_i,x\rangle\mid x\geq 0;\,
xA\geq
1\}<n\\
&\mbox{for each }e_i\mbox{ for which }\alpha-e_i\geq 0.\nonumber 
\end{align}
\end{proposition}

The following normality criterion can be used to 
show a classification of the normality of a monomial ideal in terms
of  integer programming notions, see Theorem~\ref{crit-rounding-normali}
below. 

\begin{proposition}\cite[Proposition~3.2]{icdual}\label{normality-criterion-ip} 
Let $I$ be a monomial
ideal of $S$ and let $A$ be its incidence matrix. The following conditions are equivalent.
\begin{enumerate}
\item[(a)] $I$ is a normal ideal.
\item[(b)] For each pair of vectors $\alpha\in\mathbb{N}^s$ and
$\lambda\in\mathbb{Q}_+^m$ such that $A\lambda\leq\alpha$, there is 
$\beta\in\mathbb{N}^m$ satisfying $A\beta\leq\alpha$ and
$|\lambda|=|\beta|+\epsilon$ with $0\leq\epsilon<1$.
\end{enumerate} 
\end{proposition}

The normality of $I$ is also related to integer rounding
properties \cite[Corollary~2.5]{poset}. 
The linear system $x\geq 0; xA\geq{1}$ 
has the {\it integer rounding property\/} if 
\begin{equation}\label{irp-eq1}
{\rm max}\{\langle y,{1}\rangle\mid y\in\mathbb{N}^m;\,   
Ay\leq \alpha\} 
=\lfloor{\rm max}\{\langle y,{1}\rangle\mid y\geq 0;\, 
Ay\leq \alpha\}\rfloor
\end{equation}
for each integral vector $\alpha$ for which the 
right-hand side is finite. The linear system $x\geq 0; xA\leq{1}$ 
has the \textit{integer rounding property} if 
\begin{equation}\label{irp-eq2}
\lceil{\rm min}\{\langle y,{1}\rangle\mid y\geq 0;\, Ay\geq \alpha \}\rceil
={\rm min}\{\langle y,{1}\rangle\mid y\in\mathbb{N}^m;\, Ay\geq \alpha\}
\end{equation}
for each integral vector $\alpha$ for which the left hand side is
finite. Integer rounding property are well studied, 
see \cite{baum-trotter,ainv}, \cite[Chapter~22]{Schr},
\cite[Chapter~5]{Schr2}, and references therein. 

The following duality theorem relates the integer rounding property of two
types of systems, one defined by a $0$-$1$ matrix $A$, and the other by its
dual matrix $A^*$ obtained from $A$ by replacing its entries equal to
$0$ by $1$ and 
its entries equal to $1$ by $0$.

\begin{theorem}{\rm(Brennan, Dupont, -\,, \cite[Theorem~2.11]{ainv})}\label{duality-irp} Let
$A=(a_{i,j})$ be the incidence matrix of a squarefree monomial ideal
$I$ 
and let $A^*=(a_{i,j}^*)$ be the
matrix whose $(i,j)$-entry is $a_{i,j}^*=1-a_{i,j}$. Then, 
the linear system $x\geq 0;xA\geq{1}$
has the integer rounding property if and 
only if the dual linear system $x\geq 0;xA^*\leq{1}$ has the integer rounding property. 
\end{theorem}

The following theorem was observed by N. V. Trung if $A$ is the
incidence matrix of a squarefree monomial ideal, and it 
was shown in \cite{poset} using the theory of
blocking and antiblocking polyhedra \cite{baum-trotter},
\cite[p.~82]{Schr2}. Using Proposition~\ref{normality-criterion-ip}, 
one can give a short proof of this theorem.

\begin{theorem}{\rm(Dupont, -\,,  \cite[Corollary~2.5]{poset})}\label{crit-rounding-normali} 
A  monomial
ideal $I$ with incidence matrix $A$ is normal if and only if 
the system $x\geq 0;xA\geq{1}$ has the integer rounding
property. 
\end{theorem}

The following result of Simis, Vasconcelos and the second author 
gives a method to descend the normality of
Rees algebras to the normality of monomial subrings.

\begin{theorem}{\rm(Descent of normality criterion \cite[Theorem~7.1]{ITG})}\label{feb16-23} 
Let $I=(t^{v_1},\ldots,t^{v_m})$ be an ideal of $S$ 
generated by monomials of degree $d$. If the Rees algebra
$S[Iz]$ of $I$ is normal, then the monomial subring
$K[t^{v_1},\ldots,t^{v_m}]$ is also normal.  
\end{theorem}

The converse is true if $I$ is the edge ideal of a connected graph
(see Theorem~\ref{roundup-iff-rounddown} below),  
but it is false in general as the following example shows. 

\begin{example} If $I$ is generated by 
$F=\{t_1t_2,t_2t_3,t_1t_3,t_4t_5,t_5t_6,t_4t_6\}$, that is, 
$I$ is the edge ideal of two vertex disjoint triangles, then $K[F]$ is
normal but the Rees algebra
of $I$ is not. 
\end{example}

The following is a partial converse of Theorem~\ref{feb16-23} that
answers a question of \cite{ITG}, \cite[p.~71]{aportaciones-ma}. 
For another sufficient condition for the converse to hold see
\cite[Theorem~7.6]{ITG}.   

\begin{theorem}\cite[Theorem~3.3]{roundp}\label{roundup-iff-rounddown} Let $G$ be a connected
graph, let $A$ be the incidence matrix of $G$, and let
$v_1,\ldots,v_m$ be the column vectors of $A$. The following
conditions are equivalent:
\begin{enumerate} 
\item[(a)] $x\geq 0;xA\leq{1}$ is a system with the
integer rounding property. 
\item[(b)] $S[Iz]$ is a normal domain, where $I=I(G)$ is the edge
ideal 
of $G$.
\item[(c)] $K[t^{v_1}z,\ldots,t^{v_m}z]\simeq
K[t^{v_1},\ldots,t^{v_m}]$ 
is normal. 
\end{enumerate}
\end{theorem}

We now discuss a conjecture of Simis on the normality of ideals arising 
from doubly stochastic matrices that is related to Cremona
monomial maps. 

Let $I=(t^{v_1},\ldots,t^{v_s})$ be a monomial ideal of
$S=K[t_1,\ldots,t_s]$ 
such that its incidence matrix
$A=(a_{i,j})$ is a non-singular matrix of order $s\times s$. We assume
that the set of monomials ${F}:=\{{t}^{v_1},\ldots,{t}^{v_s}\}\subset S$ have no non-trivial
 common factor, and that every $t_i$ divides at least one member of
 $F$. The matrix $A$ is 
called {\it doubly stochastic\/} of degree $d$ if 
$$
\sum_{k=1}^sa_{k,i}=\sum_{k=1}^sa_{j,k}=d\geq 2\ \mbox{ for all }\
1\leq i,j\leq s.
$$
\quad If $A$ is a $d$-stochastic matrix by columns, that is, 
$\sum_{k=1}^sa_{k,i}=d\geq 2$ for all $i$ and $S[Iz]$ is normal, 
then $\det(A)=\pm d$ \cite[Proposition~12.8.1]{monalg-rev}. The
following 
statement conjectures that a partial converse
holds.   

\begin{conjecture}{\rm(Simis)}\label{conje1} If $A$ is a doubly stochastic 
matrix of degree $d$ with entries in $\{0,1\}$ and 
$\det(A)=\pm d$, then the Rees
algebra $S[Iz]$ is normal 
\end{conjecture}

The following result gives some support for this conjecture.

\begin{proposition}{\rm(\cite{birational}, \cite[Proposition
12.8.4]{monalg-rev})}\label{may29-04}
If $A=(a_{i,j})$ is a $2$-stochastic matrix by columns with entries in
$\{0,1\}$ and $\det(A)=\pm 2$, then  $S[Iz]$ is normal. 
\end{proposition}

The set $F$ defines a rational (monomial) map 
$\mathbb{P}^{s-1}\dasharrow \mathbb{P}^{s-1}$ denoted again by $F$ and written as a tuple
$F=({t}^{v_1},\ldots, {t}^{v_s})$. $F$ is called a {\em 
Cremona map\/} if it admits an
inverse rational map with source $\mathbb{P}^{s-1}$. 
A rational monomial map $F$ is defined 
everywhere if and only if the defining monomials are pure powers of the
variables, in which case it is a Cremona map if and only if 
$F=(t_{\sigma(1)},\ldots,t_{\sigma(s)})$ for
some permutation $\sigma$.  

\begin{proposition}{\rm(Simis, -\,,
\cite{birational})}\label{march27-11} $F$ defines a Cremona 
map if and only if $\det(A)=\pm d$.
\end{proposition}

Thus, Conjecture~\ref{conje1} has the following reformulation: 
\begin{conjecture} $F\colon \mathbb{P}^{s-1}\dasharrow\mathbb{P}^{s-1}$ is
a Cremona map if and only 
if $S[Iz]$ is normal.
\end{conjecture}

Cremona monomial maps were studied by Simis and the second author 
in \cite{cremona} using linear algebra, lattice theory and linear optimization methods. 
Cremona maps defined by monomials of degree $d=2$ were thoroughly
analyzed and classified via 
integer arithmetic and graph combinatorics by Costa and Simis \cite{costa-simis}. 
The recent book of Simis and Ramos is an excellent reference for 
Cremona transformations \cite[Chapters~6 and 7]{Aron-bookII}. 

\section{Normality and Ehrhart rings: $a$-invariant and
regularity}\label{regularity-section}

Let $S=K[t_1,\ldots,t_s]$ be a 
polynomial ring over a field $K$ and let $v_1,\ldots,v_m$ be points in
$\mathbb{N}^s$. The dimension of the lattice 
polytope $\mathcal{P}:={\rm
conv}(v_1,\ldots,v_m)\subset\mathbb{R}^s$ is denoted by $d$. The \textit{Ehrhart
function} and \textit{Ehrhart
series} of $\mathcal{P}$, denoted $E_\mathcal{P}$
and $F_\mathcal{P}$, respectively \cite{beck-robins,BG-book}, 
are given by 
$$
E_\mathcal{P}(n):=|n\mathcal{P}\cap\mathbb{Z}^s|,\ n\in\mathbb{N},\ \mbox{
and }\ F_\mathcal{P}(x):=\sum_{n=0}^\infty
E_\mathcal{P}(n)x^n,
$$
and the \textit{Ehrhart ring} of $\mathcal{P}$, denoted $A(\mathcal{P})$, 
is the monomial subring of $S[z]$ given by 
\begin{equation*}
A(\mathcal{P}):=K[\{t^az^n\mid a\in n\mathcal{P}\cap\mathbb{Z}^s\}]\subset S[z],
\end{equation*}
where $z$ is a new variable \cite{beck-robins,BG-book}. This ring is graded by 
$$
A(\mathcal{P})=\bigoplus_{n=0}^\infty A(\mathcal{P})_n,
$$
where $t^az^n\in A(\mathcal{P})_n$ if and only if
$a\in n\mathcal{P}\cap\mathbb{Z}^s$. Note that $E_\mathcal{P}$ and
$F_\mathcal{P}$ are the Hilbert function and the Hilbert series of
$A(P)$, respectively. The function
$E_\mathcal{P}$ is a
polynomial of degree $d$ whose leading coefficient is the relative
volume ${\rm vol}(\mathcal{P})$ of $\mathcal{P}$ \cite{beck-robins},
and the generating 
function $F_\mathcal{P}$ of 
$E_\mathcal{P}$ is a rational function of the form
$$
F_\mathcal{P}(x)=\frac{h_\mathcal{P}(x)}{(1-x)^{d+1}},
$$
where $h_\mathcal{P}(x)$ is a polynomial with integer coefficients of degree at
most $d$ \cite{Sta1}, that is called the $h$-\textit{polynomial} of
$\mathcal{P}$. The vector
formed with the coefficients of $h_\mathcal{P}(x)$ is called the
$h$-\textit{vector} of 
$\mathcal{P}$ and is denoted by
$h(\mathcal{P})$. Stanley's positivity theorem shows that
$h(\mathcal{P})$ has non-negative integer entries, see
Theorem~\ref{Stanley-nonneg-mono}(a) below.      

Let $\mathbb{R}\subset K$ be a field extension. Given $f\in S$ such
that $f(\mathbb{Z}^s)\subset\mathbb{Z}$, let 
$f\colon K^s\rightarrow K$, $a\mapsto f(a)$, be its underlying
\textit{weight function}. The \textit{weighted Ehrhart
function} and \textit{weighted Ehrhart
series} of the lattice polytope $\mathcal{P}$ are:
$$
E^f_\mathcal{P}(n):=\sum_{\scriptstyle a\in
n\mathcal{P}\cap\mathbb{Z}^s}f(a),\ \forall\, n\in\mathbb{N},\ \mbox{
and }\ F_\mathcal{P}^f(x):=\sum_{n=0}^\infty
E_\mathcal{P}^f(n)x^n.
$$
\quad If $f=1$, then $E^f_\mathcal{P}$, $F^f_\mathcal{P}$ are 
$E_\mathcal{P}$, $F_\mathcal{P}$, respectively. 
The weighted Ehrhart theory was developed in 
several papers, see
\cite{baldoni-etal-1,baldoni-etal,Brion-Vergne,bruns-ichim-soger,Bruns-Soger} 
and references therein. We thank Jes\'us De Loera for
introducing us to this theory while he was on sabbatical at
Cinvestav in the fall of 2022. The recent paper \cite{ehrhartw}
studies Ehrhart functions and series of weighted lattice points. 

\begin{proposition}\cite[Proposition~4.1]{Brion-Vergne}\label{weighted-degree} 
Let $0\neq f\in K[t_1,\ldots,t_s]$ be a homogeneous polynomial of
degree $p$. If the interior
$\mathcal{P}^{\rm o}$ of $\mathcal{P}$ 
is nonempty, $K=\mathbb{R}$, and $f\geq 0$ on $\mathcal{P}$, then 
$E_\mathcal{P}^f$ is a polynomial of degree $s+p$. 
\end{proposition}

It is known that in
Proposition~\ref{weighted-degree}, the leading coefficient of
the polynomial $E_\mathcal{P}^f$ is equal to $\int_\mathcal{P}f$. This fact appears 
in \cite[p.~437]{baldoni-etal} and \cite[Proposition~5]{Bruns-Soger}. Integrals of the type 
$\int_\mathcal{P}f$ with $f$ a polynomial
and $\mathcal{P}$ a rational polytope were studied in
\cite{baldoni-etal-1,barvinok,barvinok1,Bruns-Soger,lasserre}.
Algorithms and implementation to compute this integral were developed
independently in \textit{LattE integrale} \cite{latte-integrale} 
and \textit{NmzIntegrate} \cite{Bruns-Soger1}. One can use \textit{NmzIntegrate} \cite{Bruns-Soger,Bruns-Soger1}
or  \textit{Normaliz} \cite{normaliz2} to compute the polynomial
 $E_\mathcal{P}^f$ and the rational function
$F_\mathcal{P}^f$.

\begin{proposition}\cite[Theorem~9.3.6]{monalg-rev}\label{ehrhart-normal} 
The Ehrhart ring $A(\mathcal{P})$ is a normal 
finitely generated graded $K$-algebra. 
\end{proposition}

The following result is a generalization 
of the descent of normality criterion (Theorem~\ref{feb16-23}). 

\begin{theorem}{\rm(Generalized descent of normality criterion 
\cite[Theorem 3.15]{ehrhart})}\label{mar3-01}  
Let 
$I$ be a monomial ideal of $S$ generated by $\{t^{v_i}\}_{i=1}^m$ and let
$\mathcal{P}={\rm conv}(v_1,\ldots,v_m)$ be the Newton polytope of $I$. If the 
Rees algebra $S[Iz]$ of $I$ is normal and 
$v_1,\ldots,v_m$ 
lie on a hyperplane 
$$
b_1x_1+\cdots+b_sx_s=1\ \ \ \ \ (b_i>0\, \, \forall i),
$$
then $K[t^{v_1}z,\ldots,t^{v_m}z]=A(\mathcal{P})$. In particular, 
$K[t^{v_1}z,\ldots,t^{v_m}z]$ is normal.
\end{theorem}

\begin{theorem}\label{Stanley-nonneg-mono} Let $\mathcal{P}$ and $\mathcal{P}_1$ be lattice
polytopes in $\mathbb{R}^s$. The following hold.
\begin{enumerate}
\item[(a)] $($Stanley's positivity theorem
\cite[Theorem~2.1]{Stanley-nonneg-h-vector}$)$ The $h$-vector 
$h(\mathcal{P})$ of $\mathcal{P}$ has non-negative integer entries.  
\item[(b)] $($Stanley's  monotonicity theorem
\cite[Theorem~3.3]{stanley-mono}, \cite[Theorem~3.3]{beck-deco}$)$ If 
$\mathcal{P}\subset\mathcal{P}_1$, then their $h$-vectors satisfy
$h(\mathcal{P})\leq h(\mathcal{P}_1)$ componentwise.
 \end{enumerate}
\end{theorem}

Since Ehrhart rings are normal, the following fundamental result of Hochster can be used to show
Stanley's positivity theorem that the
Ehrhart series of a lattice polytope has a non-negative 
integral $h$-vector (Theorem~\ref{Stanley-nonneg-mono}(a)). It
can also be used to compute the $a$-invariant and regularity of a
positively graded normal monomial subring
\cite[Lemma~9.1.7]{monalg-rev}.

\begin{theorem}{\rm(Hochster \cite{Ho1},
\cite[Theorem~6.3.5]{BHer})}\label{hoch-theo} 
If $F=\{t^{v_1},\ldots,t^{v_m}\}$ is a finite set of monomials of $S$
and $K[F]$ is normal, then $K[F]$ is
Cohen--Macaulay.
\end{theorem}

Let $S_{s,k}$ be the $k$-th \textit{squarefree Veronese subring} of $S$ given by 
$$S_{s,k}:=K[t_{i_1}\cdots t_{i_k}\mid 1\leq i_1<\cdots<i_k\leq
s]\subset S.$$
\quad This ring is normal \cite{normals}, and it is a
standard graded $K$-algebra with the normalized grading induced by
setting $\deg(t_{i_1}\cdots t_{i_k})=1$. The following
result, describing the generators of the canonical module of
$S_{s,k}$, 
was shown using Hochster's theorem and the Danilov-Stanley formula for the
canonical module of normal monomial subrings \cite{BG-book}, 
\cite[Theorem~6.3.5]{BHer}, \cite{Dan}. 

\begin{theorem}{\rm(Bruns, Vasconcelos, -\,, \cite[Theorem~2.6]{BVV})}\label{main-t}
Let $\omega_{S_{s,k}}$ be the canonical module of $S_{s,k}$ and let
$\mathfrak B$ be the set  
of monomials $t_1^{a_1}\cdots t_s^{a_s}$ satisfying the following
conditions:
\begin{enumerate}
\item[(a)] 
$a_i\geq 1$ and $(k-1)a_i\leq -1+\sum_{j\neq i}a_j$, for
all $i$. 
\item[(b)] $\sum_{i=1}^sa_i
\equiv\, 0\, {\rm mod}\, (k)$.
\item[(c)] $|\{\, i\, |\, a_i\geq 2\}|\leq k-1$.
\end{enumerate}
\noindent If $s\geq 2k\geq 4$, then $\mathfrak B$ is a 
generating set for
$\omega_{S_{s,k}}$. 
\end{theorem}

The previous result can be used to compute the \textit{type} of $S_{s,k}$, that is, 
the minimal number of generators of the canonical module
$\omega_{S_{s,k}}$ of $S_{s,k}$ \cite{BVV}. Recall that the $a$-\textit{invariant} of $S_{s,k}$, denoted
$a(S_{s,k})$, is the 
degree as a rational function of the Hilbert series of $S_{s,k}$.

\begin{proposition}{\rm(Bruns, Vasconcelos, -\,,
\cite[Corollary~2.12]{BVV})}\label{a-invariant}
If $s\geq 2k\geq 4$, then 
$$a(S_{s,k})=-\left\lceil\frac{s}{k}\right\rceil\ \mbox{ and }\ {\rm
reg}(S_{s,k})=s-\left\lceil\frac{s}{k}\right\rceil.$$
\end{proposition}

To treat the case $s<2k$ one uses duality. Given as integer 
$1\leq k \leq s-1$, there is a graded isomorphism of $K$-algebras of degree
zero: 
\[
\rho:\, S_{s,k}\longrightarrow S_{s,s-k},\, \mbox{ induced by } 
\rho(t_{i_1}\cdots t_{i_k})=t_{j_1}\cdots t_{j_{s-k}}, 
\]
where
$\{j_1,\ldots,j_{s-k}\}=\{1,\ldots,s\}\setminus\{i_1,\ldots,i_k\}$.
Thus if $s\leq 2k$, then 
\[
a(S_{s,k})=a(S_{s,s-k})=-\left\lceil \frac{s}{s-k}\right\rceil.
\]
\quad The next classification of the Gorenstein property 
was shown independently by De Negri and Hibi \cite{DeNegri} using
different methods. 

\begin{proposition}{\rm(Bruns, Vasconcelos, -\,,
\cite[Corollary~2.14]{BVV})} The $k$-th 
squarefree Veronese subring $S_{s,k}$ is a Gorenstein ring if and
only if $k\in\{1,s-1\}$ or $s=2k$.  
\end{proposition}

The \textit{Veronese subring} $S^{(k)}$ is the monomial subring 
of $S$ generated by all monomials of $S$ of degree $k$. By a result of Goto and
Matsuoka \cite{Got,Matsuo}, the ring $S^{(k)}$ is 
Gorenstein if and only if $k$ divides $s$. The ring $S^{(k)}$ is
normal and its $a$-invariant is
$a(S^{(k)})=-\left\lceil\frac{s}{k}\right\rceil$. There is a recent
formula of Lin and Shen for the regularity of the
monomial subring of an ideal of Veronese type, 
see \cite[Theorem~5.6]{Lin-Shen} and its proof. There are other normal monomial subrings, coming from integer
rounding properties, where there are formulas for the canonical
module and the $a$-invariant \cite[Section~4]{ainv}, 
\cite[Theorem~4.2]{roundp}.

\begin{theorem}{\rm(Bruns, Vasconcelos, -\,, \cite[Corollaries~3.6
and 3.8]{BVV})}\label{main-in}
Let $F$ be a finite set of squarefree monomials of
degree $k$ in $S$ such that $\dim(K[F])=s$. The following
hold.
\begin{itemize}
\item[\rm(i)] $a(\overline{K[F]})\leq a(S_{s,k})$.
\item[\rm(ii)] $a(\overline{K[F]})\leq -\left\lceil
\frac{s}{k}\right\rceil$ if $s\geq 2k$ and $a(\overline{K[F]})\leq 
-\left\lceil\frac{s}{s-k}\right\rceil$ if $s\leq 2k,\ s\neq k$.
\item[\rm(iii)] If $s\geq 2k\geq 4$, then
$\overline{K[F]}$ is generated as a $K$-algebra by elements of
normalized degree less than or equal to  $s-\left\lceil
\frac{s}{k}\right\rceil$.
\end{itemize}
\end{theorem}

The max-flow min-cut property for clutters is defined in
Section~\ref{blowuprings-section}. The following theorem bounds the 
regularity of certain Ehrhart rings. For use below recall that $\alpha_0(\mathcal{C})$ denotes
the covering number of a clutter $\mathcal{C}$ which is also the
height of $I(\mathcal{C})$. A clutter with all its edges of the same
cardinality $k$ is called $k$-uniform.

\begin{theorem}\cite[Theorem~14.4.19]{monalg-rev}\label{ehrhart-unmixed} 
Let $\mathcal{C}$ be a
$k$-uniform unmixed clutter with the max-flow min-cut property, let 
$I=(t^{v_1},\ldots,t^{v_m})$
be its edge ideal, and let $\mathcal{P}={\rm conv}(v_1,\ldots,v_m)$. 
Then, $K[t^{v_1}z,\ldots,t^{v_m}z]=A(\mathcal{P})$, 
the $a$-invariant of $A(\mathcal{P})$ is bounded from above by
$-\alpha_0(\mathcal{C})$ 
and 
$${\rm reg}(A(\mathcal{P}))\leq (k-1)(\alpha_0(\mathcal{C})-1).
$$
\end{theorem}

The $a$-invariant, the regularity, and the 
depth are closely related.

\begin{theorem}{\rm\cite[Corollary~B.4.1]{Vas1}}\label{reg-cm} If $I$
is a graded ideal of $S$, then 
$$a(S/I)\leq{\rm 
reg}(S/I)-{\rm depth}(S/I),
$$
with equality if $S/I$ is Cohen--Macaulay.
\end{theorem}

A classical result of Herzog is the following linear algebra formula to compute the 
Krull dimension of a monomial subring.

\begin{theorem}{\rm(Herzog \cite{He3})}\label{dim-toric} Let
$K[F]=K[t^{v_1},\ldots,t^{v_m}]$ be a monomial subring of $S$ and 
let $A$ be the matrix with columns $v_1,\ldots,v_m$. Then,  
$\dim(K[F])={\rm rank}(A)$.
\end{theorem}

As a consequence, using a formula for the rank of the incidence matrix
of a graph $G$ \cite{Kulk1,raei}, we obtain a formula for the Krull dimension 
of $K[G]$.

\begin{proposition}\cite{Kulk1,raei}\label{dimkg}
Let $G$ be a graph with $s$ vertices and let $c_0(G)$ be the number of connected bipartite
components of $G$. Then
$$
\dim(K[G])=s-c_0(G).
$$
\end{proposition}

For graded Cohen--Macaulay monomial subrings, computing 
the a-invariant is equivalent to computing the
regularity.

\begin{corollary}\label{jan27-25}
Let $K[F]=K[t^{v_1},\ldots,t^{v_m}]$ be a standard graded Cohen--Macaulay monomial subring of $S$ and 
let $A$ be the matrix with columns $v_1,\ldots,v_m$. Then,  
$${\rm reg}(K[F])=a(K[F])+{\rm rank}(A).$$
\end{corollary}

\begin{proof} This follows readily from Theorems~\ref{reg-cm} and
\ref{dim-toric}.
\end{proof}

Let $G$ be a connected bipartite graph. The canonical module of $K[G]$
can be described in terms of blocking polyhedra \cite[Theorem~4.8]{shiftcon}, 
and the a-invariant of $K[G]$ can be computed using a linear
program \cite[Theorem~4.1]{shiftcon}. 

\begin{corollary} Let $G$ be a connected bipartite graph and let
$K[G]$ be its edge subring. Then, 
$$
{\rm reg}(K[G])=a(K[G])+s-1.
$$
\end{corollary}
\begin{proof}
As the edge subring $K[G]$ is 
normal (Theorem~\ref{bowtie-main-teo}) and Cohen--Macaulay
(Theorem~\ref{hoch-theo}), by
Theorem~\ref{reg-cm} and Proposition~\ref{dimkg}, one has 
$$
{\rm reg}(K[G])=a(K[G])+\dim(K[G])=a(K[G])+s-c_0(G)=a(K[G])+s-1,
$$
and the proof is complete. 
\end{proof}

Hence, for connected bipartite graphs, we can also compute the regularity of $K[G]$ using a linear
program, and furthermore the $a$-invariant and the regularity of 
$K[G]$ can be interpreted in
combinatorial terms, see \cite[Remark 4.3]{Bhaskara-VanTuyl} and 
\cite[Proposition~4.2]{shiftcon}. For recent formulas of homological invariants of toric
ideals of edge subrings of bipartite graphs, see
\cite{almousa-dochtermann-smith,Bhaskara-VanTuyl,HB-reg-g}. 

If $G$ is an unmixed bipartite graph, by Theorem~\ref{ehrhart-unmixed}
one has ${\rm reg}(K[G])\leq\beta_1(G)-1$. The following result
shows this inequality for all bipartite connected graphs. 

\begin{theorem}{\rm(Herzog and Hibi~\cite[Theorem~1]{HB-reg-g})} If
$G$ is a 
connected graph and
$\beta_1(G)$ is the matching number of $G$, then 
$$
{\rm reg}(K[G])\leq\begin{cases}\beta_1(G)-1&\mbox{if }G\mbox{ is 
bipartite},\\
\beta_1(G)&\mbox{if }G \mbox{ is
non-bipartite and }K[G]\mbox{ is normal}.
\end{cases}
$$ 
\end{theorem}

\begin{proposition}\cite[Proposition~4.6]{join}\label{cone}
Let $G$ be a connected non-bipartite graph with $s$ vertices and let
$C(G)$ be the cone over $G$. 
If $K[G]$ is normal, then
\[
a(K[G])-1\leq a(K[C(G)])\leq -\left\lceil\frac{s+1}{2}
\right\rceil.
\]
\end{proposition}

Let $I$ be a monomial ideal of $S$, recall that
$\mathcal{G}(I)=\{t^{v_1},\ldots,t^{v_m}\}$ denotes the minimal set of
generators of $I$. We denote the monomial subring $K[\mathcal{G}(I)]$ of
$S$ simply by $K[I]$, and denote the monomial subring
$K[t^{v_1}z,\ldots,t^{v_m}z]$ of $S[z]$ simply by $K[Iz]$, where $z$
is a new variable. 

Fix an integer $k\geq 1$. A monomial ideal is called
$k$-\textit{uniform} if it is generated by monomials of degree $k$. 
The following result shows that the function 
$I\mapsto {\rm reg}(K[I])$, $K$ a field, $I$ a $k$-uniform normal
monomial ideal, is a monotone function.

\begin{theorem}\label{vila-obs} If $I$, $J$ are two normal monomial ideals of $S$ generated by
monomials of the same degree $k\geq 1$ and
$\mathcal{G}(I)\subset\mathcal{G}(J)$, 
then 
$$
{\rm reg}(K[I])\leq{\rm reg}(K[J]).
$$
\end{theorem}
 
\begin{proof} By the generalized descent of normality criterion
(Theorem~\ref{mar3-01}), $K[Iz]$, $K[Jz]$ are 
the Ehrhart rings $A(\mathcal{P})$, $A(\mathcal{Q})$ of the 
lattice polytopes 
$$ 
\mathcal{P}={\rm conv}(\{a\mid t^a\in\mathcal{G}(I)\}),\quad\quad  
\mathcal{Q}={\rm conv}(\{a\mid t^a\in\mathcal{G}(J)\}),
$$
respectively. In particular $K[Iz]$ and $K[Jz]$ are normal
(Proposition~\ref{ehrhart-normal}) and Cohen--Macaulay
(Theorem~\ref{hoch-theo}). Similarly, by the descent of normality criterion
(Theorem~\ref{feb16-23}), $K[I]$ and $K[J]$ are normal and
Cohen--Macaulay. The rings $K[Iz]$ and $K[I]$ (resp. $K[Jz]$
and $K[J]$) have the same
Hilbert series since they are isomorphic as standard graded algebras
with the normalized grading. Therefore, by Theorem~\ref{reg-cm}, one has
$$\deg(h_\mathcal{P}(x))={\rm reg}(K[Iz])={\rm reg}(K[I])\mbox{ and }
\deg(h_{\mathcal{P}_1}(x))={\rm reg}(K[Jz])={\rm reg}(K[J]). 
$$
\quad Hence, by  Stanley's monotonicity theorem
(Theorem~\ref{Stanley-nonneg-mono}), we get that
$0\leq h_{\mathcal{P}}(x)\leq h_{\mathcal{P}_1}(x)$ coefficientwise.
Hence, $\deg(h_{\mathcal{P}}(x))\leq 
\deg(h_{\mathcal{P}_1}(x))$ and the proof is complete.
\end{proof}

If $G_1$ is an induced subgraph of $G_2$, then ${\rm
reg}(K[G_1])\leq{\rm reg}(K[G_2])$
\cite[Theorem~3.6]{Ha-Beyarslan-Okeefe}. 
The induced subgraph assumption is
necessary here; see the example before
\cite[Question~6.12]{almousa-dochtermann-smith}. Part (3) of the
following corollary was proved in 
\cite[Theorem~6.11]{almousa-dochtermann-smith}  
for connected bipartite graphs, and part (2) recovers  
a positive answer to a recent question by Almousa, Dochtermann and
Smith \cite[Question 6.12]{almousa-dochtermann-smith} that was first
shown in the affirmative by Akiyoshi Tsuchiya. His proof appears in
\cite{almousa-dochtermann-smith} right after Question 6.12.

\begin{corollary}\label{question6.12} Let $G_1$ be a subgraph of $G_2$. The following hold.
\begin{enumerate}
\item[(1)] If $I(G_i)$ is normal for $i=1,2$, 
then ${\rm reg}(K[G_1])\leq{\rm reg}(K[G_2])$. 
\item[(2)] \cite{almousa-dochtermann-smith}  
If $G_i$ is connected and $K[G_i]$ is normal for $i=1,2$, then 
${\rm
reg}(K[G_1])\leq{\rm reg}(K[G_2])$. 
\item[(3)] If $G_i$ is bipartite for $i=1,2$, then ${\rm
reg}(K[G_1])\leq{\rm reg}(K[G_2])$. 
\end{enumerate}
\end{corollary}

\begin{proof} (1) This follow at once from Theorem~\ref{vila-obs}. 

(2) As $G_1$ and $G_2$ are connected graphs, by
Theorem~\ref{roundup-iff-rounddown}, $I(G_i)$ is a normal ideal for $i=1,2$.
Thus, by part (1), ${\rm
reg}(K[G_1])\leq{\rm reg}(K[G_2])$.  

(3) A bipartite graph has no Hochster configurations. Then, by
the normality criterion of Theorem~\ref{apr16-02}, $I(G_i)$ are
normal for $i=1,2$. Thus, by part (1), ${\rm
reg}(K[G_1])\leq{\rm reg}(K[G_2])$. 
\end{proof}

\begin{proposition}\label{may9-23} Let $S_{s,k}$ and $S^{(k)}$ be the
$k$-th squarefree Veronese subring and the $k$-th Veronese subring of
$S$, and let $I$ be a $k$-uniform monomial ideal such that
$s\geq 2k\geq 4$ and  $S_{s,k}\subset K[I]\subset S^{(k)}$. The
following hold.

{\rm (a)} If $I$ is normal, then ${\rm
reg}(K[I])=s-\lceil\frac{s}{k}\rceil$.

{\rm (b)} If $I$ is an ideal of Veronese type, then ${\rm
reg}(K[I])=s-\lceil\frac{s}{k}\rceil$.

{\rm (c)} If $k=2$, then $I$ is normal and ${\rm
reg}(K[I])=s-\lceil\frac{s}{2}\rceil$. 
\end{proposition}

\begin{proof} (a) By Proposition~\ref{a-invariant} and
the formula for $a(S^{(k)})$ given in
\cite{Got,Matsuo}, we get 
$$
{\rm reg}(S_{s,k})=s-\left\lceil\frac{s}{k}\right\rceil=
{\rm reg}(S^{(k)}).
$$
\quad Then, by Theorem~\ref{vila-obs}, we get ${\rm
reg}(K[I])=s-\lceil\frac{s}{k}\rceil$. 

(b) This follows from part (a) since ideals of Veronese type are
normal \cite[Proposition~4.9]{Lisboar}.

(c) Theorem~\ref{wolmer-ic-rees} is valid for multigraphs by
regarding loops as odd cycles \cite[p.~450]{monalg-rev}. It is seen
that $I$ is normal because $I$ is the edge ideal of the multigraph
obtained from the complete graph $\mathcal{K}_s$ by adding a loop at
each vertex $t_i$ with $t_i^2\in I$. The normality also follows
noticing that $I$ is an ideal of Veronese type and using
\cite[Proposition~4.9]{Lisboar}. Then, the formula for the 
regularity follows from part (a).
\end{proof}

Let $B\neq(0)$ be an integral matrix. The
greatest common divisor of all the non-zero $r\times r$ 
sub determinants of $B$ will be denoted by
$\Delta_r(B)$.

\begin{theorem}\cite[Theorem~3.9]{ehrhart}\label{ehrh-mon} 
Let $I=(t^{v_1},\ldots,t^{v_m})$ be a monomial ideal of $S$ and let 
$\mathcal{P}={\rm conv}(v_1,\ldots,v_m)$ be its Newton polytope. 
If $B$ is the matrix whose columns are the vectors in
$\{(v_i,1)\}_{i=1}^m$ and $r={\rm rank}(B)$, 
then $\Delta_r(B)=1$ if and only if 
$\overline{K[Iz]}=A(\mathcal{P})$.
\end{theorem}

\begin{proposition}\cite[Corollary~{10.2.12}]{monalg-rev}\label{mar9-01-1g} 
Let $I=(t^{v_1},\ldots,t^{v_m})$ be the edge ideal of a graph $G$ and let
$B$ be the matrix whose columns are the vectors in
$\{(v_i,1)\}_{i=1}^m$. If $c_1$ is the number of 
non-bipartite components of $G$ and $r$ is the rank of $B$, then 
$$
\Delta_r(B)=
\left\{\begin{array}{ll}
2^{c_1-1}&\mbox{ if }\, c_1\geq 1,\\
1&\mbox{ if }\, c_1=0.
\end{array}\right.
$$
\end{proposition}

Let $G$ be a graph with vertex set $V=\{t_1,\ldots,t_s\}$ and edge
set $E(G)$. 
The {\it edge polytope\/} of $G$, denoted $\mathcal{P}_G$, is given by
$$
\mathcal{P}_G:={\rm conv}(\{e_i+e_j\mid\{t_i,t_j\}\in
E(G)\})\subset\mathbb{R}^s.
$$
\quad The following result gives another classification of the normality of the edge ideal of a
graph.

\begin{theorem}\label{normality-criterion-ehrhart}
Let $G$ be a graph, let $I(G)$ be its edge ideal, let
$\mathcal{P}_G$ be the edge polytope of $G$ and let
$A(\mathcal{P}_G)$ be its Ehrhart ring. The following conditions
are equivalent:
\begin{enumerate}
\item[(i)] $K[I(G)z]=A(\mathcal{P}_G)$.
\item[(ii)] $G$ has at most one connected non-bipartite component $G_i$
and $I(G_i)$ is normal.
\item[(iii)] $I(G)$ is normal.
\end{enumerate}
\end{theorem}

\begin{proof} (i)$\Rightarrow$(ii) Let
$\mathcal{G}(I(G))=\{t^{v_1},\ldots,t^{v_m}\}$ be the generating set of
$I(G)$, let $B$ be the matrix with column vectors
$(v_1,1),\ldots,(v_m,1)$ and let $r$ be the rank of $B$. 
By Theorem~\ref{ehrh-mon}, one has that $\Delta_r(B)=1$. On the 
other hand, denoting the number of non-bipartite components of $G$ by
$c_1$, by Proposition~\ref{mar9-01-1g}, one has that
$\Delta_r(B)=2^{c_1-1}$ if $G$ is not bipartite and $\Delta_r(B)=1$ if
$G$ is bipartite. Thus, $c_1\leq 1$, that is, $G$ has at most $1$ one
connected non-bipartite component $G_i$. To show that $I(G_i)$ is
normal we argue by contradiction assuming that $I(G_i)$ is not normal.
Then, by the normality criterion of Theorem~\ref{apr16-02}, there is a
Hochster configuration $C_1,C_2$ of $G_i$. Then, by
Lemma~\ref{sc-lemma}, the monomial $M_{C_1,C_2}$ associated to
$C_1,C_2$ is not in the Rees algebra $S[I(G)z]$ because $C_1,C_2$ is
also a Hochster configuration of $G$, and consequently
$M_{C_1,C_2}$ is not in the subring $K[I(G)z]$. Note that 
$$
a:=\bigg(\sum_{t_i\in C_1\cup C_2}e_i\bigg)\in\left(\frac{|C_1|+|C_2|}{2}\right)\mathcal{P}_G,
$$
where $|C_i|$ is the length of $C_i$, $i=1,2$, because
$(2a)/(|C_1|+|C_2|)\in\mathcal{P}_{G}$. Thus,
$$
M_{C_1,C_2}=t^az^{\frac{|C_1|+|C_2|}{2}}\in A(\mathcal{P}_G),
$$
and consequently, $K[I(G)z)\subsetneq A(\mathcal{P}_G)$, a
contradiction.

(ii)$\Rightarrow$(i) As $I(G_i)$ is normal, by the normality criterion
of Theorem~\ref{apr16-02}, $G_i$ has no Hochster configurations. Then, $G$ has no Hochster
configurations. Indeed, if $C_1,C_2$ is a Hochster
configuration of $G$, then both odd cycles must lie in $G_i$, that is, $C_1,C_2$ is a
Hochster configuration of $G_i$, a contradiction. Then, again by
Theorem~\ref{apr16-02}, $I(G)$ is normal. Hence, by the generalized
descent of normality criterion (Theorem~\ref{mar3-01}),
$K[I(G)z]=A(\mathcal{P}_G)$. 

(ii)$\Leftrightarrow$(iii) This follows readily from the normality criterion
of Theorem~\ref{apr16-02}. 
\end{proof}

\section{Normalization index, reduction numbers, and Noether
normalizations}\label{reduction-section}

We begin by presenting some degree bounds for normalizations of Rees
algebras of monomial ideals and for normalizations of monomial
subrings. 

\begin{theorem}{\rm(Bruns, Vasconcelos, -\,,
\cite[Theorem~3.3(b)]{BVV})}\label{may4-23}
Let $I$ be a uniform squarefree monomial ideal of the polynomial ring $S=K[t_1,\ldots,t_s]$ 
generated by monomials of the same degree $k\geq 1$. 
Then, the normalization $\overline{S[Iz]}$ of the Rees algebra
$S[Iz]$ is generated as an $S[Iz]$-algebra, by
elements $g\in S[z]$ of $z$-degree at most $s-1$.
\end{theorem}

The next result complements \cite[Theorem~3.3(a)]{BVV} for the class 
of uniform monomial ideals.

\begin{proposition}\cite[Proposition~2.5]{hilbsam}\label{mar2-02} If
$I$ is a uniform monomial ideal of $S$ 
generated by monomials of degree $k$ and $2\leq k<s$, then the 
normalization $\overline{S[Iz]}$ 
is generated as an $S[Iz]$-module by elements $g\in S[z]$ 
of $z$-degree at most $s-\lfloor s/k\rfloor$. 
\end{proposition}

To give a variant of Theorem~\ref{may4-23} for the subalgebra
generated by monomials of the same degree $k$, recall that a homogeneous
polynomial of degree $ik$ has \textit{normalized degree} $i$.  

\begin{theorem}{\rm(Bruns, Vasconcelos, -\,,
\cite[Theorem~3.4]{BVV})}\label{feb13-25}
Suppose that $F$ consists of monomials of the same degree $k$, and
let $A=K[F]\subset S=K[t_1,\ldots,t_s]$ be the monomial subring
generated by $F$. Then, the normalization $\overline{A}$ of $A$ is
generated as an $A$-module, and thus as an $A$-algebra, by elements 
$g\in S$ of normalized degree at most $\dim(A)-1$.
\end{theorem}

By the following result of Vasconcelos, the $I$-filtration ${\mathcal
F}=\{\overline{I^n}\}_{n=0}^\infty$ of a monomial ideal $I$ of $S$ stabilizes 
for $n\geq s=\dim(S)$. 

\begin{theorem}{\rm(Vasconcelos \cite[Theorem~7.58]{bookthree})}\label{may3-23} Let $I$ be a monomial
ideal of $S$. Then
$$\overline{I^n}=I\, \overline{I^{n-1}}\mbox{ for all }n\geq s=\dim(S).
$$  
\end{theorem}

\begin{definition}\label{norm-index-def}
The \textit{normalization index} of an ideal $I$ of $S$ is the smallest non-negative
integer ${\rm N}={\rm N}(I)$ such that $\overline{I^{n+1}}=I\,
\overline{I^{n}}$
for all $n\geq{\rm N}$. 
\end{definition}

\begin{corollary}{\rm(Reid, Roberts and Vitulli \cite{vitulli})}\label{vitulli-roberts} If $I$ is a
monomial ideal of $S=K[t_1,\ldots,t_s]$ and $I^n$ is integrally closed 
for all $n\leq s-1$, then $I$ is normal.
\end{corollary}

Using Carath\'eodory's theorem for cones \cite[Corollary 7.1i]{Schr},
one can show the following related result. 

\begin{proposition}\cite[Proposition~2.1]{hilbsam}\label{aug6-05} Let
$I=(t^{v_1},\ldots,t^{v_m})$ be a monomial ideal and let $r_0$ be the
rank of its incidence matrix. If $v_1,\ldots,v_m$ lie in 
a hyperplane of $\mathbb{R}^s$ not containing the 
origin, then $\overline{I^n}=I\, \overline{I^{n-1}}$ for $n\geq r_0$.  
\end{proposition}

A similar result to Corollary~\ref{vitulli-roberts} holds for Simis
edge ideals of graphs. Recall that a monomial ideal $I$ is called a \textit{Simis ideal} 
if $I^{(n)}=I^n$ for all $n\geq 1$. 

\begin{proposition}{\rm(Bahiano \cite[Corollary 2.13]{bahiano})} 
Let $I=I(G)$ be the edge ideal of a graph $G$. Then, $I$ is a Simis
ideal if and only if $I^{n}=I^{(n)}$ for all $n\leq {\rm
ht}(I)$.
\end{proposition}

Let us recall the Dedekind--Mertens formula and present an application
to reduction numbers. If $S$ is a 
commutative ring and  $f=f(t) 
\in S[t]$ is a polynomial,
say 
$$f = a_0 + a_1t+\cdots + a_mt^m,$$
the {\em content\/} of $f$ is the
$S$-ideal $(a_0,\ldots, a_m)$. It is denoted  by $c(f)$.  
Given another polynomial $g$, the 
{\em Gaussian ideal\/} of $f$
and $g$ is the $S$-ideal $c(f g)$. 
By the classical lemma of Gauss: If 
$S$ is a principal ideal domain, then
$$
c(f g)= c(f)c(g).
$$
\quad In general, these two ideals are different but one has the
following formula due to Dedekind--Mertens 
(see \cite{Edwards} and \cite{Northcott}):
\begin{equation}\label{contentformula}
c(f g)c(g)^m = c(f) c(g)^{m+1}.
\end{equation}
\quad We consider the ideal $c(f g)$ in the case when $f$, $g$
are generic polynomials. Multiplying both
sides of Eq.~\eqref{contentformula} by $c(f)^m$, we get
\begin{equation}\label{contentformula2}
c(f g)[c(f) c(g)]^m = c(f) c(g) [c(f) c(g)]^{m}.
\end{equation}
\quad This last formula is sharp 
in terms of the exponent $m=\deg(f)$ \cite{CVV}. To make this connection, 
we recall the notion of a {\em reduction\/}
of an ideal. 
\begin{definition}\label{reduction-number-def}
Let $S$ be a ring and $I$ an ideal. A
{\em reduction\/} of $I$ is an 
ideal $J \subset I$ such that, for some nonnegative integer $r$, 
the equality 
$I^{r+1}=
J I^r$
holds. The smallest such integer is the {\em reduction
number\/} $r_J(I)$ of $I$ 
{\it relative\/} to $J$. 
The {\it reduction number\/} 
$r(I)$ of $I$ is the 
smallest reduction number among all 
reductions $J$ of $I$. 
\end{definition}

Note that Eq.~\eqref{contentformula2} says that  $J=c(f g)$ is a reduction 
for $I = c(f) c(g)$, and that the reduction number is at most 
$\min\{\deg(f), \deg(g)\}$. 

\begin{theorem}\cite[Part $0$, p.~3]{Edwards}\label{sharpind} 
Let $S=K[x_0, \ldots, x_m,y_0, \ldots, y_n]$ be a polynomial ring 
over a field $K$ and let $h_q= \sum_{i+j=q}x_iy_j$. 
Then 
$$
A=K[h_0,h_1,\ldots,h_{m+n}]\hookrightarrow 
K[\mathcal{K}_{m+1,n+1}]=K[\{x_iy_j\vert\, 0\leq i\leq m,\ 0\leq j\leq n\}]
$$
is a Noether normalization of $K[\mathcal{K}_{m+1,n+1}]$.
\end{theorem}

The next theorem determines the 
reduction number of $I = c(f)c(g)$ relative to $J= c(f g)$.

\begin{theorem}{\rm(Corso, Vasconcelos, -\,, \cite[Theorem~2.1]{CVV})}
Let $S=K[x_0, \ldots, x_m,y_0, \ldots, y_n]$ be a polynomial ring 
over a field $K$ and let $f,g\in S[t]$ be the generic polynomials 
\[
f=x_0+x_1t+\cdots+x_mt^m 
\quad \mbox{and} \quad g=y_0+y_1t+\cdots+y_nt^n
\]  
with $n\geq m$. If $I = c(f)c(g)$ and $J= c(f g)$, 
then $r_J(I) = m$ and the factor $c(f)^m$ in the content formula
of Eq.~\eqref{contentformula} is sharp.
\end{theorem}

\section{Multiplicities, Hilbert functions and $\mathfrak{m}$-fullness
of monomial ideals}\label{multiplicity-section}
Let $S=K[t_1,\ldots,t_s]$ be a polynomial ring over a field $K$,
$s\geq 2$, and let $I$ a zero-dimensional monomial ideal of $S$
minimally generated by $t^{v_1},\ldots,t^{v_m}$. We may assume that 
$v_i=a_ie_i$ for $i=1,\ldots,s$, where $a_i\in\mathbb{N}_+$ and $|v_i|=a_i$ for $i=1,\ldots,s$. Setting 
$\alpha_0=(1/|v_1|,\ldots,1/|v_s|)$, we may also assume that
$\{v_{s+1},\ldots,v_r\}$ is the set of all $v_i$ such that $\langle
v_i,\alpha_0\rangle<1$. Consider the lattice polytopes in $\mathbb{R}^s$ given
by 
$$
\mathcal{P}_0:={\rm conv}(v_1,\ldots,v_r),\quad \Delta:={\rm
conv}(0,v_1,\ldots,v_s)=\{x\mid x\geq0;\, \langle
x,\alpha_0\rangle\leq 1\},
$$
and the so-called \textit{region} $\mathcal{P}:=\Delta\setminus\mathcal{P}_0$
\textit{defined} by $I$ \cite[p.~93]{mc8}, see
Example~\ref{example-multi}. 

The
\textit{multiplicity} of $I$, denoted $e(I)$, is the integer
$$
e(I):=\lim_{n\rightarrow\infty}\frac{\ell(S/I^n)}{n^s}\,d!,
$$
where $\ell(\cdot)$ is the length function, see \cite[Propositions~2.8
and 7.7]{bookthree}. 

The multiplicity of $I$ is related to volumes: 

\begin{proposition}{\rm(Teissier \cite[p.~131]{Teissier})}\label{teissier-e}
If $\mathcal{P}=\Delta\setminus\mathcal{P}_0$ is the region of
$\mathbb{R}^s$ 
defined by $I$, then
$$
e(I)={s!}\,{\rm vol}(\mathcal{P}).
$$
\end{proposition}

\begin{example}\label{example-multi} Consider the ideal 
$I=(t_1^6,\,t_2^5,\,t_1^2t_2^2,\,t_1^3t_2)$. Let 
$\mathcal{P}=\Delta\setminus\mathcal{P}_0$ be the region defined by
$I$ depicted in 
Figure~\ref{figure}, where $\Delta={\rm
conv}(0,\, 6e_1,\, 5e_2)$ and $\mathcal{P}_0={\rm conv}(6e_1,\, 5e_2,\,
(2,2),(3,1))$. Then, ${\rm vol}(\mathcal{P})={\rm vol}(\Delta)-{\rm
vol}(\mathcal{P}_0)=15-5=10$ and $e(I)=2!{\rm vol}(\mathcal{P})=20$.
\begin{figure}[ht]
\begin{tikzpicture}[scale = 0.5]
\draw[thick,->] (-1,0)--(8,0) node[right,below] {}; 
\foreach \x/\xtext in {1/1, 2/2, 3/3, 4/4, 5/5, 6/6} 
\draw[shift={(\x,0)}] (0pt,0.5pt)--(0pt,-0.5pt) node[below] {$\xtext$};
%
\foreach \y/\ytext in {1/1, 2/2, 3/3, 4/4, 5/5} 
\draw[shift={(0,\y)}] (0.5pt,0pt)--(-0.5pt,0pt) node[left] {$\ytext$};
\draw[thick,->] (0,-1)--(0,7) node[left,above] {}; 
\node[below] at (-0.5,-0.1){$0$};
\draw[black,thick,-] (0,5) -- (6,0); 
\draw[black,thick,-] (0,5) -- (2,2); 
\draw[black,thick,-] (2,2) -- (3,1); 
\draw[black,thick,-] (3,1) -- (6,0); 

\node[color=black,right] at (-0.2,4.5) {$\cdot$};
\node[color=black,right] at (-0.1,4) {$\cdot$};
\node[color=black,right] at (-0.2,3.7) {$\cdot$};
\node[color=black,right] at (0.1,3.4) {$\cdot$};
\node[color=black,right] at (0.1,3) {$\cdot$};
\node[color=black,right] at (-0.22,3) {$\cdot$};
\node[color=black,right] at (0.3,2.9) {$\cdot$};
\node[color=black,right] at (-0.2,2.8) {$\cdot$};
\node[color=black,right] at (0.7,2.8) {$\cdot$};
\node[color=black,right] at (0.23,2.7) {$\cdot$};
\node[color=black,right] at (-0.1,2.5) {$\cdot$};
\node[color=black,right] at (-0.2,2.6) {$\cdot$};
\node[color=black,right] at (0.7,2.5) {$\cdot$};
\node[color=black,right] at (0.1,2.2) {$\cdot$};
\node[color=black,right] at (0.8,2.1) {$\cdot$};
\node[color=black,right] at (1.3,2.1) {$\cdot$};
\node[color=black,right] at (-0.2,2.0) {$\cdot$};
\node[color=black,right] at (1.0,2.0) {$\cdot$};
\node[color=black,right] at (0.1,1.8) {$\cdot$};
\node[color=black,right] at (-0.1,1.8) {$\cdot$};
\node[color=black,right] at (1.5,1.8) {$\cdot$};
\node[color=black,right] at (0.35,1.7) {$\cdot$};
\node[color=black,right] at (0.3,1.5) {$\cdot$};
\node[color=black,right] at (1.7,1.5) {$\cdot$};
\node[color=black,right] at (2.1,1.5) {$\cdot$};
\node[color=black,right] at (1.7,1.3) {$\cdot$};
\node[color=black,right] at (1.3,1.3) {$\cdot$};
\node[color=black,right] at (0.1,1.1) {$\cdot$};
\node[color=black,right] at (2.3,1.1) {$\cdot$};
\node[color=black,right] at (1.7,1) {$\cdot$};
\node[color=black,right] at (0.3,0.9) {$\cdot$};
\node[color=black,right] at (2.2,0.9) {$\cdot$};
\node[color=black,right] at (2.4,0.8) {$\cdot$};
\node[color=black,right] at (1.5,0.7) {$\cdot$};
\node[color=black,right] at (2.1,0.7) {$\cdot$};
\node[color=black,right] at (2.5,0.5) {$\cdot$};
\node[color=black,right] at (3.5,0.5) {$\cdot$};
\node[color=black,right] at (0,0.5) {$\cdot$};
\node[color=black,right] at (1.0,0.4) {$\cdot$};
\node[color=black,right] at (3,0.4) {$\cdot$};
\node[color=black,right] at (1.2,0.4) {$\cdot$};
\node[color=black,right] at (2.7,0.3) {$\cdot$};
\node[color=black,right] at (2.3,0.3) {$\cdot$};
\node[color=black,right] at (2,0.3) {$\cdot$};
\node[color=black,right] at (4,0.2) {$\cdot$};
\node[color=black,right] at (0.1,0.2) {$\cdot$};
\node[color=black,right] at (0.5,0.2) {$\cdot$};
\node[color=black,right] at (0.9,0.2) {$\cdot$};

\node[color=black,right] at (0.7,1) {$\mathcal{P}$};
\node[color=black,right] at (2.5,1.7) {$\mathcal{P}_{0}$};
\draw[fill=black] (0,5) circle (3pt);
\draw[fill=black] (6,0) circle (3pt);
\draw[fill=black] (2,2) circle (3pt);
\draw[fill=black] (3,1) circle (3pt);
\end{tikzpicture}
\caption{Region defined by 
$I=(t_1^6,\,t_2^5,\,t_1^2t_2^2,\,t_1^3t_2)$.}\label{figure}
\end{figure}
\end{example}

\begin{proposition}\cite[Proposition~1.7]{mc8}
If $p={\rm vol}(\mathcal{P}_0)/{\rm vol}(\Delta)$, then 
$$
e(I)=(1-p)|v_1|\cdots|v_s|=|v_1|\cdots|v_s|-s!{\rm
vol}(\mathcal{P}_0).
$$
\end{proposition} 

\begin{proof} Let $\mathcal{P}=\Delta\setminus\mathcal{P}_0$ be the
region defined by $I$. Then, by Proposition~\ref{teissier-e}, we get 
\begin{align*}
e(I)&=s!{\rm vol}(\mathcal{P})=s!\left[{\rm
vol}(\Delta)-{\rm vol}(\mathcal{P}_0)\right]=s!\left[\frac{|v_1|
\cdots|v_s|}{s!}-{\rm vol}(\mathcal{P}_0)\right]\\
 &=(1-p)|v_1|\cdots|v_s|,
\end{align*}
and consequently $e(I)=s!\left[{|v_1|
\cdots|v_s|}/{s!}-{\rm vol}(\mathcal{P}_0)\right]=|v_1|\cdots|v_s|-s!{\rm
vol}(\mathcal{P}_0)$.
\end{proof}

Thus, computing the multiplicity $e(I)$ of $I$ is equivalent to computing either ${\rm
vol(\mathcal{P}_0})$ or $p$. The volume ${\rm
vol(\mathcal{P}_0})$ of $\mathcal{P}_0$ can be computed with \textit{Normaliz}
\cite{Bruns-Volume,normaliz2}. Vasconcelos et al. \cite{mc8} found an
alternative method to approach the computation of $p$, i.e., the
computation of ${\rm vol}(\mathcal{P}_0)$ as a 
fraction of ${\rm vol}(\Delta)$ \cite[Section~2]{mc8},
\cite[Section~7.3]{bookthree}. The \text{multiplicity}  
function in \textit{Macaulay}$2$ \cite{mac2} gives us a handy way to
compute $e(I)$. 

The {\it Hilbert function\/} of the filtration 
${\mathcal F}=\{\overline{I^n}\}_{n=0}^\infty$ is defined as
$$
f(n)=\ell(S/\overline{I^n})=\dim_K(S/\overline{I^n});\ \ \
n\in\mathbb{N}\setminus\{0\};\ \ f(0)=0.
$$
\quad The \textit{Newton polyhedron} of $I$ is given by 
$\mathbb{R}_+^s+{\rm conv}(v_1,\ldots,v_m)$.

\begin{proposition}\cite{hilbsam}\label{incredibles} 
$\ell(S/\overline{I^n})=|\mathbb{N}^s\setminus n\mathcal{Q}|$, $n\geq
1$, 
where $\mathcal{Q}$ is the Newton polyhedron of $I$.
\end{proposition}

The Hilbert function of $\mathcal F$ is a polynomial function of
degree $s$ that can be expressed as a difference of the Ehrhart
polynomials $E_{\Delta}$ and $E_{\mathcal{P}_0}$ of $\Delta$ and
$\mathcal{P}_0$, respectively.

\begin{proposition}\cite[Proposition~3.6]{hilbsam}\label{may29-1-05}
$f(n)=E_{\Delta}(n)-E_{\mathcal{P}_0}(n)$ for all 
$n\in\mathbb{N}$. 
\end{proposition}

We describe two results of Vasconcelos et al. \cite{icmi} 
about $\mathfrak{m}$-full ideals. The first result concerns 
$\mathfrak{m}$-full $\mathfrak{m}$-primary monomial ideals in 
dimension two.

Let $S = K[t_1,t_2]$ be a polynomial ring over an infinite field $K$, let
$\mathfrak{m}=(t_1,t_2)$ and let
$I$ be an $\mathfrak{m}$-primary ideal of $S$ minimally generated by $n$
monomials, $n=\mu(I)$, that are listed lexicographically,
$$I=(t_1^{a_1},\,  t_1^{a_2}t_2^{b_{n-1}},\ldots,\,
t_1^{a_i}t_2^{b_{n-i+1}},\,  \ldots,\, 
t_1^{a_{n-1}}t_2^{b_2},\,  t_2^{b_1})
$$ 
with $a_1>\cdots>a_{n-1}>a_n:=0$, 
$b_1>\cdots>b_{n-1}>b_n:=0$. We assume that $I \neq
\mathfrak{m}^{n-1}$.   

The ideal $I$ is said to be {\em $\mathfrak{m}$-full} if
$(\mathfrak{m}I\colon f) = I$ for some $f\in\mathfrak{m}
\setminus\mathfrak{m}^2$ or equivalently $I$ is
$\mathfrak{m}$-full if $(I\colon\mathfrak{m})=(I\colon f)$
\cite[Proposition~14.1.6]{huneke-swanson-book}. 

A useful 
characterization of $\mathfrak{m}$-full ideals is the  
following result:

\begin{theorem}{\rm(Rees and Watanabe \rm\cite[Theorem 4]{JWat})}
\label{Reestheorem} If $I$ is an $\mathfrak{m}$-primary ideal of a
regular local ring $(S,\mathfrak{m})$ of dimension 
two, then $I$ is $\mathfrak{m}$-full if and only if for all ideals $I\subset J$, $\mu(J)\leq \mu(I)$.
\end{theorem}

The {\it order} of $I$ 
is given by ${\rm ord}(I):=\max\{k\mid  I\subset
\mathfrak{m}^k\}$. In regular local rings of
dimension two one has the inequality $\mu(I)\leq {\rm ord}(I)+1$ 
\cite[Lemma~14.1.3]{huneke-swanson-book}. 
If $\mu(I)={\rm ord}(I)+1$, then $I$ is $\mathfrak{m}$-full 
(see \cite[p.~212]{icmi}, \cite[p.~502]{monalg-rev}).

\begin{theorem}{\rm(Gimenez, Simis, Vasconcelos, -\,,
\cite[Theorem~2.9]{icmi})}\label{p-w-v} $I$ is
$\mathfrak{m}$-full 
if and only if there is $1\leq k\leq n$ such
that the following conditions hold
\begin{enumerate}
\item[\rm(i)] $b_{n-i}-b_{n-i+1}=1$ for $1\leq i\leq k-1$,
\item[\rm(ii)]
$k=n$ or $k<n$ and $b_{n-k}-b_{n-k+1}\geq 2$,
\item[\rm(iii)]
$a_i-a_{i+1}=1$ for $k\leq i\leq n-1$.
\end{enumerate}
\end{theorem}

If $I$ is complete, that is, $I=\overline{I}$, 
then $I$ is $\mathfrak{m}$-full \cite[Theorem 5]{JWat}. 
In polynomial rings in two variables any complete 
ideal is normal; more generally one has the following result of 
Zariski.

\begin{theorem}{\rm\cite[Appendix~5]{zariski-samuel-ii}}\label{Zariski} 
If $I_1,\ldots,I_r$ are complete ideals in a regular local ring $S$
of dimension two, then $I_1\cdots I_r$ is complete.
\end{theorem}

In \cite{crispin-quinonez} Qui\~nonez studied the normality of
$\mathfrak{m}$-primary monomial ideals in $K[t_1,t_2]$ 
and gave a criterion in terms of certain partial blocks and
associated sequences of rational numbers. 

A problem suggested by Simis and
Vasconcelos is to study the arithmetical and
homological properties of the Rees algebras of $\mathfrak{m}$-full ideals.
Even for monomial ideals, these often fail to be Cohen-Macaulay, as
the following example shows \cite[p.~223]{icmi}. The ideal  
$$I = (t_1^{11},\,  t_1^8t_2,\,  t_1^6t_2^2,\,  t_1^5t_2^3,\, 
t_1t_2^4,\,  t_2^{10})$$
is $\mathfrak{m}$-full and $S[Iz]$  is {\em not} Cohen-Macaulay. 

The \textit{special fiber ring} is given by 
 $\mathcal{F}(I):=\mathcal{R}(I)/\mathfrak{m}\mathcal{R}(I)$,
$\mathcal{R}(I)=S[Iz]$. The following criterion complements a result
of Corso, Ghezzi, Polini and Ulrich \cite[Corollary 2.11]{CGPU03}. 

\begin{theorem}{\rm(Gimenez, Simis, Vasconcelos, -\,,
\cite[Theorem~3.8]{icmi})} \label{CMviafiber} Let $I$ 
be an $\mathfrak{m}$-primary $\mathfrak{m}$-full ideal. If the special fiber
ring $\mathcal{F}(I)$ is Cohen-Macaulay, then $S[Iz]$ is also Cohen-Macaulay.
 \end{theorem}

The same assertion holds for two-dimensional regular local rings of 
infinite residue field \cite{icmi}. The question of when $\mathfrak{m}I$
is integrally closed is studied by H\"ubl and Huneke \cite{Hubl-Huneke} using the 
special fiber ring $\mathcal{F}(I)$. 

\section{Blowup algebras of edge ideals of
clutters}\label{blowuprings-section}

In the 2000' the algebraic properties and invariants of
blowup algebras of monomial ideals are 
related to combinatorial optimization problems, polyhedral geometry and 
integer programming. In this section we introduce blowup algebras 
and their interplay with combinatorial optimization and the containment problem
comparing symbolic and ordinary powers of ideals.  

To study blowup algebras of squarefree monomial ideals Gitler, Reyes,
Valencia and the second author \cite{reesclu,clutters} introduced combinatorial optimization
methods---based on a paper by Huneke, Simis and Vasconcelos
\cite{HuSV}, a paper by Escobar, the second author and Yoshino
\cite{Lisboar}, and the books of Cornu\'ejols \cite{cornu-book} and Schrijver
\cite{Schr2}---and showed that blowup algebras can be used to study combinatorial
optimization problems of clutters. 

Edge ideals of clutters were introduced in 
Section~\ref{edge-ideals-all-flavors}. To
avoid repetitions, throughout this section, we continue to employ 
the notations and definitions used in 
that section. 

Let $S=K[t_1,\ldots,t_s]$ be a polynomial ring over a field $K$, let
$\mathcal C$ be a clutter with vertex  
set $V(\mathcal{C})=\{t_1,\ldots,t_s\}$ and edge set $E(\mathcal{C})$,
let $I=I(\mathcal{C})$ be the edge ideal of the clutter $\mathcal{C}$, let 
$\mathcal{G}(I):=\{t^{v_1},\ldots,t^{v_m}\}$ be the minimal set of
generators of $I$, and let $A$ be the incidence matrix of $I$ 
with column vectors $v_1,\ldots,v_m$. The matrix $A$ coincides 
with the incidence matrix of $\mathcal{C}$. In what follows we assume that
the height ${\rm ht}(I)$ of $I$ is at least $2$, that is, the covering number
$\alpha_0(\mathcal{C})$ of $\mathcal{C}$ is at least two, and that the
rows of $A$ are different from zero. 

The {\it blowup algebras\/} associated to $I$ are the following: 
(a) {\it Rees algebra\/}
\begin{equation}\label{jan21-25-1}
\mathcal{R}(I):=S[Iz]=S\oplus Iz\oplus\cdots\oplus I^{i}z^i\oplus\cdots
\subset S[z],
\end{equation}
where $z$ is a new variable, (b) {\it extended Rees
algebra\/}
\begin{equation}\label{jan21-25-2}
S[Iz,z^{-1}]:=S[Iz][z^{-1}]\subset S[z,z^{-1}],
\end{equation}
(c) {\it symbolic Rees algebra}
\begin{equation}\label{jan21-25-3}
\mathcal{R}_s(I):=S+I^{(1)}z+I^{(2)}z^2+\cdots+I^{(i)}z^i+\cdots\subset S[z],
\end{equation}
where $I^{(i)}$ is the $i$-th symbolic
power of $I$, and (d) {\it associated graded
ring\/}
\begin{equation}\label{jan21-25-4}
{\rm gr}_I(S):=S/I\oplus I/I^2\oplus\cdots\oplus
I^i/I^{i+1}\oplus\cdots\simeq S[Iz]\otimes_S(S/I),
\end{equation}
with multiplication
$$
(a+I^{i+1})(b+I^{j+1})=ab+I^{i+j+1}\ \ \ \ \  (a\in I^{i},\  b\in
I^{j}).
$$
\quad The following theorem in the 1980' is an
important element to make 
the connection between algebraic properties of blowup algebras and combinatorial
properties of clutters. As is seen later, any of the conditions below is equivalent 
to the equality of ordinary and symbolic powers of an edge ideal $I$. Recall that 
the ring ${\rm gr}_I(S)$ is called \textit{reduced} if its nilradical is zero.

\begin{theorem}{\rm(Huneke, Simis and Vasconcelos \cite[Theorem~1.11]{HuSV})}
\label{reduced-normal-cones} Let $I$ be the edge ideal of a clutter.
If the height of $I$ is $\geq 2$, then the 
following are equivalent{\rm :}
\begin{enumerate} 
\item[(i)] ${\rm gr}_I(S)$ is torsion-free over $S/I$.
\item[(ii)] ${\rm gr}_I(S)$ is reduced.
 \item[(iii)] $S[Iz]$ is normal and ${\rm Cl}(S[Iz])$, the divisor
class group of $S[Iz]$, is a free abelian group whose rank is the
number of minimal primes of $I$.
\end{enumerate}
\end{theorem}

Let ${\rm RC}(I)=\mathbb{R}_+\mathcal{A'}$ be the Rees cone of $I$
defined in Eq.~\eqref{rees-cone-eq}. Given $a\in {\mathbb
R}^{s+1}\setminus\{0\}$, the 
\textit{positive closed halfspace}
$H^+_a$ and its \textit{bounding hyperplane} $H_a$ are defined as
\[
H^+_a:=\{x\in{\mathbb R}^{s+1}\vert\, 
\langle x,a\rangle\geq 0\}\ \mbox{ and }\ H_a:=\{x\in{\mathbb R}^{s+1}\vert\, 
\langle x,a\rangle=0\},
\]
respectively. The
$\textit{covering polyhedron}$ of $I$, denoted by
$\mathcal{Q}(I)$, is the rational polyhedron
$$
\mathcal{Q}(I):=\{x\mid x\geq 0;\,xA\geq 1\},
$$
where $1=(1,\ldots,1)$.
The map
\begin{equation}\label{jan22-25}
E(\mathcal{C}^\vee)\rightarrow\{0,1\}^s,\quad C\mapsto\textstyle\sum_{t_i\in
C}e_i,
\end{equation}
induces a bijection between the set $E(\mathcal{C}^\vee)$ of minimal
vertex covers of $\mathcal{C}$ and the set
$\{u_1,\ldots,u_r\}$ of integral 
vertices of $\mathcal{Q}(I)$ \cite[Corollary~2.3]{reesclu}.

A polyhedron with only integral vertices is called \textit{integral} \cite[p.~232]{Schr}.

\begin{theorem}{\rm(\cite[Corollary~3.3]{reesclu},
\cite[Proposition~1.1.51]{monalg-rev})}\label{irred-rep-rc} The Rees cone of 
$I$ has a unique irreducible representation 
\begin{equation}\label{supp-hyp}
{\rm RC}(I)=\bigg(\bigcap_{i=1}^{s+1}H^+_{e_i}\bigg)\bigcap
\bigg(\bigcap_{i=1}^rH^+_{(\gamma_i,-d_i)}\bigg)
\bigcap\bigg(\bigcap_{i=r+1}^p H^+_{(\gamma_i,-d_i)}\bigg),
\end{equation}
where $r\leq p$, $\gamma_i=u_i$ and $d_i=1$ for $i=1,\ldots,r$, 
$\gamma_i\in\mathbb{N}^s\setminus\{0\}$ and 
$d_i\in\mathbb{N}\setminus\{0,1\}$ for $i>r$, and the non-zero entries of
$(\gamma_i,-d_i)$ are relatively prime for all $i$. Moreover, the covering
polyhedron $\mathcal{Q}(I)$ is integral if and only if $r=p$.
\end{theorem}

The hyperplanes bounding the closed halfspaces of Eq.~\eqref{supp-hyp} are the
\textit{supporting hyperplanes} of the Rees cone of $I$, and the $\gamma_i$'s
and $d_i$'s can be computed using \textit{Normaliz} \cite{normaliz2}. 

\begin{theorem}{\rm(\cite[Theorem~3.1]{reesclu},
\cite[Theorem~3.2]{clutters})}\label{reesclu-theorem} 
The vertex set $V(\mathcal{Q}(I))$ of the covering polyhedron
$\mathcal{Q}(I)$ is given by
$$V(\mathcal{Q}(I))=\{\gamma_1/d_1,\ldots,\gamma_r/d_r,\ldots,\gamma_p/d_p\}.$$
\end{theorem}

\noindent {\it Notation} In what follows we set 
$\ell_i:=(\gamma_i,-d_i)=d_i(\gamma_i/d_i,-1)$ and $u_i=\gamma_i/d_i$
for $i=1,\ldots,p$. 

Finding the vertices of covering polyhedra is useful for determining
the asymptotic resurgence of edge ideals
(Theorem~\ref{duality-ic-resurgence}). 
By the last two results, finding the supporting hyperplanes
of ${\rm RC}(I)$ is equivalent to finding the vertices of
$\mathcal{Q}(I)$. Another way to compute the vertices of
$\mathcal{Q}(I)$ 
is to find the extreme rays, i.e., the
$1$-dimensional faces of the \textit{Simis cone} of $I^\vee$ \cite{Lisboar}: 
\begin{equation*}
{\rm SC}(I^\vee):=\{x\in\mathbb{R}^{s+1}\vert\, x\geq 0;\, \langle
x,(v_i,-1)\rangle\geq 0\ \mbox{ for } \ i=1,\ldots,m\}, 
\end{equation*}
see \cite[Proposition 3.15]{intclos}.  For 
more information about Simis cones, see the discussion after
Eq.~\eqref{simis-cone}. 
There are methods to find all vertices of a general
polyhedron \cite{avis-fukuda,avis-fukuda1}. 

If the Rees algebra $S[Iz]$ is normal, the following result 
shows that the rank of the divisor class
group of $S[Iz]$ is the integer
$p$ appearing in the irreducible representation of the Rees 
cone of $I$ given in Eq~(\ref{supp-hyp}). Note that $p$ is also the number of
vertices of $\mathcal{Q}(I)$. 

\begin{proposition}{\rm(Simis and Trung \cite[Theorem~1.1]{simis-trung},
\cite[Proposition~12.7.3]{monalg-rev})}\label{okayama3}  
If $S[Iz]$ is normal, then its divisor class 
group ${\rm Cl}(S[Iz])$ is a free abelian group of rank $p$.   
\end{proposition}

Since the program {\it Normaliz\/} \cite{normaliz2} computes 
the irreducible representation of the Rees cone of $I$ and the 
integral closure of $S[Iz]$, 
the following result gives effective criteria for the reducedness of the associated 
graded ring ${\rm gr}_I(S)$ of $I$ (see \cite[Example~14.2.20]{monalg-rev}).

\begin{proposition}{\rm(Effective reducedness criterion
\cite[Proposition~3.4]{Lisboar})}\label{effcred} 
{\rm(a)} ${\rm gr}_I(S)$ is reduced 
if and only if the Rees algebra $S[Iz]$ is normal and $r=p$. {\rm(b)} ${\rm gr}_I(S)$ is reduced 
if and only if $S[Iz]$ is normal and 
$\langle\ell_i,e_{s+1}\rangle=-1$ for $i=1,\ldots,p$.
\end{proposition}

\begin{proof} $\Rightarrow$) Note that $r$ is the number of minimal
primes of $I$, since $r$ is the number of minimal
vertex covers of the clutter $\mathcal{C}$ (see Eq.~\eqref{jan22-25}). Then, by
Theorem~\ref{reduced-normal-cones} and Proposition~\ref{okayama3},
$S[Iz]$ is normal and 
${\rm Cl}(S[Iz])\simeq\mathbb{Z}^r\simeq\mathbb{Z}^p$. Thus, $r=p$.

$\Leftarrow)$ As $S[Iz]$ is normal and $r=p$, by Proposition~\ref{okayama3}, 
one has ${\rm Cl}(S[Iz])\simeq\mathbb{Z}^r$. Hence, by
Theorem~\ref{reduced-normal-cones}, ${\rm gr}_I(S)$ is reduced.
\end{proof}

To connect the result of Huneke, Simis and Vasconcelos
(Theorem~\ref{reduced-normal-cones}) with symbolic
powers, we recall the following: the \textit{Newton
polyhedron} of $I$ is the 
integral polyhedron 
\begin{equation}\label{NP-def}
{\rm NP}(I):=\mathbb{R}_+^s+{\rm 
conv}(v_1,\ldots,v_m),
\end{equation}
the integral closure of $I^n$ can be expressed as   
\begin{equation}
\overline{I^n}=
(\{t^a\mid a/n\in{\rm NP}(I)\})
\end{equation}
for all $n\geq 1$ \cite[Proposition~3.5(a)]{reesclu}, and 
the $n$-th {symbolic
power} of $I$ is
given by 
\begin{equation}\label{jun21-21-1}
I^{(n)}=(\{t^a\mid a/n\in\mathcal{Q}(I^\vee)\}),
\end{equation}
where $I^\vee$ is the Alexander dual of $I$ \cite[p.~78]{reesclu}.
Recall that $I^\vee$ is $I_c(\mathcal{C})$, the ideal of covers of
$\mathcal{C}$. 
The covering polyhedron $\mathcal{Q}(I^\vee)$ of $I^\vee$ is also called
the \textit{symbolic polyhedron} of
$I$ \cite[p.~50]{Cooper-symbolic}. 

For monomial ideals the Newton polyhedron is a sort of covering
polyhedron:  

\begin{proposition}\label{np-qa} Let $u_1,\ldots,u_p$ be 
the vertices of $\mathcal{Q}(I)$, $u_i=\gamma_i/d_i$ for
$i=1,\ldots,p$, and let $B$ be the matrix with
column vectors $u_1,\ldots,u_p$. The following hold.
\begin{enumerate}
\item[(a)] \cite[Proposition~3.5(b)]{reesclu} ${\rm
NP}(I)=\mathcal{Q}(B)=\{x\mid x\geq 0;\,
xB\geq 1\}$. 
\item[(b)] The vertices of ${\rm NP}(I)$ are
contained in $\{v_1,\ldots,v_m\}$. 
\end{enumerate}
\end{proposition}
\begin{proof} (b) Since ${\rm NP}(I)=\mathbb{R}_+^s+{\rm
conv}(v_1,\ldots,v_m)$, by \cite[Propositions 1.1.36 and
1.1.39]{monalg-rev}, the vertices of ${\rm NP}(I)$ are contained 
in the set $\{v_1,\ldots,v_m\}$. 
\end{proof}

The following uniform containment theorem is a major result in the 
containment problem
comparing symbolic and ordinary powers of ideals  
\cite{symbolic-powers-survey,Francisco-TAMS,Grifo-etal}, and it is
related to the resurgence theory of ideals
\cite{resurgence,DiPasquale-Drabkin,intclos,asymptotic-resurgence}, see below. 

\begin{theorem}{\rm(Hochster and Huneke
\cite[Theorem~1.1]{Hochster-Huneke})}
\label{HH} Let $S$ be a Noetherian regular ring containing a field, let $I$ be any
ideal of $S$, and let $h$ be the largest height of any associated
prime of $I$. Then, $I^{(hn)}\subset I^n$ for all positive integers
$n$. 
\end{theorem}

The resurgence and asymptotic resurgence of ideals were introduced 
by Bocci and Harbourne \cite{resurgence} and by Guardo, Harbourne and Van
Tuyl \cite{asymptotic-resurgence}, respectively. The resurgence of an ideal
relative to the integral closure was introduced by DiPasquale, Francisco, Mermin and
Schweig \cite{Francisco-TAMS}. The \textit{resurgence},
\textit{asymptotic resurgence}, and \textit{ic-resurgence} of an ideal
$I$ are given
by  
\begin{align*}
&\rho(I):=\sup \left\{{n}/{r}
\ \big|\   
I^{(n)}\not\subset I^r\right\},\\ 
&\ \ \ \ \ \widehat{\rho}(I):=\sup
\left \{ {n}/{r}
\ \big|\  
I^{(nt)}\not\subset I^{rt} \text{ for all } t\gg 0\right\},\\ 
&\ \ \ \ \ \ \ \ \ \ \ \rho_{ic}(I):=\sup
\left \{ {n}/{r}
\ \big|\  
I^{(n)}\not\subset\overline{I^{r}}\right\},\
\text{respectively}.
\end{align*}
\quad In general, $h\geq{\rho}(I)\geq {\rho}_{ic}(I)$
\cite[p.~66]{DiPasquale-Drabkin}, where $h$ is
the big height of $I$ that appears in Theorem~\ref{HH},  and $\widehat{\rho}(I)={\rho}_{ic}(I)$ 
\cite[Corollary~4.14]{Francisco-TAMS}. In particular
$\widehat{\rho}(I)$ is ${\rho}(I)$ if $I$ is normal. The resurgence of
$I$ is \textit{expected} if it is strictly less than 
the big height of $I$. Edge ideals of clutters of height at least $2$ 
have expected resurgence (see Theorem~\ref{DiB} below). If $I$ is a 
radical ideal in a regular ring, Grifo, Huneke and Mukundan
\cite{grifo-h-m} give a sufficient condition for the expected
resurgence of $I$.

There is a  
recent algorithm to compute the ic-resurgence $\rho_{ic}(I)$ of $I$ \cite{intclos} using linear-fractional
programming \cite{boyd}. The following result gives us an alternative
method for computing $\rho_{ic}(I)$ and a duality formula for 
the ic-resurgence of $I$ in terms of 
the vertices of $\mathcal{Q}(I)$ and $\mathcal{Q}(I^\vee)$, see also
the very recent preprint of H\`a, Jayanthan, Kumar and Nguyen 
\cite{Ha-Jayanthan-Kumar-Nguyen}. 

\begin{theorem}{\rm(Duality formula \cite[Theorem~3.7]{ic-resurgence})}\label{duality-ic-resurgence} 
If $I$ is the edge ideal of $\mathcal{C}$, then 
$$
\frac{1}{\rho_{ic}(I)}=\min\left\{\langle u,v\rangle\mid u\in V(\mathcal{Q}(I)),\,
v\in V(\mathcal{Q}(I^\vee))\right\},
$$
where $V(\mathcal{Q}(I))$ is the vertex
set of $\mathcal{Q}(I)$. In particular,
$\rho_{ic}(I)=\rho_{ic}(I^\vee)$ and $\rho_{ic}(I)\in\mathbb{Q}$.
\end{theorem}

The ic-resurgence of $I$ classifies the integrality of  
$\mathcal{Q}(I)$ because $\rho_{ic}(I)\geq 1$ with equality if and only if
$\mathcal{Q}(I)$ is integral \cite{ic-resurgence}. 

\begin{proposition}{\rm(\cite[Proposition~4.5]{Seceleanu-packing},
\cite[Theorem~7.8]{intclos}, \cite{ic-resurgence})}\label{qi-integral} 
The following conditions are equivalent.
\begin{enumerate}
\item[(a)] $\mathcal{Q}(I^\vee)$ is integral. {\rm(b)} ${\rm
NP}(I)=\mathcal{Q}(I^\vee)$. {\rm(c)} $\overline{I^n}=I^{(n)}$
for all $n\geq 1$. {\rm(d)} $\rho_{ic}(I)=1$.  
\end{enumerate}
\end{proposition}

Then, as a consequence of the duality formula for the ic-resurgence 
(Theorem~\ref{duality-ic-resurgence}), one recovers
the following fact.

\begin{theorem}\cite[Theorem~1.17]{cornu-book}\label{duality-qi} 
$\mathcal{Q}(I)$ is integral if and only if $\mathcal{Q}(I^\vee)$ is
integral.
\end{theorem}

\begin{definition} A monomial ideal $I$ is called a \textit{Simis
ideal} if $I^n=I^{(n)}$ for all $n\geq 1$.
\end{definition}

We give a short proof of the following theorem using the last two
results. This theorem together with \textit{Normaliz}
\cite{normaliz2}, 
can be used effectively as a test 
to determine whether or not an edge ideal is a Simis ideal 
\cite[Remark~3.5]{clutters}, \cite[Example 3.6]{clutters}. A main problem in this
area is to classify Simis ideals for  
non-squarefree monomial ideals. For some progress related to this 
problem, see 
\cite{Banerjee-etal,Cooper-symbolic,Das-K,cm-oriented-trees,
weighted-symbolic,Nasernejad-ars,Mandal-Pradhan,Mandal-Pradhan1}.

\begin{theorem}{\rm(Gitler, Valencia, -\,,
\cite[Theorem~3.4]{clutters})}\label{ntf-normal-i} Let $I$ be the
edge ideal of a clutter. The following two conditions are equivalent. 
\begin{itemize}
\item [\rm(a)] $I$ is a Simis ideal, that is, $I^n=I^{(n)}$ for all
$n\geq 1$. 
\item[\rm(b)] The Rees algebra $S[Iz]$ of $I$ is normal and
$\mathcal{Q}(I)$ is an integral polyhedron.
\end{itemize}
\end{theorem}

\begin{proof} (a)$\Rightarrow$(b) Taking integral closures in
$I^n=I^{(n)}$ and using that $I^{(n)}$ is complete, we get that
$\overline{I^n}=I^{(n)}=I^n$ for all $n\geq 1$. Thus, $S[Iz]$ is normal. As 
$\overline{I^n}=I^n=I^{(n)}$ for all $n\geq 1$, by
Proposition~\ref{qi-integral}, $\mathcal{Q}(I^\vee)$ is
integral. Hence, by Theorem~\ref{duality-qi}, $\mathcal{Q}(I)$ is
integral.

(a)$\Leftarrow$(b) As $\mathcal{Q}(I)$ is
integral, by Proposition~\ref{qi-integral} and
Theorem~\ref{duality-qi}, we get that $\overline{I^n}=I^{(n)}$ for all
$n\geq 1$. Thus, by the normality of $I$, $I^n=I^{(n)}$ for all
$n\geq 1$.
\end{proof}

\begin{theorem}{\rm(DiPasquale and Drabkin \cite[Corollary
4.20]{DiPasquale-Drabkin}\label{DiB})} 
If $I$ is a squarefree monomial ideal in $S=K[t_1,\ldots,t_s]$ with
big height $h\geq 2$, then $\rho(I)<h$ and $\rho_{ic}(I)\leq
h-\frac{1}{s}$.
\end{theorem}

Let $h$ and $h^\vee$ be the big heights of $I$ and $I^\vee$,
respectively. Then, by Theorem~\ref{DiB}, one has
$$
\rho_{ic}(I)\leq h-\frac{1}{s}\ \mbox{ and }\ \rho_{ic}(I^\vee)\leq
h^\vee-\frac{1}{s},
$$
and, by Theorem~\ref{duality-ic-resurgence}, 
we get $\rho_{ic}(I)<r:=\min\{h,h^\vee\}$. Thus, 
$I^{(nr)}\subset\overline{I^n}$ for all $n\geq 1$.

The \textit{containment problem}  
for $I$ is to determine for which $r$ and $n$ the containment
$I^{(n)}\subset I^r$ holds. This problem appeared in a paper of Schenzel
 \cite[p.~144]{Schenzel}, who ask whether the function 
\begin{equation}\label{Schenzel-function}
f(r):=\min\{n\geq 1 \mid I^{(n)}\subset I^r\},\quad r=1,2,\ldots,
\end{equation}
becomes a polynomial for large $r$ assuming that $I$ is a prime ideal
of a Noetherian ring $S$ such that the filtrations $\{I^{(n)}\}$ and $\{I^n\}$ define equivalent
topologies.

\quad By Theorem~\ref{HH}, this function is well-defined and $r\leq
f(r)\leq rh$, where $h={\rm bight}(I)$. It can be 
computed noticing that $f(r)$ is the function ``containmentProblem$(I,r)$'' implemented
in \textit{Macaulay}$2$ \cite{mac2} by Drabkin, Grifo, 
Seceleanu and Stone \cite{Grifo-etal}.  
The resurgence measures to what extent
the containment hold because $I^{(n)}\subset I^r$ if
$n/r>\rho(I)$.

\begin{lemma}\label{may15-24} If $k$ is a positive integer such that 
$I^{(kn)}\subset \overline{I^n}$ $($resp. $I^{(kn)}\subset{I^n}$$)$
for all $n\geq 1$, then $\rho_{ic}(I)\leq k$ $($resp. $\rho(I)\leq k$$)$. 
\end{lemma}

\begin{proof} Assume that $I^{(n)}\not\subset\overline{I^r}$ (resp.
$I^{(n)}\not\subset{I^r}$). We claim that $k>n/r$. We argue by
contradiction assuming that $k\leq n/r$, that is, $kr\leq n$. Hence,
$$ 
I^{(n)}\subset I^{(kr)}\subset\overline{I^r}\ \ (\mbox{resp.}\
I^{(n)}\subset I^{(kr)}\subset{I^r}),
$$
a contradiction. Thus, $k>n/r$, and by taking the supremum over all
$n/r$ such that $I^{(n)}\not\subset\overline{I^r}$ (resp.
$I^{(n)}\not\subset{I^r}$), we obtain the inequality $k\geq\rho_{ic}(I)$ (resp. $k\geq\rho(I)$). 
\end{proof}

\begin{lemma}\cite[Lemma~4.12]{Francisco-TAMS}\label{jan16-25} If $I$ is an ideal,
then the following hold: 
\begin{enumerate}
    \item[\rm(1)] If $I^{(n)}\not\subset\overline{I^r}$, then $\frac{n}{r}<\rho_{ic}(I)$. 
    \item[\rm(2)] If $\frac{n}{r}<\rho_{ic}(I)$, then
    $I^{(nt)}\not\subset\overline{I^{rt}}$ for all $t\gg 0$.  
\end{enumerate}    
\end{lemma}

\begin{remark}\label{jan18-25} By
Lemma~\ref{jan16-25}(1), the ic-resurgence of a squarefree monomial ideal
$I$ cannot the attained in the set $
L=\left. \left\{{n}/{r}
\ \right|\, 
I^{(n)}\not\subset \overline{I^{r}}\right\}$, that is,
$\rho_{ic}(I)\notin L$.
\end{remark}

\begin{proposition}\label{grisalde} Let $I$ be the edge ideal of a clutter with $s$
vertices. Then, $\rho(I)=1$ if and only if $\mathcal{Q}(I)$ is
integral and $I^{(r+1)}\subset I^r$ for $r=1,\ldots,s-1$.
\end{proposition}

\begin{proof} By \cite[Corollary~4.17]{Francisco-TAMS},
$\rho(I)=1$ if and only if $\rho_{ic}(I)=1$ and
$\overline{I^{r+1}}\subset I^r$ for all $r\geq 1$. Then, by
Proposition~\ref{qi-integral}, $\rho(I)=1$ if and only if
the covering polyhedron $\mathcal{Q}(I)$ of $I$ is integral and $I^{(r+1)}\subset I^r$ for all
$r\geq 1$. Hence, we need only show the following implication:

$\Leftarrow$) By Theorem~\ref{may3-23} and Proposition~\ref{qi-integral},
$I^{(r+1)}=I I^{(r)}$ for all $r\geq s-1$. Thus, it suffices to show
that $I^{(r+1)}\subset I^r$ for all $r\geq s$ and this follows by
induction of $r\geq s$.  
\end{proof}

For normal ideals, the ic-resurgence essentially minimizes the uniform containment
theorem of Hochster and Huneke. 
    
\begin{proposition}\label{vila-gonzalo}\cite{thesis-Gonzalo} {\rm (a)} $\lceil\rho_{ic}(I)\rceil=
\min\{h\in\mathbb{N}_+\mid I^{(hn)}\subset \overline{I^n}\, \mbox{ for
all }n\geq 1\}$. 

{\rm (b)} $\lceil\rho(I)\rceil\leq 
\min\{h\in\mathbb{N}_+\mid I^{(hn)}\subset {I^n}\, \mbox{ for
all }n\geq 1\}$. 
\end{proposition}

\begin{proof} (a) We set $h_0:=\lceil\rho_{ic}(I)\rceil$ and 
$h_1:=\min\{h\in\mathbb{N}_+\mid I^{(hn)}\subset \overline{I^n}\, \mbox{ for
all }n\geq 1\}$. Applying Lemma~\ref{may15-24}, one has
$\rho_{ic}(I)\leq h_1$ and $h_0\leq h_1$. To show the inequality
$h_0\geq h_1$ it suffices to show that $I^{(h_0n)}\subset\overline{I^n}$
for all $n\geq 1$. We argue by contradiction assuming that
$I^{(h_0n)}\not\subset\overline{I^n}$ for some $n\geq 1$. Then, by
Lemma~\ref{jan16-25}(1),
$h_0=\frac{h_0n}{n}<\rho_{ic}(I)\leq\lceil\rho_{ic}(I)\rceil=h_0$, a
contradiction.

(b) The inequality follows by applying Lemma~\ref{may15-24}. 
\end{proof}

\begin{example}\cite{thesis-Gonzalo}\label{jan23-25} 
Let $I=(t_1t_2t_5,\,  t_1t_3t_4,\,  t_2t_3t_6,\,  t_4t_5t_6)$ be the edge ideal
of the clutter $\mathcal{Q}_6$ of \cite[Example~1.9]{cornu-book} and let $f$ be
the Schenzel function of Eq.~\eqref{Schenzel-function}. As 
$\mathcal{Q}(I)$ is integral and $f(1)=1$, $f(r)=r+1$ for
$r=2,\ldots,6$, by Proposition~\ref{grisalde}, one has $\rho(I)=1$.
Note that $I^{(2)}\not\subset I^2$ and $\rho(I)=1\in
L'=\left. \left\{{n}/{r}
\ \right|\, 
I^{(n)}\not\subset{I^{r}}\right\}$ (cf. Remark~\ref{jan18-25}). The
inequality in Proposition~\ref{vila-gonzalo}(b) is strict because
$\rho(I)=1$ and $I^{2}\subsetneq I^{(2)}$. 
\end{example}

We now connect the equality of ordinary and symbolic powers of
edge ideals of clutters 
with the result of Huneke, Simis
and Vasconcelos (see Theorem~\ref{reduced-normal-cones}) that classifies reduced associated graded rings.

\begin{theorem}\cite{Lisboar,clutters} ${\rm gr}_I(S)$ is reduced if and only $I^n=I^{(n)}$
for all $n\geq 1$.  
\end{theorem}

\begin{proof} By Theorem~\ref{ntf-normal-i}, $I^n=I^{(n)}$
for all $n\geq 1$ if and only if $S[Iz]$ is normal and
$\mathcal{Q}(I)$ is integral. Then, by Theorem~\ref{irred-rep-rc}, 
$I^n=I^{(n)}$ for all $n\geq 1$ if and only if $S[Iz]$ is normal and 
$r=p$. Finally, by Theorem~\ref{effcred}, $I^n=I^{(n)}$
for all $n\geq 1$ if and only if ${\rm gr}_I(S)$ is reduced.
\end{proof}

A clutter $\mathcal{C}$, with incidence matrix
$A$, has the {\it max-flow
min-cut\/} (MFMC) 
property if both sides 
of the LP-duality equation
\begin{equation}\label{jun6-2-03}
{\rm min}\{\langle \alpha,x\rangle \vert\, x\geq 0; xA\geq{1}\}=
{\rm max}\{\langle y,{1}\rangle \vert\, y\geq 0; Ay\leq\alpha\} 
\end{equation}
have integral optimum solutions $x,y$ for each nonnegative integral
vector $\alpha$ \cite[p.~3]{cornu-book}.  

A breakthrough in the area of edge ideals is the following theorem relating 
symbolic powers and the max-flow min-cut property of integer
programming, creating another bridge between commutative algebra and optimization
problems. 

\begin{theorem}\label{ntf-char}{\rm 
(\cite[Corollary~3.14]{clutters}, cf. \cite[Theorem
1.4]{hhtz})} If $I$ is the edge ideal of
a clutter $\mathcal{C}$, then $I^n=I^{(n)}$ for all
$n\geq 1$ if and only if $\mathcal{C}$ has the max-flow min-cut
property.
\end{theorem}

The following notions come from combinatorial optimization \cite{Schr2}. For
$t_i\in V(\mathcal{C})$, the {\it
contraction} 
$\mathcal{C}/t_i$ 
and {\it deletion} 
$\mathcal{C}\setminus{t_i}$ are the
clutters constructed as follows: both have $V(\mathcal{C})\setminus\{t_i\}$ as
vertex set, $E(\mathcal{C}/t_i)$ is the set of minimal elements 
of $\{e\setminus\{t_i\}\vert\, e\in E(\mathcal{C}\}$, minimal with
respect to inclusion, and 
$E(\mathcal{C}\setminus{t_i})$ is the set 
$\{e\vert\, t_i\notin e\in E(\mathcal{C})\}$. A {\it minor} of a clutter $\mathcal{C}$
is a clutter obtained from $\mathcal{C}$ by a sequence of deletions 
and contractions in any order. The clutter $\mathcal{C}$ is considered
a minor by convention.

The \textit{covering number} of a clutter $\mathcal{C}$, denoted
$\alpha_0(\mathcal{C})$, is the minimum cardinality of a vertex cover
of $\mathcal{C}$. Note that $\alpha_0(\mathcal{C})$ is the height of
$I(\mathcal{C})$. A set of pairwise disjoint edges of $\mathcal{C}$ is called
a {\it matching}. The \textit{matching number} of $\mathcal{C}$,
denoted $\beta_1(\mathcal{C})$, is the maximum
cardinality of a matching of $\mathcal{C}$. We say that a clutter $\mathcal{C}$ is
\textit{K\H{o}nig} or has the {\it K\H{o}nig property\/} if
the covering number $\alpha_0(\mathcal{C})$ is 
the matching number $\beta_1(\mathcal{C})$. 

\begin{definition}
A clutter $\mathcal C$ satisfies the {\it packing property\/} if
$\alpha_0({\mathcal C}')=\beta_1({\mathcal C}')$  
for any minor ${\mathcal C}'$ of $\mathcal C$; that is, all minors of
$\mathcal{C}$ satisfy the K\H{o}nig property. 
\end{definition}

\begin{theorem}{\rm (A. Lehman \cite{lehman}, 
\cite[Theorem~1.8]{cornu-book})}\label{lehman} If\, a clutter $\mathcal C$ has the
packing property, then $\mathcal{Q}(I(\mathcal{C}))$ is 
integral. 
\end{theorem}

The converse of this result is not true. A famous example is the clutter
$\mathcal{Q}_6$ of \cite[Example~1.9]{cornu-book} (see
Example~\ref{jan23-25}). It is not K\H{o}nig
and $\mathcal{Q}(I(\mathcal{Q}_6))$ is integral.  

\begin{proposition}{\rm(Seymour \cite{seymour},
\cite[Proposition~1.36]{cornu-book})}
\label{aug25-06-1}
If a clutter
$\mathcal C$ has the max-flow
min-cut property,  
then $\mathcal C$ has the packing property. 
\end{proposition}

To the best of our knowledge, the converse is still an unsolved conjecture. 
 
\begin{conjecture}{\rm(The packing problem, Conforti--Cornu\'ejols
\cite{CC}, \cite[Conjecture~1.6]{cornu-book})}
\label{conforti-cornuejols1}\rm
A clutter $\mathcal C$ has the
packing property if and only if 
$\mathcal C$ has the max-flow min-cut property.
\end{conjecture}

As we now explain, the packing problem has been brought into 
commutative algebra by Gitler, Valencia, and the second author
\cite{clutters}. Following the algebraic approach of \cite{reesclu,clutters}, some 
advances on the packing problem and some related conjectures are
given in \cite{mfmc,chordal}, see also \cite{cm-mfmc,graphs}. Consequences and recent discussions of 
the packing problem can be found in
\cite{Alilooee-Benerjee,Seceleanu-packing,symbolic-powers-survey,Grifo-etal,HaM,edge-ideals,monalg-rev}. 

The edge ideals of a deletion and a contraction have a nice algebraic
interpretation. For $t_i\in V(\mathcal{C})$, define the {\it
contraction\/} and {\it
deletion} of  
$I=I(\mathcal{C})$ as the ideals: 
$$
(I\colon t_i)\ \quad\mbox{ and }\ \quad I^c=I\cap
K[t_1,\ldots,\widehat{t}_i,\ldots,t_s],
$$
respectively. The clutter associated to 
the squarefree monomial ideal $(I\colon t_i)$ (resp. $I^c$) is the 
contraction $\mathcal{C}/t_i$  (resp. deletion $\mathcal{C}\setminus
t_i$), i.e., the edge ideals of $\mathcal{C}/t_i$  and
$\mathcal{C}\setminus t_i$ are $(I\colon t_i)$  and $I^{c}$, 
respectively. This indicates how to define the notion
of minor for edge ideals. 

\begin{definition}\cite[Definition~3.8]{clutters}\label{minor-def} A {\it minor\/} of $I$
is a proper ideal $(0)\subsetneq I'\subsetneq S$ 
obtained from $I$ by making any sequence of the $t_i$-variables equal to
$1$ or $0$ in $\mathcal{G}(I)=\{t^{v_1},\ldots,t^{v_m}\}$, and then
taking the ideal $I'$ generated by the resulting monomials. The ideal
$I$ is 
considered a minor by convention. 
\end{definition}

Minors of $\mathcal{C}$ correspond to minors of $I=I(\mathcal{C})$ and
vice versa. The covering and matching number of $\mathcal{C}$ 
can be expressed algebraically as \cite{reesclu,clutters}: $\alpha_0(\mathcal{C})$ is the
height ${\rm ht}(I(\mathcal{C}))$ of $I(\mathcal{C})$, and $\beta_1(\mathcal{C})$ is the
{monomial grade} ${\rm mgrade}(I(\mathcal{C}))$ of $I(\mathcal{C})$, that is, the maximum cardinality of a regular 
sequence in $I(\mathcal{C})$ consisting of monomials.

We say that $I$ is \textit{K\H{o}nig} or has the \textit{K\H{o}nig
property} if ${\rm ht}(I)={\rm mgrade}(I)$, and we say that 
$I$ satisfies the \textit{packing
property} if each minor $I'$ of $I$ satisfies the K\H{o}nig property.

We come to the commutative algebra translation of the packing problem.

\begin{conjecture}{\rm(The packing problem,
\cite[Conjecture~3.10]{clutters},
\cite[Theorem~4.6]{reesclu})} If $I$ is a squarefree monomial 
ideal, then $I^n=I^{(n)}$ for all $n\geq 1$  
if and only if $I$ has the packing property.
\end{conjecture}

Using Theorems~\ref{ntf-normal-i} and \ref{lehman} this conjecture reduces to:  

\begin{conjecture}\label{con-cor-vila} If $I$ has the packing
property, then the Rees algebra $S[Iz]$ is normal.
\end{conjecture} 

The packing problems holds for graphs, see
Theorem~\ref{graphs-ntf} below.

\begin{proposition}\cite[Proposition~4.10 and 4.11]{reesclu} 
\label{balanced-mfmc-blocker}
If $\mathcal{C}$ is a balanced clutter,  
then the Rees algebra and the symbolic Rees algebra of 
both $I(\mathcal{C})$ and $I_c(\mathcal{C})$ are equal, that is,
\begin{enumerate}
\item[(i)] $S[I({\mathcal C})z]=\mathcal{R}_s(I({\mathcal
C}))$\quad and\quad {\rm (ii)} $S[I_c({\mathcal C})z]=
\mathcal{R}_s(I_c({\mathcal C}))$.
\end{enumerate}
\end{proposition}

Part (i) and (ii) of this result were first shown for bipartite graphs in 
\cite[Theorem~5.9]{ITG} (Theorem~\ref{ntf-bipartite}) and
\cite[Corollary~2.6]{alexdual},
respectively, and later generalized to balanced 
clutters in \cite[Propositions~4.10 and 4.11]{reesclu}. 

\begin{corollary}{\rm(Faridi \cite[Theorem~5.3]{faridijct})} If 
$\mathcal{C}$ is the clutter of facets of a simplicial forest, then
$\mathcal{C}$ has the K\H{o}nig property.   
\end{corollary}

\begin{proof} By Theorem~\ref{soleyman-zheng-char}, $\mathcal{C}$ is
totally balanced.    
Hence, by Theorem~\ref{ntf-char} and
Proposition~\ref{balanced-mfmc-blocker}, 
$\mathcal{C}$ has the max-flow min-cut property. Therefore, by
Proposition~\ref{aug25-06-1}, all minors of $\mathcal{C}$ 
satisfy the K\H{o}nig property. In particular, $\mathcal{C}$ has the
K\H{o}nig property.
\end{proof}

\begin{corollary}\label{dec26-02} 
If $A$ is totally unimodular, that is, each 
$i\times i$ subdeterminant of 
$A$ is $0$ or $\pm 1$ for all $i\geq 1$, then 
$\mathcal{C}$ has the max-flow min-cut property. 
\end{corollary}

\begin{proof} If $A$ is totally unimodular, then $\mathcal{C}$ is
balanced. Thus, we can apply Proposition~\ref{balanced-mfmc-blocker}.
\end{proof}

\begin{theorem}\cite[Theorem~1.3]{CGM}\label{dyadic-mfmc} If $\mathcal{Q}(I)$ is integral
and $\mathcal{C}$ is dyadic, that is, $|C\cap B|\leq2$ 
for all $C\in E(\mathcal{C})$ and $B\in E(\mathcal{C}^\vee)$, then $\mathcal{C}$ has the max-flow min-cut
property.
\end{theorem}

\begin{theorem}{\rm(Gitler, Reyes, -\,, \cite{alexdual})}\label{ntf-dual-bipartite}
Let $G$ be a graph and let $I_c(G)$ be the ideal of covers of $G$. The following 
conditions are equivalent:
\begin{enumerate}
\item[(i)] $G$ is a bipartite graph.\quad {\rm(ii)} $I_c(G)^{(k)}=I_c(G)^k$ for all $k\geq 1$.
\end{enumerate}
\end{theorem}

\begin{proof} (i) $\Rightarrow$ (ii) A bipartite graph is balanced.
Thus, by Proposition~\ref{balanced-mfmc-blocker}(ii), $I_c(G)$ is a Simis
ideal and the proof is complete.

(ii) $\Rightarrow$ (i) By \cite[Theorem~3.4]{clutters},
Theorem~\ref{ntf-normal-i}, the covering
polyhedron of the ideal $I_c(G)$ is integral. Then, by 
\cite[Theorem~1.17]{cornu-book}, Theorem~\ref{duality-qi}, the covering polyhedron 
of the edge ideal $I(G)$ is also integral. As $G$ is dyadic, by
Theorems~\ref{ntf-char} and \ref{dyadic-mfmc}, $I(G)$ is a Simis
ideal. Thus, by Theorem~\ref{ntf-bipartite}, $G$ is bipartite. 
\end{proof}

\begin{theorem}\cite{CGM,alexdual,reesclu,ITG}\label{graphs-ntf} If $G$ is
a graph and $I=I(G)$ is its edge ideal, then the following  
conditions are equivalent{\rm :}
\begin{enumerate}
\item[(a)] $I^n=I^{(n)}$ for all $n\geq 1$.\quad {\rm (b)}
$I_c(G)^n=I_c(G)^{(n)}$ for all $n\geq 1$.
\item[(c)] $G$ is bipartite.\quad {\rm (d)} $\mathcal{Q}(I)$ is
integral. 
\item[(e)] $G$ has the packing property.\quad {\rm (f)} ${\rm
gr}_I(S)$ is reduced. 
\end{enumerate}
\end{theorem}

The \textit{Simis cone} of $I$ is the rational polyhedral cone:
\begin{equation}\label{simis-cone}
{\rm SC}(I)=H_{e_1}^+\cap \cdots\cap H_{e_{s+1}}^+
\cap H_{\ell_1}^+\cap\cdots\cap H_{\ell_r}^+,
\end{equation}
where $\ell_i=(u_i,1)$ for $i=1,\ldots,r$ and $\{u_1,\ldots,u_r\}$ are
the integral vertices of $\mathcal{Q}(I)$, see Eq.~\eqref{jan22-25}. 
The term Simis cone was coined in \cite{Lisboar} to recognize the
pioneering work
of Aron Simis on symbolic powers of monomial ideals
\cite{aron-hoyos}. 
This notion has been extended to study the symbolic Rees algebras of 
filtrations associated to covering polyhedra \cite{intclos}. 
The Simis cone is a finitely generated rational cone and there is
a finite set of integral vectors $\mathcal{H}\subset\mathbb{Z}^{s+1}$ 
such that 
$$\mathbb{Z}^{s+1}\cap \mathbb{R}_+\mathcal{H}=\mathbb{N}\mathcal{H}\
\mbox{ and }\ {\rm RC}(I)=\mathbb{R}_+\mathcal{H}.
$$ 
\quad The set $\mathcal{H}$ is called a {\it Hilbert basis\/} of
${\rm SC}(I)$. 
The Simis cone can be used to compute the symbolic Rees algebra
$\mathcal{R}_s(I)$ of $I$ using \textit{Normaliz} \cite{normaliz2}. 

\begin{theorem}\cite[Theorem~3.5]{Lisboar}\label{vila-yoshino} Let ${\mathcal H}\subset
\mathbb{N}^{s+1}$ 
be a Hilbert basis of ${\rm SC}(I)$. 
If $K[\mathbb{N}{\mathcal H}]$ is the semigroup 
ring of $\mathbb{N}{\mathcal H}$, then $\mathcal{R}_s(I)=K[\mathbb{N}{\mathcal
H}]$.
\end{theorem}

This result shows that the symbolic Rees algebra of $I$ is a 
normal $K$-algebra generated by the finite set of monomials that
corresponds to the points of $\mathcal{H}$. One can also compute generators 
for the symbolic Rees algebra of $I$ using the algorithm in the proof 
of \cite[Theorem 1.1]{cover-algebras}. The minimal generators of the
symbolic Rees algebra of the edge ideal of a clutter are in one to one
correspondence with the indecomposable parallelizations of the clutter
\cite[Example~3.6]{symboli}.  

For use below recall that a graph $G$ is called {\it unmixed\/} or
\textit{well-covered} if all maximal independent sets of $G$ have the
same cardinality. The following result links an algebraic property
of a blowup algebra with a graph theoretical property.    

\begin{proposition}\cite[Corollary~4.3]{roundp}\label{extended-gorenstein} Let $G$ be a connected
bipartite graph and let 
$I=I(G)$ be its edge ideal. Then, the extended Rees algebra 
$S[Iz,z^{-1}]$ is a Gorenstein standard $K$-algebra if and only if
$G$ is unmixed. 
\end{proposition}

\section{The Vasconcelos formula for the generalized Hamming
weights}\label{coding-theory-section}

Let $S=K[t_1,\ldots,t_s]=\bigoplus_{d=0}^{\infty} S_d$ be a polynomial ring over
a field $K$ with the standard grading and let $I\neq(0)$ be a graded ideal
of $S$. Given $d,r\in\mathbb{N}_+$, let $\mathcal{F}_{d,r}$ be the set:
$$
\mathcal{F}_{d,r}:=\{\, \{f_i\}_{i=1}^r\subset S_d\mid 
\{\overline{f}_i\}_{i=1}^r\, \mbox{linearly
independent over }K, (I\colon(\{f_i\}_{i=1}^r))\neq I\}, 
$$
where $\overline{f}=f+I$ is the class of $f$ modulo $I$. 
We denote the {\it degree\/} or \textit{multiplicity} of $S/I$ by $\deg(S/I)$. The function 
$\delta_I\colon \mathbb{N}_+\times\mathbb{N}_+\rightarrow \mathbb{Z}$ given by 
$$
\delta_I(d,r):=\left\{\begin{array}{ll}\deg(S/I)-\max\{\deg(S/(I,F))\mid
F\in\mathcal{F}_{d,r}\}&\mbox{if }\mathcal{F}_{d,r}\neq\emptyset,\\
\deg(S/I)&\mbox{if\ }\mathcal{F}_{d,r}=\emptyset,
\end{array}\right.$$
is called the {\it generalized minimum distance function\/} of $I$
\cite{min-dis-generalized,rth-footprint}. If $r=1$, one recovers 
the minimum distance function of $I$ studied in
\cite{hilbert-min-dis,footprint-ci,Bounds-MinDis}.     
To compute $\delta_I(d,r)$ is a
difficult problem, but there are footprint lower bounds 
for $\delta_I(d,r)$ which 
are easier to compute using Gr\"obner basis
\cite{min-dis-generalized,rth-footprint}. In certain cases (e.g., 
for complete intersection monomial ideals of dimension $\geq 1$ and
for vanishing ideals of Cartesian products) the
footprint gives the exact value of $\delta_I(d,r)$
\cite[Corollary~4.4]{min-dis-ci}, \cite[Theorem~5.5]{footprint-ci}. 

The definition of $\delta_I(d,r)$ was motivated by the notion of 
generalized Hamming weight of a linear code. For convenience we
recall this notion. Let $K=\mathbb{F}_q$ be a finite field and let $C$ be a $[m,k]$ {\it linear
code} of {\it length} $m$ and {\it dimension} $k$, 
that is, $C$ is a linear subspace of $K^m$ with $k=\dim_K(C)$. Let $1\leq r\leq k$ be an integer.  
Given a subcode $D$ of $C$ (that is, $D$ is a linear subspace of $C$),
the {\it support\/} $\chi(D)$ of $D$ is the set of non-zero positions of $D$, that is,  
$$
\chi(D):=\{i\mid \exists\, (a_1,\ldots,a_m)\in D,\, a_i\neq 0\}.
$$
\quad The $r$-th {\it generalized Hamming weight\/} of $C$, denoted
$\delta_r(C)$, is the size of the smallest support of an
$r$-dimensional subcode. If $r=1$, $\delta_r(C)$ is the minimum
distance of $C$. The sequence $\delta_1(C), \ldots, \delta_k(C)$ is the
\textit{weight hierarchy} of $C$ and one has 
$\delta_1(C)<\cdots<\delta_k(C)$ \cite{wei}. The generalized Hamming
weights of a linear code are parameters 
of interest in many applications 
\cite{rth-footprint,Pellikaan,Johnsen,olaya,schaathun-willems,tsfasman,wei,Yang}
and they have been nicely related to  
the minimal graded free resolution of the ideal of cocircuits of the matroid
of a linear code \cite{JohVer,Johnsen}, to the nullity
function of the dual matroid of a linear code \cite{wei}, and to the
enumerative combinatorics of linear codes 
\cite{Britz,Klove-1992,MacWilliams-Sloane}.
Because of this, 
their study has attracted considerable attention, 
but determining them is in general a
difficult problem. The notion of generalized Hamming weight was
introduced by Helleseth, Kl{\o}ve and Mykkeltveit 
 \cite{helleseth} and by Wei \cite{wei}. 

The minimum distance of projective Reed-Muller-type codes has been
studied using Gr\"obner bases and commutative algebra techniques; see
\cite{carvalho-lopez-lopez,geil-thomsen,cartesian-codes,hilbert-min-dis,algcodes}
and references  
therein. These techniques were extended in
\cite{min-dis-generalized,rth-footprint,duality-projective} to study the $r$-th
generalized Hamming weight of projective 
Reed-Muller-type codes. These linear codes are constructed
as follows. 

Let $K=\mathbb{F}_q$ be a finite field with $q$ elements,
let $\mathbb{P}^{s-1}$ be a projective space over 
$K$, and let $\mathbb{X}$ be a subset of
$\mathbb{P}^{s-1}$. The {\it vanishing ideal\/} of
$\mathbb{X}$, denoted $I(\mathbb{X})$,  is the ideal of $S$ 
generated by the homogeneous polynomials that vanish at all points of
$\mathbb{X}$. The Hilbert function of $S/I(\mathbb{X})$ is denoted by
$H_\mathbb{X}(d)$ or $H_{I(\mathbb{X})}(d)$. We can write
$\mathbb{X}=\{[P_1],\ldots,[P_m]\}\subset\mathbb{P}^{s-1}$ 
with $m=|\mathbb{X}|$. Here we assume that the first non-zero entry of
each $[P_i]$ is $1$. In the special case that $\mathbb{X}$ has the
form $[X\times\{1\}]$ for some $X\subset\mathbb{F}_q^{s-1}$, we
do not make this assumption.  

Fix a degree $d\geq 1$. There is an \textit{evaluation} $K$-linear map given by  
\begin{equation*}
{\rm ev}_d\colon S_d\rightarrow K^{m},\ \ \ \ \ 
f\mapsto
\left(f(P_1),\ldots,f(P_m)\right).
\end{equation*}
\quad The image of $S_d$ under ${\rm ev}_d$, denoted by  $C_\mathbb{X}(d)$, is
called a {\it projective Reed-Muller-type code\/} of
degree $d$ on $\mathbb{X}$ \cite{duursma-renteria-tapia,GRT}. The
points in $\mathbb{X}$ are often called evaluation points in the
algebraic coding context. 
The {\it parameters} of the linear
code $C_\mathbb{X}(d)$ are:
\begin{itemize}
\item[(a)] {\it length\/}: $|\mathbb{X}|$,
\item[(b)] {\it dimension\/}: $\dim_K C_\mathbb{X}(d)$,
\item[(c)] $r$-th {\it generalized Hamming weight\/}: 
$\delta_\mathbb{X}(d,r):=\delta_r(C_\mathbb{X}(d))$. 
\end{itemize}

If $r=1$, $\delta_\mathbb{X}(d,r)$ is the minimum distance of
$C_\mathbb{X}(d)$. The $r$-th generalized Hamming weight
$\delta_\mathbb{X}(d,r)$ of $C_\mathbb{X}(d)$ is equal to $\delta_{I(\mathbb{X})}(d,r)$
\cite{min-dis-generalized,rth-footprint}. Thus, generalized minimum
distance functions are a generalization of the
generalized Hamming weights of a linear code and this is one of the
main reasons to study
them. 

The \textit{Vasconcelos function} (v-function for short) of a graded
ideal $I\subset S$, denoted $\vartheta_I$,  is
the function $\vartheta_I\colon
\mathbb{N}_+\times\mathbb{N}_+\rightarrow \mathbb{N}$ given by  
\begin{equation}\label{vas-function}
\vartheta_I(d,r):=\begin{cases}\min\{\deg(S/(I\colon(F)))\mid
F\in\mathcal{F}_{d,r}\}&\mbox{if }\mathcal{F}_{d,r}\neq\emptyset,\\ 
\deg(S/I)&\mbox{if\ }\mathcal{F}_{d,r}=\emptyset.
\end{cases}
\end{equation}
\quad This function was suggested to us by Wolmer Vasconcelos while 
we were visiting him in 2015 as an alternate way to extend the 
generalized Hamming weights of Reed--Muller-type codes. 
It was shown in \cite{min-dis-generalized,rth-footprint} that 
$\vartheta_{I(\mathbb{X})}(d,r)$ is also equal to $\delta_\mathbb{X}(d,r)$
(Theorem~\ref{rth-min-dis-vi}). The functions $\delta_I$ and
$\vartheta_I$ are two abstract algebraic extensions of 
$\delta_\mathbb{X}(d,r)$ that gives us a tool to study 
generalized Hamming weights. These two functions can be computed, for
small examples, using the algorithms in \cite{coding-theory-package,min-dis-generalized,rth-footprint}. 

Given a graded ideal $I\subset S$ 
define its {\it zero set\/} relative to $\mathbb{X}$ as  
$$V_\mathbb{X}(I)=\left\{[\alpha]\in \mathbb{X}\mid  
f(\alpha)=0\,\,  
\mbox{ for all }f\in I\, \mbox{ homogeneous} \right\}.
$$ 
\begin{lemma}\cite{rth-footprint}\label{degree-formula-for-the-number-of-non-zeros}
Let $\mathbb{X}$ be a subset of 
$\mathbb{P}^{s-1}$ over a finite field $K$ and let $I(\mathbb{X})\subset S$ be its
vanishing ideal. If 
$F=\{f_1,\ldots,f_r\}$ is a set of homogeneous polynomials of
$S\setminus\{0\}$, then 
$$
|\mathbb{X}\setminus V_{\mathbb{X}}(F)|=\begin{cases}
\deg(S/(I(\mathbb{X})\colon(F)))&\mbox{if }\, (I(\mathbb{X})\colon(F))\neq
I(\mathbb{X}),\\ 
\deg(S/I(\mathbb{X}))&\mbox{if }\, (I(\mathbb{X})\colon (F))=I(\mathbb{X}).
\end{cases}
$$
\end{lemma}

\begin{lemma}\cite{rth-footprint}\label{degree-formula-for-the-number-of-zeros-proj}
Let $\mathbb{X}$ be a subset of 
$\mathbb{P}^{s-1}$ over a finite field $K$ and let $I(\mathbb{X})\subset S$ be its
vanishing ideal. If 
$F=\{f_1,\ldots,f_r\}$ is a set of homogeneous polynomials of
$S\setminus\{0\}$, then the number of points of $V_\mathbb{X}(F)$ is given by 
$$
|V_{\mathbb{X}}(F)|=\left\{
\begin{array}{cl}
\deg(S/(I(\mathbb{X}),F))&\mbox{if }\, (I(\mathbb{X})\colon(F))\neq
I(\mathbb{X}),\\ 
0&\mbox{if }\, (I(\mathbb{X})\colon(F))=I(\mathbb{X}).
\end{array}
\right.
$$
\end{lemma}

\begin{lemma}\cite{rth-footprint}\label{mar14-17} 
Let $\mathbb{X}=\{[P_1],\ldots,[P_m]\}$ be a subset of
$\mathbb{P}^{s-1}$. Then 
$$\delta_r(C_\mathbb{X}(d))=\min\{|\mathbb{X}\setminus
V_\mathbb{X}(F)|:\, F=\{f_i\}_{i=1}^r\subset S_d,\,
\{\overline{f}_i\}_{i=1}^r
\mbox{ linearly independent over } K\}.
$$
\end{lemma}

\begin{theorem}\cite{min-dis-generalized,rth-footprint}\label{rth-min-dis-vi} 
Let $K$ be a finite field and let $\mathbb{X}$ be a subset of
$\mathbb{P}^{s-1}$. If  $|\mathbb{X}|\geq 2$ and
$\delta_\mathbb{X}(d,r)$ is the $r$-th generalized Hamming weight of
$C_\mathbb{X}(d)$, then 
$$\delta_\mathbb{X}(d,r)=\delta_{I(\mathbb{X})}(d,r)=\vartheta_{I(\mathbb{X})}(d,r)\
\mbox{ for all }d\geq 1\mbox{ and }1\leq r\leq H_{I(\mathbb{X})}(d),$$
and $\delta_\mathbb{X}(d,r)=r$ for all $d\geq {\rm reg}(S/I(\mathbb{X}))$.
\end{theorem}

\begin{proof} If $\mathcal{F}_{d,r}=\emptyset$, then using
Lemmas~\ref{degree-formula-for-the-number-of-non-zeros}, 
\ref{degree-formula-for-the-number-of-zeros-proj}, and
\ref{mar14-17} we get that $\delta_\mathbb{X}(d,r)$, 
$\delta_{I(\mathbb{X})}(d,r)$, and $\vartheta_{I(\mathbb{X})}(d,r)$ are equal to
$\deg(S/I(\mathbb{X}))=|\mathbb{X}|$. Assume that
$\mathcal{F}_{d,r}\neq \emptyset$ and set $I=I(\mathbb{X})$. 
Using Lemma~\ref{mar14-17} and the formula for $V_\mathbb{X}(F)$ of 
Lemma~\ref{degree-formula-for-the-number-of-zeros-proj}, we
obtain
\begin{eqnarray*}
\delta_\mathbb{X}(d,r)&\stackrel{\tiny(\ref{mar14-17})}{=}&\min\{|\mathbb{X}\setminus
V_\mathbb{X}(F)|\colon F\in
\mathcal{F}_{d,r}\}
\stackrel{\tiny(\ref{degree-formula-for-the-number-of-zeros-proj})}{=}
|\mathbb{X}|-\max\{\deg(S/(I,F))\vert\,
F\in \mathcal{F}_{d,r}\}\\
&=&\deg(S/I)-\max\{\deg(S/(I,F))\vert\,
F\in \mathcal{F}_{d,r}\}=\delta_{I}(d,r),\mbox{ and}\\
\delta_\mathbb{X}(d,r)&\stackrel{\tiny(\ref{mar14-17})}{=}&\min\{|\mathbb{X}\setminus
V_\mathbb{X}(F)|\colon F\in
\mathcal{F}_{d,r}\}
\stackrel{\tiny(\ref{degree-formula-for-the-number-of-non-zeros})}{=}
\min\{\deg(S/(I\colon(F)))\vert\,
F\in \mathcal{F}_{d,r}\}=\vartheta_{I}(d,r).
\end{eqnarray*}
In these equalities we used the fact that
$\deg(S/I(\mathbb{X}))=|\mathbb{X}|$. As $H_{I}(d)=|\mathbb{X}|$ for
$d\geq{\rm reg}(S/I)$, using the generalized Singleton bound for the
generalized Hamming weights \cite[Theorem~7.10.6]{Huffman-Pless} and
the fact that the weight hierarchy 
is an increasing sequence  \cite[Theorem~1, Corollary~1]{wei}, we obtain that $\delta_\mathbb{X}(d,r)=r$ for all 
$d\geq {\rm reg}(S/I(\mathbb{X}))$.
\end{proof}

Renter\'\i a, Simis and the
second author introduced algebraic methods to study parameterized codes 
and showed that the vanishing ideal of an algebraic toric set parameterized by
Laurent monomials over a finite field is a lattice ideal
\cite{algcodes}, see also the paper of \c{S}ahin \cite{mesut-sahin} and references therein.  
Vanishing ideals of sets parameterized by rational functions were thoroughly
studied in \cite{TV} over any field $K$. A wide open problem in this
area is to find formulas for the minimum distance of parameterized
codes and formulas for the regularity of vanishing ideals over graphs
\cite{NVV}. The recent book of Toh\v{a}neanu \cite{stefan} gives an overview of commutative
algebra methods in coding theory since the 1990'. 

The \textit{Vasconcelos number} (v-{\em number} for short) of a graded ideal $I$ of $S$, denoted ${\rm
v}(I)$, is the following invariant of
$I$ that was introduced by Cooper, Seceleanu, Toh\v{a}neanu, 
Vaz Pinto and the second author to study the asymptotic behavior 
of the minimum distance of projective Reed--Muller-type codes
\cite[Definition~4.1]{min-dis-generalized}:
\begin{equation}\label{v-number-def}
{\rm v}(I):=\min\{d\geq 0 \mid\exists\, f 
\in S_d \mbox{ and }\mathfrak{p} \in {\rm Ass}(I) \mbox{ with } (I\colon f)
=\mathfrak{p}\}.
\end{equation} 
\quad Since then, the {\rm v-number} has been studied by several
authors for certain
classes of graded ideals (e.g., edge ideals, monomial ideals, ideals
of covers of graphs, binomial 
edge ideals, and Gorenstein ideals)      
\cite{civan,Ficarra,duality-projective,Delio-Lisa,
v-number,Kumar-Nanduri-Saha,weightedma,Saha,Saha-Gorenstein,saha-sengupta},
for homogeneous ideals of finitely generated $\mathbb{N}$-graded algebra domains over
Noetherian rings \cite{Conca}, and for graded modules \cite{Fiorindo-Ghosh}. The {\rm v}-number 
of graded ideals can be computed using \cite[Theorem 1]{im-vnumber}
(for the case of unmixed graded ideals see \cite[Proposition~4.2]{min-dis-generalized}).

The function ${\rm v}(I^k)$,
$k=1,2,\ldots$, is also called the ${\rm v}$-\textit{function} of $I$. 
The asymptotic behavior of the v-function is studied by Ficarra and
Sgroi in \cite{Ficarra-Sgroi} and Conca in \cite{Conca}, they independently proved that for a graded 
ideal $I$ of $S$, the v-function is an asymptotic linear function, i.e., ${\rm v}(I^k)$
 is a linear function of the form ${\rm v}(I^k)=\alpha(I)k+b$ for $k\gg 0$, where
 $\alpha(I)$ is the initial degree of $I$ and $b\in\mathbb{Z}$. The
 asymptotic behavior of the v-function for Noetherian graded filtrations 
is studied by Kumar, Nanduri and Saha \cite{Kumar-Nanduri-Saha}.

Let $\mathbb{P}^{s-1}$ be the projective space over a finite field
$K=\mathbb{F}_q$, let $\mathbb{X}$ be a set of points of
$\mathbb{P}^{s-1}$, and let $\delta_\mathbb{X}(d)$ be the minimum
distance of $C_\mathbb{X}(d)$. An upper bound for the 
v-number of $I(\mathbb{X})$ is the regularity of $S/I(\mathbb{X})$
\cite[Theorem 4.10]{min-dis-generalized}. 
The function $\delta_\mathbb{X}(d)$ is
strictly decreasing as a function of 
$d$ until it reaches the value $1$ \cite{min-dis-generalized}.
Potentially good codes, capable of correcting errors in the transmission of
information, should have minimum distance greater than $1$ \cite{MacWilliams-Sloane}. 

The next result shows the significance of the {\rm
v}-number for coding theory: 
 
\begin{theorem}\cite[Corollary 5.6.]{min-dis-generalized}
\label{min-dis-v-number} Let $\mathbb{X}$ be a set of points of
$\mathbb{P}^{s-1}$ and let $\delta_\mathbb{X}(d)$ be the minimum
distance of $C_\mathbb{X}(d)$. Then, $ \delta_\mathbb{X}(d)=1$ if and
only if $d\geq {\rm v}(I(\mathbb{X}))$, that is, the
{\rm v}-number of $I(\mathbb{X})$ is the smallest $d\geq 1$
such that $\delta_\mathbb{X}(d)=1$.
\end{theorem}

The {\rm v}-number of edge ideal of
graphs can be used to classify the family of $W_2$ graphs
\cite{v-number}. A graph $G$ belongs to class $W_2$ if and only if $G$ is well-covered,
$G\setminus v$ is well-covered for all $v\in
V(G)$ and $G$ and has no isolated vertices. The next result gives us
an algebraic method to determine if 
a given graph is in $W_2$ using \textit{Macaulay}$2$ \cite{mac2}.

\begin{theorem}\cite{v-number}\label{v-numbers-w2} Let $G$ be a graph
without isolated vertices and let 
$I=I(G)$ be its edge ideal. Then, $G$ is in $W_2$  if and only if
${\rm v}(I)=\dim(S/I)$.  
\end{theorem}

\begin{lemma}{\rm\cite[Lemma 5, Appendix 6]{zariski-samuel-ii}}
\label{ntr-complete-intersection}
 Let $I\subset S$ be an ideal generated by a regular sequence.
Then, $I^n$ is unmixed for all $n\geq 1$. In particular $I^{n}=I^{(n)}$ 
for all $n\geq 1$. 
\end{lemma}

\begin{theorem}{\rm\cite[(14) Corollary,
p.~38]{brodmann}}\label{oct14-17} 
Let $(S, \mathfrak{m})$ be a local Cohen--Macaulay ring, and let
$I\subset S$ be an ideal of height $h > 0$.   Assume that
$IS_\mathfrak{p}$ is generated by $h$ elements for each minimal prime
$\mathfrak{p}$ of $I$. Then, the following statements are equivalent:
\begin{enumerate}
\item[$(i)$] $I^{k-1}/I^k$ is a Cohen-Macaulay module over $S/I$ for
infinitely many $k$.
\item[$(ii)$] $S/I^k$ is a Cohen-Macaulay ring for infinitely many $k$.
\item[$(iii)$] $I$ is generated by $h$ elements $($hence a complete
intersection$)$.
\end{enumerate}
\end{theorem}

\begin{remark}\label{oct15-17} If $I\subset S$ is a complete intersection graded ideal
of a polynomial ring $S$, then $S/I^{k}$ is Cohen--Macaulay for 
$k\geq 1$ (see \cite[Lemma~2.7]{HU} and \cite[17.4, p.~139]{Mats} for more
general statements).
\end{remark}

\begin{proposition}\cite[Proposition~5.27]{Vivares-thesis}\label{vanishing-ci-ntf} 
Let $\mathbb{X}$ be a
finite set of points in $\mathbb{P}^{s-1}$ over a field $K$. Then
$I(\mathbb{X})^k=I(\mathbb{X})^{(k)}$ for all $k\geq 1$ if and 
only if $I(\mathbb{X})$ is a complete intersection. 
\end{proposition}

\begin{proof} $\Rightarrow$) Assume that
$I(\mathbb{X})^k=I(\mathbb{X})^{(k)}$ for all $k\geq 1$. We proceed
by contradiction assuming that $I(\mathbb{X})$ is not a complete
intersection. Then, by Theorem~\ref{oct14-17}, $I(\mathbb{X})^k$ is
not Cohen--Macaulay for some $k$. Hence the depth of
$S/I(\mathbb{X})^k$ is $0$ because $I(\mathbb{X})^k$ has dimension
$1$. Thus $\mathfrak{m}$ is an associated
prime of $S/I(\mathbb{X})^k$, a contradiction because 
${\rm Ass}(I(\mathbb{X})^k)={\rm Ass}(I(\mathbb{X})^{(k)})={\rm
Ass}(S/I(\mathbb{X}))$ and $I(\mathbb{X})$
is a radical Cohen--Macaulay ideal of height $s-1$.

$\Leftarrow)$ This is a special case of a classical result (see
Lemma~\ref{ntr-complete-intersection} and Remark~\ref{oct15-17}). 
\end{proof}

If $I(\mathbb{X})$ is a complete
intersection, then the Rees algebra $S[I(\mathbb{X})z]$ of
$I(\mathbb{X})$ is normal.
Indeed, this follows from a result of Vasconcelos et. al.
\cite[Corollary~5.3]{ITG} 
after noticing that $I(\mathbb{X})$ is a radical ideal 
which is generically a complete intersection.

\section*{Acknowledgments} We thank Aldo Conca for asking if the
letter v in the definition of the v-number 
has some meaning related to Wolmer Vasconcelos, see the last paragraph of
Section~\ref{intro-section}.  We also thank Adam Van Tuyl for pointing out 
that \cite[Question 6.12]{almousa-dochtermann-smith} was shown in the
affirmative by Akiyoshi Tsuchiya, and that his proof appeared in
\cite{almousa-dochtermann-smith} right after Question 6.12. 
Computations with \textit{Normaliz} \cite{normaliz2} and 
\textit{Macaulay}$2$ \cite{mac2} were important to gain a better understanding
of the regularity and normality of a monomial ideal.


\bibliographystyle{plain}

\end{document}